\documentclass[article,onefignum,onetabnum]{siamart171218}

\usepackage{amsmath,amssymb, graphicx,latexsym}

\usepackage{array}
\usepackage{subcaption,caption}
\usepackage{url}
\usepackage{color}
\usepackage{epstopdf}

\usepackage{algorithm}
\usepackage{algorithmic}
\usepackage{mathtools}
\usepackage{multirow}
\usepackage{float}
\usepackage{geometry}
\usepackage{mathptmx}      
\usepackage{latexsym}

\usepackage{lipsum}
\usepackage{amsfonts}
\usepackage{graphicx}
\usepackage{epstopdf}
\usepackage{algorithmic}
\ifpdf
  \DeclareGraphicsExtensions{.eps,.pdf,.png,.jpg}
\else
  \DeclareGraphicsExtensions{.eps}
\fi

\newcommand{\be}[1]{\begin{equation}\label{#1}}
\newcommand{\ee}{\end{equation}}

\renewcommand{\i}{\mathrm{i}}

\newcommand\treq{\stackrel{\mathclap{\normalfont\mbox{tr}}}{=}}

\newsiamremark{remark}{Remark}
\newsiamremark{example}{Example}
\newsiamremark{hypothesis}{Hypothesis}
\crefname{hypothesis}{Hypothesis}{Hypotheses}
\newsiamthm{claim}{Claim}

\headers{An inverse spectral problem for a damped wave operator}{G. Bao, X. Xu and J. Zhai}

\title{An inverse spectral problem for a damped wave operator}

\author{Gang Bao\thanks{School of Mathematical Sciences, Zhejiang University, Hangzhou, Zhejiang, China.
{G. Bao's research was supported in part by NSFC 11621101. (\email{baog@zju.edu.cn}). }
}
\and Xiang Xu\thanks{School of Mathematical Sciences, Zhejiang University, Hangzhou, Zhejiang, China. X. Xu's research was supported in part by the Fundamental Research Funds for the Central Universities. (\email{xxu@zju.edu.cn}).}
\and Jian Zhai \thanks{Institute for Advanced Study, The Hong Kong University of Science and Technology, Hong Kong, China.
  (\email{iasjzhai@ust.hk}).}
   }

\usepackage{amsopn}

\makeatletter
\newcommand*{\addFileDependency}[1]{
  \typeout{(#1)}
  \@addtofilelist{#1}
  \IfFileExists{#1}{}{\typeout{No file #1.}}
}
\makeatother

\begin{document}
\maketitle

\begin{abstract}
 This paper proposes a new and efficient numerical algorithm for recovering the damping coefficient from the spectrum of a damped wave operator, which is a classical {\it Borg-Levinson inverse spectral problem}. The algorithm is based on inverting a sequence of {\em trace formulas}, which are deduced by a recursive formula, bridging geometrical and spectrum information explicitly in terms of Fredholm integral equations. Numerical examples are presented to illustrate the efficiency of the proposed algorithm.
\end{abstract}



\section{Introduction}
A damped wave equation describes a wave whose amplitude of oscillation decreases with time. It has far-ranging applications in many directions such as electromagnetic waves, acoustic waves and elastic waves. For instance, it was the first practical model to describe the radio transmission by spark gap transmitters during the wireless telegraphy era, which is now generally referred to as ``Class B" emission. In \cite{Bamberger1982}, the authors studied the harmonics on stringed instruments and the damping coefficient was considered as the frictional resistance of the string, which may be caused by external forces. Moreover, similar mathematical models with damping term are proposed for linear elastic systems in \cite{Chen1982,Falun1990}, where the damping coefficient was considered as viscosity. More applications can be found in the survey \cite{gesztesy2011damped} and the references cited therein. Consider the one-dimensional damped wave equation with unit wave speed and viscous damping $\alpha(x)$:
\begin{equation}\label{eq1}
\begin{split}
&u_{tt}(x,t)-u_{xx}(x,t)+\alpha(x)u_t(x,t)=0,\quad (x,t)\in (0,1)\times[0,\infty),\\
&u(\cdot,0)=f_0,\quad u_t(\cdot,0)=f_1,\quad t>0,\\
&u(\cdot, t)\,\text{ satisfies certain boundary conditions at }x=0\text{ and }x=1\text{ for }t\in\mathbb{R}^{+}.
\end{split}
\end{equation}
Assume that $\alpha(x)\in L^\infty(0,1)$  is real-valued and $0\leq 2a\leq \alpha(x)\leq 2b<+\infty$. We can rewrite \eqref{eq1} in a vector form:
\begin{equation}\label{eqV}
V_t=A(\alpha)V,
\end{equation}
where $V=(u,u_t)$ and
\begin{equation}\label{damp_OP1}
A(\alpha)=\left(
\begin{array}{cc}
0 & I\\
\left(\frac{\mathrm{d}^2}{\mathrm{d}x^2}\right)_{\mathrm{bc}}&-\alpha(x)
\end{array}
\right).
\end{equation}
Here the subscript ``$\mathrm{bc}$" represents appropriate boundary conditions at $x=0,1$, eg. Dirichlet, etc., to be described in details in Section \ref{sec2}. The initial condition $(f_0,f_1)$ needs to be consistent with the boundary condition.
 The well-posedness of the initial boundary value problem for \eqref{eqV} with initial value $V(0)=(u(\cdot,0),\, u_t(\cdot,0))=(f_0,f_1)\in  L^2(0,1)^2$ can be obtained by the standard semigroup approach \cite{Engel2000,Fattorini1985}. Moreover, it is known that if $\alpha(x)\in L^{\infty}(0,1)$ then $A(\alpha)$ has a compact inverse and hence a discrete spectrum, consisting of countably many eigenvalues, denoted by $\sigma_p(A(\alpha))=\{\lambda_j(A(\alpha))\}_{j\in J}$.
\par
The present work is devoted to the inverse problem of recovering the damped coefficient $\alpha(x)$ from the spectrum $\sigma_p(A(\alpha))$. This is a classical inverse spectral problem in mathematical physics and relates to a variety of vibration absorption problems in the engineering literature, see \cite{mottershead2006inverse}. Mathematically, it can be viewed as a classical {\it Borg-Levinson inverse spectral problem}. The uniqueness on determination of $\alpha(x)$ from the Dirichlet eigenvalues was established for $\alpha$ even, with respect to $x=1/2$, see \cite{Cox2011reconstructing}. In \cite{Pivovarchik1999}, for weakly damped
strings, i.e., with no purely imaginary eigenvalues, the determination of the potential and the boundary conditions were considered by the given spectrum and length of the string. In \cite{BORISOV2009}, Borisov et al showed the criterion for the damping term to be constant and expect this inverse problem to be more rigid than Sturm-Liouville problem since there is no other smooth damping term yielding the same spectrum as constant damping. For numerical reconstruction of the damping coefficient, to the authors' best knowledge, the only available approach was introduced by Cox and Embree \cite{Cox2011reconstructing}, which was based on a refined asymptotic formula for the large eigenvalues. However, it is known that for inverse Sturm-Liouville operators, there are many works on numerical algorithms, see \cite{lowe1992recovery, rundell1992reconstruction, Sacks2015} and the references therein for an overview on numerical progress. Moreover, Xu and Zhai \cite{XZ1} have developed a numerical scheme for recovering a density in the Sturm-Liouville operator based on a sequence of {\it trace formulas} which give an explicit relation between the eigenvalues and the unknown coefficient recently.


  In this paper, we propose a novel numerical scheme for recovering the damping coefficient $\alpha(x)$ from the spectrum $\sigma_p(A(\alpha))$ in a similar framework as \cite{XZ1}. The scheme is based on the explicit formulas which will be derived in the next section for the following maps
\begin{equation}\label{map1}
\alpha\rightarrow \sum_{j\in J}\lambda_j(A(\alpha))^{-s}, \mbox{ for } s = 1,\cdots, \infty.
\end{equation}
where $\sum_{j\in J}\lambda_j(A(\alpha))^{-s}$ are traces of $(A(\alpha))^{-s}$. It has been shown in \cite{XZ1} that inverting the above maps are severely ill-posed when $A$ is a Laplacian operator with Dirichelt bounary conditions. According to the property of trace class operators of $(A(\alpha))^{-s}$ ($s=2,\cdots,\infty$), it makes sense to reduce the numerical instability by inverting the following maps
\begin{equation*}
\alpha\rightarrow \Big\{\sum_{j\in J}T_n(\lambda_j(A(\alpha))^{-1})\Big\}_{n=1}^\infty,
\end{equation*}
with a collection of carefully chosen polynomials $\{T_n(z)\}_{n=1}^\infty$, $z\in\mathbb{C}$.
It should be noted that due to the inherent difficulties for damped wave operator, the two ingredients of the numerical algorithm in \cite{XZ1}, i.e., trace formulas and stabilizing polynomials are completely different. Due to the model difference, the trace formulas are derived based on the resolvent of $A(\alpha)$, instead of the Green function for the Sturm-Liouville operator. Moreover, since the eigenvalue distribution is no longer in the real axis as in the previous case, the choice of stabilizing polynomial which depends on the spectrum distribution becomes more complicated.


The rest of the paper is organized as follows. Section \ref{sec2} is devoted to establishing the desired trace formulas. By analyzing the resolvent of $A(\alpha)$, we arrive at some explicit recursive formulas. In Section \ref{sec3}, we show the injectivity of the Fr\'echet derivative of the map \eqref{map1} at a constant damping. In Section \ref{sec4}, we present the algorithm with implementation details. In Section \ref{sec5}, we conduct several numerical experiments to illustrate the efficiency of our algorithm. Impacts of different parameters are also discussed in this section.

\section{Trace formulas}\label{sec2}
In this section, we derive a sequence of trace formulas useful for inverting $\alpha(x)$.
Let $T$ be an unbounded operator on $L^2(0,1)$ such that $Tf=\i\frac{\mathrm{d}}{\mathrm{d}x}f$ with appropriate boundary conditions at $x=0,1$. The operator $T$ needs to be densely defined and closed. Here, we list some examples of $T$, namely, $T_{\min},\,T_0,\,T_1,\,T_\omega$, which are carefully defined and characterized in \cite{gesztesy2011damped}. For the convenience of readers,  we summarize some results here. The domains of these operators are:
\begin{equation*}
\begin{split}
\mathrm{dom}(T_{\min})=&\{f\in L^2(0,1)\vert f\in AC([0,1]);\, f(0)=f(1)=0\},\\
\mathrm{dom}(T_0)=&\{f\in L^2(0,1)\vert f\in AC([0,1]); f(0)=0; f'\in L^2(0,1)\},\\
\mathrm{dom}(T_1)=&\{f\in L^2(0,1)\vert f\in AC([0,1]); f(1)=0; f'\in L^2(0,1)\},\\
\mathrm{dom}(T_\omega)=&\{f\in L^2(0,1)\vert f\in AC([0,1]); f(1)=\omega f(0); f'\in L^2(0,1)\}.
\end{split}
\end{equation*}
Here $AC([0,1])$ denotes the space of absolutely continuous functions on $[0,1]$.
For $\omega\in\mathbb{R}\setminus\{0,1\}$, we have
\begin{equation*}
\ker (T_{\min})=\ker (T_0)=\ker (T_1)=\ker (T_\omega)=\{0\}.
\end{equation*}
Then $T^*Tf =-f''$ for any $T=T_{\min},\,T_0,\,T_1,\,T_\omega$ with
\begin{equation*}
\begin{split}
\mathrm{dom}({T_{\min}}^*T_{\min})=&\{f\in L^2(0,1)\vert f,f'\in AC([0,1]), f(0)=f(1)=0, f''\in L^2\},\\
\mathrm{dom}({T_0}^*T_0)=&\{f\in L^2(0,1)\vert f,f'\in AC([0,1]), f(0)=f'(1)=0, f''\in L^2\},\\
\mathrm{dom}({T_1}^*T_1)=&\{f\in L^2(0,1)\vert f,f'\in AC([0,1]), f'(0)=f(1)=0, f''\in L^2\},\\
\mathrm{dom}({T_\omega}^*T_\omega)=&\{f\in L^2(0,1)\vert f,f'\in AC([0,1]), f(1)=\omega f(0), \omega f'(1)= f'(0); f''\in L^2\}.
\end{split}
\end{equation*}
By the fact
\begin{equation*}
\ker(T^*T)=\ker(T),
\end{equation*}
we have the invertibility of $T^*T$ for $T=T_{\min},\,T_0,\,T_1,\,T_\omega$ with $\omega\in\mathbb{R}\setminus\{0,1\}$.
\begin{remark}
Notice that ${T_{\min}}^*T_{\min}=-\Delta_D=-\left(\frac{\mathrm{d}^2}{\mathrm{d}x^2}\right)_D$ is the Dirichlet Laplacian.
\end{remark}
\par We take $T$ to be any of the above defined operators. Define
\begin{equation}\label{damp_OP}
A(\alpha)=\left(
\begin{array}{cc}
0 & I\\
-T^*T&-\alpha(x)
\end{array}
\right)
\end{equation}
on the space $L^2([0,1])^2$.
Since $T^*T$ is coercive, then $0\in\rho(A(\alpha))$ \cite[Theorem 2.3]{gesztesy2011damped}. It is easy to see that if $\lambda$ is an eigenvalue of $A(\alpha)$ with eigenvector $u=[y,z]$, then $z=\lambda y$ and
\begin{equation}\label{eq_eigenvector}
y''-\lambda\alpha y-\lambda^2 y=0,
\end{equation}
with $y$ satisfying suitable boundary conditions. It is clear that $\overline{\lambda}$ is also an eigenvalue of $A(\alpha)$ with eigenvector $\overline{u}=[\overline{y},\overline{z}]=[\overline{y},\overline{\lambda} \overline{y}]$. Moreover, by \cite[Lemma 2.5]{gesztesy2011damped}, the two eigenvalues $\lambda$ and $\overline{\lambda}$ have the same geometric and algebraic multiplicities.
Actually, the spectrum $\sigma_p(A(\alpha))$ consists of two infinite sequences $\{\lambda_{\pm j}(A(\alpha))\}_{j=1}^\infty$, where $\mathrm{Im}\,\lambda_{-j}=-\mathrm{Im}\,\lambda_j$. We denote $J=\mathbb{Z}\setminus\{0\}$, and
\begin{equation*}
\sigma_p(A(\alpha))=\{\lambda_j\}_{j\in J}=\{\lambda_j(A(\alpha))\}_{j\in J}.
\end{equation*}
The eigenvalues are ordered as follows
\begin{equation*}
\cdots\leq \mathrm{Im}\,\lambda_{-2}\leq \mathrm{Im}\,\lambda_{-1}\leq \mathrm{Im}\,\lambda_{1}\leq \mathrm{Im}\,\lambda_{2}\leq\cdots
\end{equation*}
counting algebraic multiplicities. If the spectrum does not contain real eigenvalues, this labeling of eigenvalues is clear and $\lambda_{-j}=\overline{\lambda_j}$ for any $j$. If real eigenvalues exist, one can invoke \cite[Lemma 4.1]{BORISOV2009} and \cite[Theorem 5.3]{CoXZua1994}. Next we give a sufficient condition for the nonexistence of real eigenvalues.

\begin{lemma}
If $b<\sqrt{\mu_1(T^*T)}$, where $\mu_1(T^*T)$ is the smallest eigenvalue of $T^*T$, then $\sigma_p(A(\alpha))\cap\mathbb{R}=\emptyset$.
\end{lemma}

\begin{proof}
Integrating \eqref{eq_eigenvector} against $\overline{y}$, we obtain
\begin{equation*}
\int_0^1|y'|^2\mathrm{d}x+\lambda\int_0^1\alpha|y|^2\mathrm{d}x+\lambda^2\int_0^1|y|^2\mathrm{d}x=0
\end{equation*}
for $y\in \mathrm{dom}(T^*T)$, $T=T_{\min},\,T_0,\,T_1,\,T_\omega$ with $\omega\in\mathbb{R}\setminus\{0,1\}$. Then we find
\begin{equation*}
\lambda=\frac{-\int_0^1\alpha|y|^2\mathrm{d}x\pm \left((\int_0^1\alpha|y|^2\mathrm{d}x)^2-4\int_0^1|y'|^2\mathrm{d}x\int_0^1|y|^2\mathrm{d}x\right)^{1/2}}{2\int_0^1|y|^2\mathrm{d}x}.
\end{equation*}
Using $\alpha\leq 2b$, we have
\begin{equation*}
\begin{split}
& \left(\int_0^1\alpha|y|^2\mathrm{d}x\right)^2-4\int_0^1|y'|^2\mathrm{d}x\int_0^1|y|^2\mathrm{d}x\\
\leq &4b^2\left(\int_0^1|y|^2\mathrm{d}x\right)^2-4\int_0^1|y'|^2\mathrm{d}x\int_0^1|y|^2\mathrm{d}x\\
\leq&4\left(\int_0^1|y|^2\mathrm{d}x\right)^2\left(b^2-\frac{\int_0^1|y'|^2\mathrm{d}x}{\int_0^1|y|^2\mathrm{d}x}\right).
\end{split}
\end{equation*}
Notice that the smallest eigenvalue of $T^*T$ is $\mu_1(T^*T)=\inf_{y\in \mathrm{dom}(T^*T)}\frac{\int_0^1|y'|^2\mathrm{d}x}{\int_0^1|y|^2\mathrm{d}x}$. Therefore if $b<\sqrt{\mu_1(T^*T)}$, $\lambda$ is not real-valued.  
\hfill \end{proof}

Note that $\mu_1({T_{min}}^*T_{min})=\pi^2$, $\mu_1({T_0}^*T_0)=\mu_1({T_1}^*T_1)=\frac{1}{4}\pi^2$.

\par Now we proceed to deriving the trace formulas for $(A(\alpha))^{-n-1}$, $n=0,1,2\cdots$. All the trace formulas can be generated by a recursive relation, which is used for the inversion algorithm. The trace formulas for $n=2k$ are obtained in \cite{gesztesy2011damped}, but in a less explicit form. Denote
\begin{align*}
R(\zeta)&=-(2\zeta+\alpha)Q(\zeta), \\
Q(\zeta)&=(T^*T+\zeta^2+\zeta\alpha)^{-1}.
\end{align*}
Note that $Q(\zeta)$, $R(\zeta)$ are of trace class in the separable Hilbert space $L^2(0,1)$. Some useful properties of trace-class operators are summarized in \cite{XZ1}. We denote $A\treq B$ if the operators $A$ and $B$ have the same trace.
\par
By simple calculations, we have the following explicit expression for the resolvent of $A(\alpha)$ (cf. \cite{CoXZua1994})
\begin{equation*}
(A(\alpha)-\zeta)^{-1}=\left(
\begin{array}{cc}
-Q(\zeta)(\zeta+\alpha) & -Q(\zeta)\\
I-\zeta Q(\zeta)(\zeta+\alpha) &-\zeta Q(\zeta)
\end{array}
\right),
\end{equation*}
for $\zeta\in\mathbb{R}\setminus\{0\}$ with $|\zeta|$ sufficiently small. Notice that the operator $(A(\alpha)-\zeta)^{-1}$ is not of trace class (The only ``bad" term is the identity operator in the lower left entry).
However it is clear that
\begin{equation*}
\frac{\partial}{\partial\zeta}(A(\alpha)-\zeta)^{-1}=\frac{\partial}{\partial\zeta}\left(
\begin{array}{cc}
-Q(\zeta)(\zeta+\alpha) & -Q(\zeta)\\
-\zeta Q(\zeta)(\zeta+\alpha) &-\zeta Q(\zeta)
\end{array}
\right)
\end{equation*}
is of trace class. Moreover, we have
\begin{equation*}
\frac{\partial}{\partial\zeta}(A(\alpha)-\zeta)^{-1}\treq \frac{\partial}{\partial\zeta}[-Q(\zeta)(\zeta+\alpha)-\zeta Q(\zeta)]\treq \frac{\partial}{\partial\zeta}R(\zeta).
\end{equation*}
We note here that although $(A(\alpha)-\zeta)^{-1}$ is not of trace class, the operator $R(\zeta)$ is.\\

Next we derive a sequence of trace formulas associated with $R(\zeta)$. First notice
\begin{lemma}
\begin{equation*}
Q'(\zeta)=-Q(\zeta)(2\zeta+\alpha)Q(\zeta).
\end{equation*}
\end{lemma}
\begin{proof} To prove this, we only need to directly calculate
\begin{equation*}
\begin{split}
Q'(\zeta)=&\lim_{h\rightarrow 0}\frac{Q(\zeta+h)-Q(\zeta)}{h}\\
=&\lim_{h\rightarrow 0}\frac{\left(T^*T+(\zeta+h)^2+(\zeta+h)\alpha\right)^{-1}-\left(T^*T-\zeta^2-\zeta\alpha\right)^{-1}}{h}\\
=&\lim_{h\rightarrow 0}\frac{1}{h}\left(T^*T+(\zeta+h)^2+(\zeta+h)\alpha\right)^{-1}\left(T^*T+\zeta^2+\zeta\alpha-T^*T-(\zeta+h)^2-(\zeta+h)\alpha\right)\\
&\quad\quad\quad\quad\quad\quad\quad\quad\quad\left(T^*T+\zeta^2+\zeta\alpha\right)^{-1}\\
=&\lim_{h\rightarrow 0}\frac{1}{h}\left(T^*T+(\zeta+h)^2+(\zeta+h)\alpha\right)^{-1}\left(-2h\zeta+h^2-h\alpha\right)\left(T^*T+\zeta^2+\zeta\alpha\right)^{-1}\\
=&-\left(T^*T+\zeta^2+\zeta\alpha\right)^{-1}(2\zeta+\alpha)\left(T^*T+\zeta^2+\zeta\alpha\right)^{-1}\\
=&-Q(\zeta)(2\zeta+\alpha)Q(\zeta).
\end{split}
\end{equation*}

\hfill\end{proof}

To derive trace formulas, we start with
\begin{equation*}
R(\zeta)=-(2\zeta+\alpha)Q(\zeta).
\end{equation*}
By the chain rule, we can calculate the derivatives of $R(\zeta)$ with respect to $\zeta$ as follows
\begin{equation*}
R'(\zeta)= -2Q(\zeta)+(2\zeta+\alpha)Q(\zeta)(2\zeta+\alpha)Q(\zeta),
\end{equation*}
and
\begin{equation*}
\begin{split}
R''(\zeta)=&2Q(\zeta)(2\zeta+\alpha)Q(\zeta)+2Q(\zeta)(2\zeta+\alpha)Q(\zeta)+2(2\zeta+\alpha)Q(\zeta)Q(\zeta)\\
&-2(2\zeta+\alpha)Q(\zeta)(2\zeta+\alpha)Q(\zeta)(2\zeta+\alpha)Q(\zeta)\\
=&2(2\zeta+\alpha)Q(\zeta)Q(\zeta)+2\big(2Q(\zeta)-(2\zeta+\alpha)Q(\zeta)(2\zeta+\alpha)Q(\zeta)\big)(2\zeta+\alpha)Q(\zeta)\\
=&-2R(\zeta)Q(\zeta)-2R'(\zeta)(2\zeta+\alpha)Q(\zeta).
\end{split}
\end{equation*}
We can continue and obtain
\begin{equation*}
\begin{split}
R'''(\zeta)=&-2R'(\zeta)Q(\zeta)+2R(\zeta)Q(\zeta)(2\zeta+\alpha)Q(\zeta)-2R''(\zeta)(2\zeta+\alpha)Q(\zeta)\\
&-4R'(\zeta)Q(\zeta)+2R'(\zeta)(2\zeta+\alpha)Q(\zeta)(2\zeta+\alpha)Q(\zeta)\\
=&-6R'(\zeta)Q(\zeta)-R''(\zeta)(2\zeta+\alpha)Q(\zeta)-2R''(\zeta)(2\zeta+\alpha)Q(\zeta)\\
=&-6R'(\zeta)Q(\zeta)-3R''(\zeta)(2\zeta+\alpha)Q(\zeta).
\end{split}
\end{equation*}
We observe that
\begin{equation*}
\frac{1}{2!}R''(\zeta)= -R(\zeta)Q(\zeta)-R'(\zeta)(2\zeta+\alpha)Q(\zeta),
\end{equation*}
\begin{equation*}
\frac{1}{3!}R'''(\zeta)= -R'(\zeta)Q(\zeta)-\frac{1}{2!}R''(\zeta)(2\zeta+\alpha)Q(\zeta).
\end{equation*}
Generally, we have the following recursive relation:
\begin{lemma}
Assume $\zeta\in\mathbb{R}\setminus \{0\}$ with $|\zeta|$ sufficiently small, such that $\zeta\in \rho(A(\alpha))$, then
\begin{equation}\label{recursiveR}
\frac{1}{n!}R^{(n)}(\zeta)= -\frac{1}{(n-2)!}R^{(n-2)}(\zeta)Q(\zeta)-\frac{1}{(n-1)!}R^{(n-1)}(\zeta)(2\zeta+\alpha)Q(\zeta).
\end{equation}
\end{lemma}

\begin{proof}
We prove it by induction. We have already seen that \eqref{recursiveR} holds true for the case $n=2$. Assume that for $n$, the above \eqref{recursiveR} holds true. Then we proceed
\begin{equation*}
\begin{split}
&\frac{1}{(n+1)!}R^{(n+1)}(\zeta)\\
=&\frac{1}{n+1}\frac{1}{n!}\frac{\mathrm{d}}{\mathrm{d}\zeta}R^{(n)}(\zeta)\\
=&\frac{1}{n+1}\frac{1}{(n-2)!}R^{(n-2)}(\zeta)Q(\zeta)(2\zeta+\alpha)Q(\zeta)-\frac{1}{n+1}\frac{1}{(n-2)!}R^{(n-1)}(\zeta)Q(\zeta)\\
&-\frac{1}{n+1}\frac{1}{(n-1)!}R^{(n)}(\zeta)(2\zeta+\alpha)Q(\zeta)-\frac{2}{n+1}\frac{1}{(n-1)!}R^{(n-1)}(\zeta)Q(\zeta)\\
&+\frac{1}{n+1}\frac{1}{(n-1)!}R^{(n-1)}(\zeta)(2\zeta+\alpha)Q(\zeta)(2\zeta+\alpha)Q(\zeta)\\
=&\frac{1}{n+1}\left(\frac{1}{(n-2)!}R^{(n-2)}(\zeta)Q(\zeta)+\frac{1}{(n-1)!}R^{(n-1)}(\zeta)(2\zeta+\alpha)Q(\zeta)\right)(2\zeta+\alpha)Q(\zeta)\\
&-\frac{1}{n+1}\frac{1}{(n-1)!}R^{(n)}(\zeta)(2\zeta+\alpha)Q(\zeta)-\frac{1}{n+1}\frac{1}{(n-2)!}R^{(n-1)}(\zeta)Q(\zeta)\\
&-\frac{2}{n+1}\frac{1}{(n-1)!}R^{(n-1)}(\zeta)Q(\zeta)\\
=&-\frac{1}{(n+1)!}R^{(n)}(\zeta)(2\zeta+\alpha)Q(\zeta)-\frac{1}{n+1}\frac{1}{(n-1)!}R^{(n)}(\zeta)(2\zeta+\alpha)Q(\zeta)\\
&-\left(\frac{1}{n+1}+\frac{2}{(n+1)(n-1)}\right)\frac{1}{(n-2)!}R^{(n-1)}(\zeta)Q(\zeta)\\
=&-\frac{1}{n!}R^{(n)}(\zeta)(2\zeta+\alpha)Q(\zeta)-\frac{1}{(n-1)!}R^{(n-1)}(\zeta)Q(\zeta).
\end{split}
\end{equation*}
The lemma is proved. \hfill \end{proof}

We use the notation $R_n(\alpha)=\frac{1}{n!}R^{(n)}(0)$.
Evaluating the recursive relation \eqref{recursiveR} at $\zeta=0$ gives the following proposition.
\begin{proposition}\label{Rn}
The follow recursive formula holds:
\begin{align}\label{recursive_Rn}R_n(\alpha)= -R_{n-2}(\alpha)(T^*T)^{-1}-  R_{n-1}(\alpha)\alpha(T^*T)^{-1}
\end{align}
for $n=2,3,\cdots$, with
\begin{align*}
R_0(\alpha)&= R(0)=-\alpha Q(0)=-\alpha (T^*T)^{-1},\\
R_1(\alpha)&= R'(0)=-2Q(0)+\alpha Q(0)\alpha Q(0)=-2(T^*T)^{-1}+\alpha (T^*T)^{-1}\alpha (T^*T)^{-1}.\\
\end{align*}
\end{proposition}
The following lemma is similar to \cite[Theorem 5.11]{gesztesy2011damped}.
\begin{lemma} \label{traceformula}
Denote $\lambda_j=\lambda_j(A(\alpha))$.
We have that for any $n=1,2,\cdots$,
\begin{align} \label{traceimagpart}
\mathrm{Im}\sum_{j\in J}\lambda_j(A(\alpha))^{-n-1}=0
\end{align}
 and
\begin{align}\label{tracerealpart}
\mathrm{trace}(R_{n}(\alpha))=\sum_{j\in J}\lambda_j(A(\alpha))^{-n-1}=\mathrm{Re}\sum_{j\in J}\lambda_j(A(\alpha))^{-n-1}.
\end{align}
\end{lemma}
\begin{proof}
Assume $\zeta\in (-\varepsilon_0,\varepsilon_0)\setminus \{0\}$, $\varepsilon_0>0$ sufficiently small such that $ (-\varepsilon_0,\varepsilon_0)\subset \rho(A(\alpha))$. Then, we have
\begin{equation*}
\begin{split}
\mathrm{trace}\left(\frac{\partial}{\partial\zeta}[(A(\alpha)-\zeta)^{-1}]\right)=&\sum_{j\in J}\frac{\partial}{\partial\zeta}[(\lambda_j(A(\alpha))-\zeta)^{-1}]\\
=&\sum_{j\in J}\frac{\partial}{\partial\zeta}[\lambda_j(A(\alpha))^{-1}\left(1-\zeta\lambda_j(A(\alpha))^{-1}\right)^{-1}]\\
=&\sum_{j\in J}\frac{\partial}{\partial\zeta}[\sum_{n=0}^\infty\lambda_j(A(\alpha))^{-1}\left(\zeta\lambda_j(A(\alpha))^{-1}\right)^{n}]\\
=&\sum_{n=1}^\infty\left(\sum_{j\in J}\lambda_j(A(\alpha))^{-n-1}\right)\zeta^{n-1}.
\end{split}
\end{equation*}
Notices that
\begin{equation*}
\mathrm{trace}\left(\frac{\partial}{\partial\zeta}[(A(\alpha)-\zeta)^{-1}]\right)=\mathrm{trace}(R'(\zeta))=\sum_{n=1}^\infty \mathrm{trace}(R_n(\alpha))\zeta^{n-1}.
\end{equation*}
thus the lemma is proved. \hfill \end{proof}

\begin{remark}\label{cauchy_sum}
Note that
$
\sum_{j\in J}\lambda_j(A(\alpha))^{-1}
$
is not summable. However, it is proved in \cite[Theorem 5.11]{gesztesy2011damped} that
\begin{equation*}
\sum_{j\in J}\mathrm{Re}\,\lambda_j(A(\alpha))^{-1}=\sum_{j\in J}\mathrm{Re}\frac{\overline{\lambda_j}}{|\lambda_j|^2}=\mathrm{trace}(R_0(\alpha)),
\end{equation*}
where the sum is convergent.
Also, it is clear that $(\sum_{j=-N}^{-1}+\sum_{j=1}^N)\mathrm{Im}\,\lambda_j^{-1}=0$ for any $N$, since $\mathrm{Im}\,\lambda_{-j}=-\mathrm{Im}\,\lambda_j$ Therefore, the identities \eqref{traceimagpart} and \eqref{tracerealpart} are also valid for $n=0$ when using the regularized sum $\lim_{N\rightarrow+\infty}(\sum_{j=-N}^{-1}+\sum_{j=1}^N)$.
\end{remark}

One can use Proposition \ref{Rn} and Lemma \ref{traceformula} to derive an infinite sequence of trace formulas. Let us write down a few ones.
\begin{equation*}
\begin{split}
\sum_{j\in J}\lambda_j(A(\alpha))^{-1}=&\mathrm{trace}(R_0(\alpha))=\mathrm{trace}(-\alpha (T^*T)^{-1}),\\
\sum_{j\in J}\lambda_j(A(\alpha))^{-2}=&\mathrm{trace}(R_1(\alpha))=\mathrm{trace}(-2(T^*T)^{-1}+\alpha (T^*T)^{-1}\alpha (T^*T)^{-1}),\\
\sum_{j\in J}\lambda_j(A(\alpha))^{-3}=&\mathrm{trace}(R_2(\alpha))\\
=&\mathrm{trace}(\alpha(T^*T)^{-2}+2(T^*T)^{-1}\alpha(T^*T)^{-1}-\alpha(T^*T)^{-1}\alpha(T^*T)^{-1}\alpha(T^*T)^{-1})\\
=&\mathrm{trace}(3(T^*T)^{-1}\alpha(T^*T)^{-1}-\alpha(T^*T)^{-1}\alpha(T^*T)^{-1}\alpha(T^*T)^{-1}),\\
\sum_{j\in J}\lambda_j(A(\alpha))^{-4}=&\mathrm{trace}(R_3(\alpha))\\
=&\mathrm{trace}\Big(2(T^*T)^{-2}-\alpha(T^*T)^{-1}\alpha(T^*T)^{-2}-\alpha(T^*T)^{-2}\alpha(T^*T)^{-1}\\
&\quad\quad\quad -2(T^*T)^{-1}\alpha(T^*T)^{-1}\alpha(T^*T)^{-1}\\
&\quad\quad\quad+\alpha(T^*T)^{-1}\alpha(T^*T)^{-1}\alpha(T^*T)^{-1}\alpha(T^*T)^{-1}\Big)\\
=&\mathrm{trace}\left(2(T^*T)^{-2}-4(T^*T)^{-1}\alpha(T^*T)^{-1}\alpha (T^*T)^{-1}\right.\\
&\quad\quad\quad+\left.\alpha(T^*T)^{-1}\alpha(T^*T)^{-1}\alpha(T^*T)^{-1}\alpha(T^*T)^{-1}\right),\\
etc.\quad\quad\quad\quad\quad&
\end{split}
\end{equation*}
We see that the above trace formulas establish a very clear relation between the damping coefficient $\alpha$ and the spectrum of $A(\alpha)$.

We propose an inversion scheme for the map
\begin{equation}\label{trace_map}
\mathcal{F}:\alpha\rightarrow \{\mathbf{t}_n(\alpha)\}_{n=0}^\infty:=\{\mathrm{trace}(R_n(\alpha))\}_{n=0}^\infty=\{\sum_{j\in J}\lambda_j(A(\alpha))^{-n-1}\}_{n=0}^\infty
\end{equation}
for the recovery of $\alpha(x)$.

\section{Injectivity of a linearized map}\label{sec3}
The unique determination of an even damping $\alpha(x)=\alpha(1-x)$ from the Dirichlet eigenvalues $\{\lambda_j(A(\alpha))\}_{j\in J}$ is known (cf. \cite{Cox2011reconstructing}). However, it is not clear  whether there is a one-to-one correspondence between $\{\lambda_j(A(\alpha))\}_{j\in J}$ and $\{\sum_{j\in J}\lambda_j(A(\alpha))^{-n-1}\}_{n=0}^\infty$. It is also not clear whether the map \eqref{trace_map} is injective. In the next section, we consider the linearization of the map $\mathcal{F}$ at constant damping and show the injectivity of the linearized map.
\begin{theorem}
Assume $T=T_{\min}$. The Fr\'echet derivative of the map $\mathcal{F}$ at $\alpha=\alpha_0$, where $\alpha_0$ is a constant,
\begin{equation*}
\mathcal{F}'[\alpha_0]:\,\delta\alpha\rightarrow \{\mathbf{t}_n'[\alpha_0](\delta\alpha)\}_{n=0}^\infty=\{\mathrm{trace}(R_n'[\alpha_0](\delta\alpha))\}_{n=0}^\infty
\end{equation*}
is injective for $\delta\alpha(x)=\delta\alpha(1-x)$.
\end{theorem}
\begin{proof}
We calculate
\begin{equation*}
R_0'[\alpha_0](\delta\alpha)=-\delta\alpha(T^*T)^{-1},
\end{equation*}
\begin{equation*}
R_0(\alpha_0)= -\alpha_0(T^*T)^{-1},
\end{equation*}
and
\begin{equation*}
R_1'[\alpha_0](\delta\alpha)= \alpha_0\delta\alpha(T^*T)^{-2}+\alpha_0(T^*T)^{-1}\delta\alpha(T^*T)^{-1},
\end{equation*}
\begin{equation*}
R_1(\alpha_0)= \alpha_0^2(T^*T)^{-2}-2(T^*T)^{-1}.
\end{equation*}
We claim that
\begin{equation*}
\begin{split}
R_{n-1}'[\alpha_0](\delta\alpha)=& (-1)^n\alpha_0^{n-1}\Big(\delta\alpha(T^*T)^{-n}+(T^*T)^{-1}\delta\alpha(T^*T)^{-n+1}+\cdots+  \\
&(T^*T)^{-n+1}\delta\alpha(T^*T)^{-1}\Big) + \sum_{m+\ell\leq n-1}c_{n-1,m,\ell}(\alpha_0)(T^*T)^{-m}\delta\alpha(T^*T)^{-\ell},\\
R_{n-1}(\alpha_0)= &(-1)^n\alpha_0^n(T^*T)^{-n}+\sum_{k\leq n-1}d_{n-1,k}(\alpha_0)(T^*T)^{-k},
\end{split}
\end{equation*}
and prove by induction. Here $c_{n-1,m,\ell}(\alpha_0)$ and $d_{n-1,k}(\alpha_0)$ are some constants depending on $\alpha_0$. 
Using the recursive formula \eqref{recursive_Rn}, we have
\begin{equation*}
\begin{split}
R_n'[\alpha_0](\delta\alpha)=& -R_{n-2}'[\alpha_0](\delta\alpha)(T^*T)^{-1}-  R_{n-1}'[\alpha_0](\delta\alpha)\alpha_0(T^*T)^{-1}-R_{n-1}(\alpha_0)\delta\alpha(T^*T)^{-1}\\
= & (-1)^{n+1}\alpha_0^{n-1}\Big(\delta\alpha(T^*T)^{-n}+(T^*T)^{-1}\delta\alpha(T^*T)^{-n+1}+\cdots \\
& + (T^*T)^{-n+1}\delta\alpha(T^*T)^{-1}\Big)\alpha_0(T^*T)^{-1}+(-1)^{n+1}\alpha_0^n(T^*T)^{-n}\delta\alpha(T^*T)^{-1} \\
&+\sum_{m+\ell\leq n}c_{n,m,\ell}(\alpha_0)(T^*T)^{-m}\delta\alpha(T^*T)^{-\ell}\\
=&(-1)^{n+1}\alpha_0^n\Big(\delta\alpha(T^*T)^{-n-1}+(T^*T)^{-1}\delta\alpha(T^*T)^{-n}+\cdots\\
&+ (T^*T)^{-n}\delta\alpha(T^*T)^{-1}\Big)+\sum_{m+\ell\leq n}c_{n,m,\ell}(\alpha_0)(T^*T)^{-m}\delta\alpha(T^*T)^{-\ell}.
\end{split}
\end{equation*}
Similarly, we can prove
\begin{equation*}
R_{n}(\alpha_0)= (-1)^{n+1}\alpha_0^{n+1}(T^*T)^{-n-1}+\sum_{k\leq n}d_{n,k}(\alpha_0)(T^*T)^{-k}.
\end{equation*}
The claim is proved. This implies that
 \begin{equation*}
 \mathrm{trace}(R_{n-1}'[\alpha_0](\delta\alpha))=(-1)^{n+1}n\alpha_0^{n-1}\mathrm{trace}(\delta\alpha(T^*T)^{-n})+\sum_{k=1}^{n-1}c_{n,k}(\alpha_0)\mathrm{trace}(\delta\alpha(T^*T)^{-k})
 \end{equation*}
 with some constants $c_{n,k}(\alpha_0)$ depending on $\alpha_0$.

Therefore, if $\mathcal{F}'[\alpha_0](\delta\alpha)=0$, we have $\mathrm{trace}(R_{n-1}'[\alpha_0](\delta\alpha))=0$ for $n=1,2,\cdots$, and thus
\begin{equation*}
\mathrm{trace}(\delta\alpha (T^*T)^{-n})=\mathrm{trace}(\delta\alpha (-\Delta_D)^{-n})=0.
\end{equation*}
Equivalently, we have
\begin{equation*}
\int_0^1g_n(x,x)\delta\alpha(x)\mathrm{d}x=0,
\end{equation*}
where $g_n(x,y)$ is the Green's function for $(-\Delta_D)^n$, and by \textit{Mercer's Theorem} (see, for example, \cite{Lax}) we have
\begin{equation*}
g_n(x,y)=\sum_{m=1}^\infty\frac{2}{m^{2n}\pi^{2n}}\sin m\pi x\,\sin m\pi y,
\end{equation*}
and thus
\begin{equation*}
g_n(x,x)=\sum_{m=1}^\infty\frac{2}{m^{2n}\pi^{2n}}(\sin m\pi x)^2=\sum_{m=1}^\infty\frac{2}{m^{2n}\pi^{2n}}\frac{1-\cos 2m\pi x}{2}.
\end{equation*}
Notice that
\begin{equation*}
0=\lim_{n\rightarrow+\infty}\pi^{2n}\int_0^1g_n(x,x)\delta\alpha(x)\mathrm{d}x=\int_0^1\delta\alpha(x)(1-\cos 2\pi x)\mathrm{d}x.
\end{equation*}
Then
\begin{equation*}
\int_0^1\left(\sum_{m=2}^\infty\frac{2}{m^{2n}\pi^{2n}}\frac{1-\cos 2m\pi x}{2}\right)\delta\alpha(x)\mathrm{d}x=0,
\end{equation*}
for every $n$. Then
\begin{equation*}
0=\lim_{n\rightarrow+\infty}2^{2n}\pi^{2n}\int_0^1\left(\sum_{m=2}^\infty\frac{2}{m^{2n}\pi^{2n}}\frac{1-\cos 2m\pi x}{2}\right)\delta\alpha(x)\mathrm{d}x=\int_0^1\delta\alpha(x)(1-\cos 4\pi x)\mathrm{d}x.
\end{equation*}
Continuing this process, we have
\begin{equation*}
\int_0^1\delta\alpha(x)(1-\cos 2m\pi x)\mathrm{d}x=0
\end{equation*}
for each $m$. Taking the limit $m\rightarrow+\infty$, and invoking the \textit{Riemann-Lebesgue Lemma}
\begin{equation*}\lim_{m\rightarrow +\infty}\int_{0}^1\delta\alpha(x)\cos 2m\pi x\mathrm{d}x=0,\end{equation*}
 we have
\begin{equation*}
\int_0^1\delta\alpha(x)\mathrm{d}x=0.
\end{equation*}
Thus we end up with
\begin{equation*}
\int_0^1\delta\alpha(x)\cos 2m\pi x\mathrm{d}x=0
\end{equation*}
for $m=0,1,2,\cdots$. Then $\delta\alpha=0$, and the injectivity is proved. \hfill \end{proof} 


\section{Inversion Algorithm}\label{sec4}
We derive an algorithm for recovering $\alpha(x)$ from the spectral of $A(\alpha)$ based on the trace formulas derived in Section \ref{sec2}. We only describe the algorithm for the operator with Dirichlet boundary condition, that is $T=T_{\min}$. Other boundary conditions can be dealt with in the same way.

Assume $\{\mu_\ell,\phi_\ell(x)\}_{\ell=1}^\infty$ are the eigenvalues and eigenfunctions of $-\Delta_D={T_{\min}}^*T_{\min}$, where
\begin{equation*}
-\Delta_D\phi_\ell=-\phi_\ell''=\mu_\ell\phi_\ell,
\end{equation*}
\begin{equation*}
\mu_\ell=\ell^2\pi^2,\quad\phi_\ell(x)=\sqrt{2}\sin \ell\pi x.
\end{equation*}
Define the unitary operator $\mathbf{W}:L^2(0,1)\rightarrow l^2$ such that
\begin{equation*}
\mathbf{W}f=\{f_1,f_2,\cdots\},
\end{equation*}
where $f$ admits the decomposition under the basis $\{\phi_\ell\}_{\ell=1}^\infty$ of $L^2(0,1)$:
\begin{equation*}
f(x)=\sum_{\ell=1}^\infty f_\ell\phi_\ell(x).
\end{equation*}
Then we have the spectral decomposition of $(-\Delta_D)^{-1}$
as
\begin{equation*}
(-\Delta_D)^{-1}=\mathbf{W}^{-1}\mathrm{diag}\left(\mu_1^{-1},\,\mu_2^{-1},\,\cdots,\mu_\ell^{-1},\cdots\right)\mathbf{W}.
\end{equation*}
Similarly, the multiplication operator $M_\alpha:L^2(0,1)\rightarrow L^2(0,1)$: $(M_\alpha f)(x)= \alpha(x) f(x)$ also can be decomposed as follows
\begin{equation*}
M_\alpha=\mathbf{W}^{-1}\mathbf{M}(\alpha)\mathbf{W},
\end{equation*}
where
\begin{equation*}
(\mathbf{M}(\alpha))_{ij}=\int_0^1\alpha(x)\phi_i(x)\phi_j(x)\mathrm{d}x.
\end{equation*}
To see this, one only needs to notice
\begin{equation*}
\begin{split}
(M_\alpha f)(x)=&\sum_{i=1}^\infty\left(\int_0^1\alpha f(x)\phi_i(x)\mathrm{d}x\right)\phi_i(x)\\
=&\sum_{i=1}^\infty\left(\int_0^1\alpha(x) \sum_{j=1}^\infty f_j\phi_j(x)\phi_i(x)\mathrm{d}x\right)\phi_i(x)\\
=&\sum_{i=1}^\infty\left(\sum_{j=1}^\infty\left(\int_0^1\alpha(x)\phi_j(x)\phi_i(x)\mathrm{d}x\right)f_j\right)\phi_i(x).
\end{split}
\end{equation*}

Denote
\begin{equation*}
(\mathbf{M}_1(\alpha))_{ij}=-\mu_j^{-1}\int_0^1\alpha(x)\phi_{i}(x)\phi_{j}(x)\mathrm{d}x,
\end{equation*}
\begin{equation*}
\mathbf{M}_2(\alpha)=2\mathbf{M}_1(1)+\mathbf{M}_1(\alpha)^2.
\end{equation*}
By direct calculation, it is easy to see that
\begin{align*}
R_0(\alpha)&=-\alpha(T^{*}T)^{-1}=-M_\alpha (-\Delta_D)^{-1}\\
&= -\mathbf{W}^{-1}\mathbf{M}(\alpha)\mathbf{W} \mathbf{W}^{-1}\mathrm{diag}\left(\mu_1^{-1},\,\mu_2^{-1},\,\cdots,\mu_n^{-1},\cdots\right)\mathbf{W}\\
&=  \mathbf{W}^{-1}\mathbf{M}_1(\alpha)\mathbf{W},
\end{align*}
and therefore
\begin{equation*}
\begin{split}
R_1(\alpha)=&-2(-\Delta_D)^{-1}+M_\alpha (-\Delta_D)^{-1} M_\alpha (-\Delta_D)^{-1}\\
=& 2\mathbf{W}^{-1}\mathbf{M}_1(1)\mathbf{W} + \mathbf{W}^{-1}\mathbf{M}_1(\alpha)\mathbf{W} \mathbf{W}^{-1}\mathbf{M}_1(\alpha)\mathbf{W}\\
=&\mathbf{W}^{-1}\mathbf{M}_2(\alpha)\mathbf{W}.
\end{split}
\end{equation*}
Generally, we define
\begin{align}
\mathbf{M}_{n}(\alpha)=\mathbf{M}_{n-1}(\alpha)\mathbf{M}_1(\alpha)+\mathbf{M}_{n-2}(\alpha)\mathbf{M}_1(1).
\end{align}
Then, one can verify that
\begin{align}
R_{n-1}(\alpha)=\mathbf{W}^{-1}\mathbf{M}_{n}(\alpha)\mathbf{W}.
\end{align}

Since $\mathbf{W}$ is a unitary operator which can be viewed as a rotation transformation and keeps eigenvalues invariant when both $\mathbf{W}$ and $\mathbf{W}^{-1}$ are applied, thus we have
\begin{proposition}
The following relations hold:
\begin{align}\mathrm{trace}(\mathbf{M}_{n}(\alpha))=\mathrm{trace}(R_{n-1}(\alpha))=\sum_{j\in J}\lambda_j^{-n}, \mbox{ for } n= 1,2,\cdots. \label{discretetrace}\end{align}
\end{proposition}
When $n=1$ in the above formula, we need to use the regularized summation as in Remark \ref{cauchy_sum}.
\begin{remark}
The proposition gives an explicit expression between the damping coefficient $\alpha(x)$ and the spectral data $\{\lambda_j\}_{j\in J}$ in terms of a series of Fredholm equations. For example if $n=1$, then we have
\begin{align*}
\sum_{\ell =1}^\infty-\mu_\ell^{-1}\int_{0}^{1}\alpha(x)\phi_\ell^2(x)dx = \sum_{j\in J}\lambda_j^{-1}.
\end{align*}

Solving an infinite series of Fredholm integral equations \eqref{discretetrace} is severely ill-posed. The main reason is that
\begin{align}\label{least_square}
\sum_{n=1}^{N}\left(\sum_{j\in J}\lambda_j^{-n}(\alpha) - \sum_{j\in J}\lambda_j^{-n}(\alpha_{true})\right)^2
\end{align}
 is not a good choice to measure the misfit. As in \cite{XZ1}, we need to use a sequence of ``proper" polynomials $\{T_n\}_{n=1}^N$ and measure the misfit as
 \begin{align}\label{residual_measure}
\sum_{n=1}^{N}\left(\sum_{j\in J}T_n(\lambda_j^{-1}(\alpha)) - \sum_{j\in J}T_n(\lambda_j^{-1}(\alpha_{true}))\right)^2.
\end{align}

\end{remark}
\par Before proceeding to seeking proper polynomials, which is critical to the success of the inversion, let us first summarize some properties of the spectrum of $A(\alpha)$. We refer to \cite{BORISOV2009,CoXZua1994} for more details.


Assume $\alpha_0=\int_0^1\alpha(x)\mathrm{d}x$. Then
\begin{enumerate}
\item The spectrum of $A(\alpha)$ is symmetric about the real axis, i.e., $\sigma_p(A(\alpha))=\overline{\sigma_p(A(\alpha))}$;
\item The spectrum of $A(\alpha)$ is contained in
\begin{equation*}
\{\lambda\in\mathbb{C}:\,|\lambda|\geq \pi,\, -b\leq\mathrm{Re}\,\lambda\leq -a\}\cup[-b-(b^2-\pi^2)^{1/2}_+,-a+(b^2-\pi^2)^{1/2}_+];
\end{equation*}
\item The eigenvalue $\lambda_j(A(\alpha))$ has the asymptotic behavior
\begin{equation}\label{eigenvalue_asymptotic}
\lambda_j(A(\alpha))=-\frac{\alpha_0}{2}+j \pi\i +\mathcal{O}\left(\frac{1}{j}\right).
\end{equation}
\end{enumerate}
The distribution of a sample damping coefficient is depicted in Figure \ref{eigen_distribution}.

\begin{figure}[t!]
    \begin{subfigure}{0.32\textwidth}
      \includegraphics[width=\textwidth]{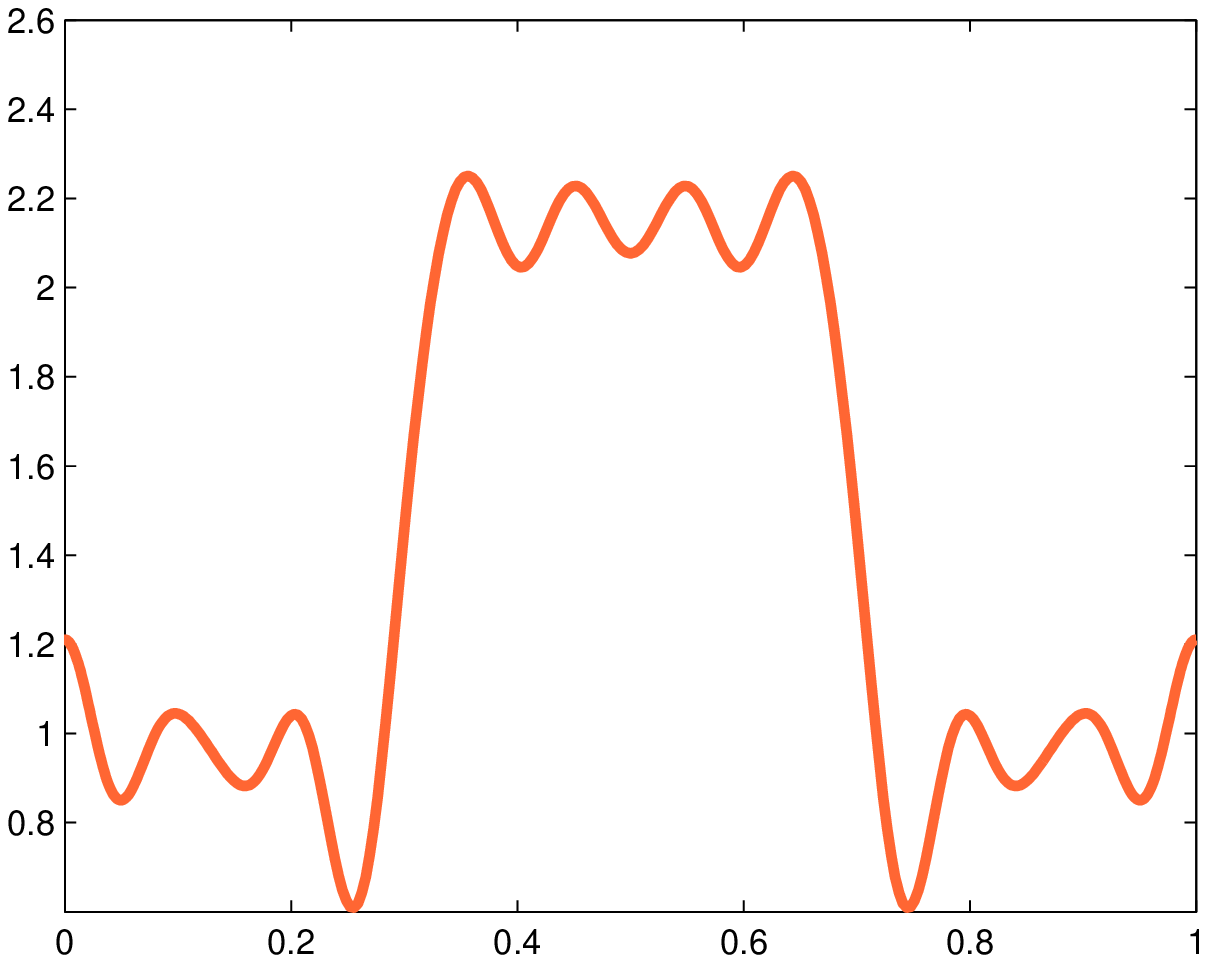}
      \caption{ the damping coefficient}
    \end{subfigure}
    \begin{subfigure}{0.32\textwidth}
      \includegraphics[width=\textwidth]{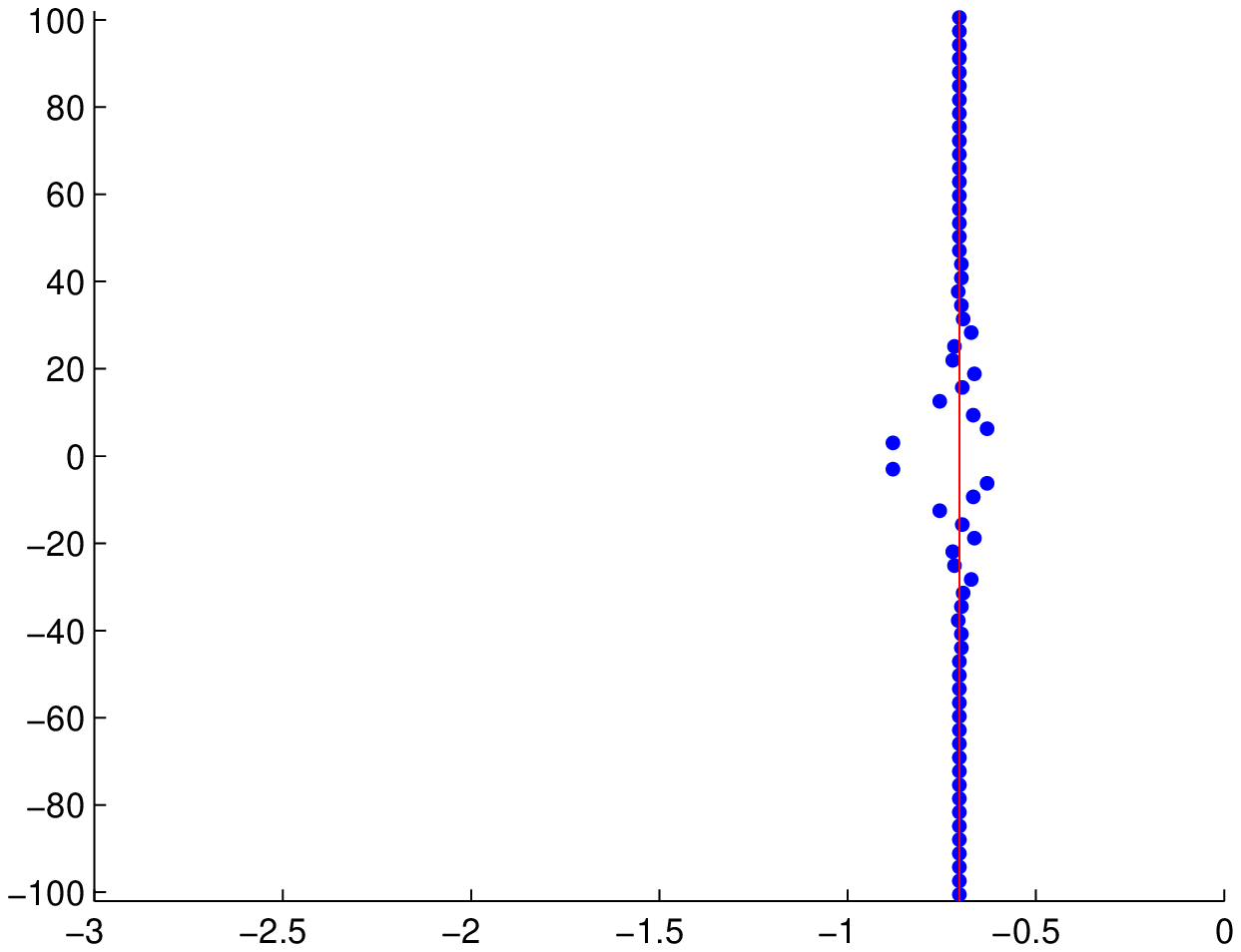}
      \caption{ distribution of the eigenvalues}
    \end{subfigure}
     \begin{subfigure}{0.32\textwidth}
      \includegraphics[width=\textwidth]{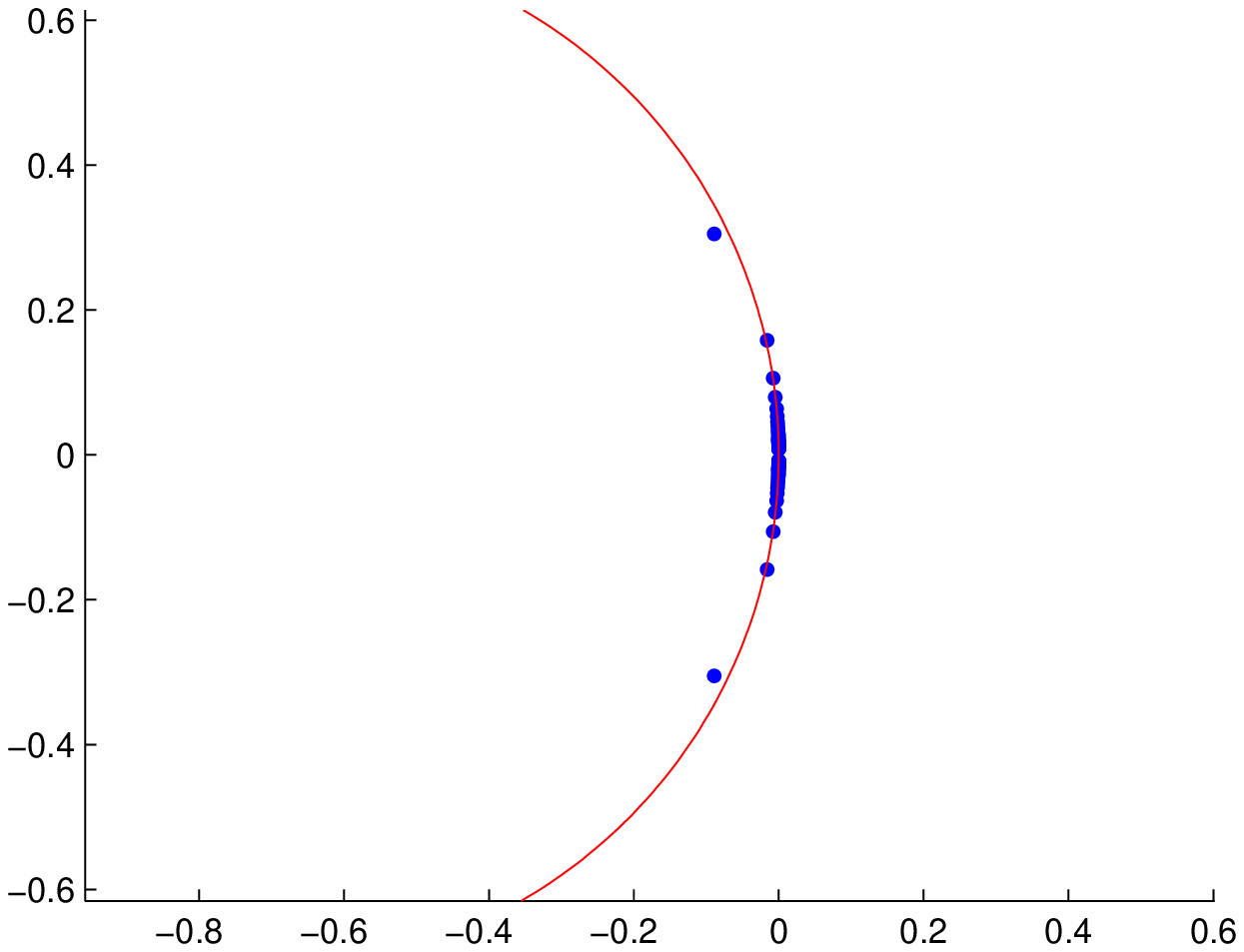}
      \caption{  the reciprocal of the eigenvalues}
    \end{subfigure}
    \caption{Distribution of eigenvalues}\label{eigen_distribution}
  \end{figure}

Recall that the conformal mapping $\frac{1}{z}$ on the complex plane $\mathbb{C}$ maps the line $\{\Re z=-\frac{\alpha_0}{2}\}$ to the circle $B_{(-\frac{1}{\alpha_0},0)}(\frac{1}{\alpha_0})$, then $\{\lambda_j(A(\alpha))^{-1}\}_{j\in J}$ scatter near that circle if  the damping $\alpha$ is not large, see Figure \ref{eigen_distribution}(c). We need the polynomials to be well-behaved on the circle, and create enough oscillations near $z=0$ to discriminate the measured eigenvalues.
We use the polynomials
\begin{equation*}
T_n(z)=z\left(\alpha_0 z+1\right)^{n-1},
\end{equation*}
where $\alpha_0$ is approximated using the asymptotics \eqref{eigenvalue_asymptotic}.
Moreover, denote
\begin{equation*}
\widetilde{T}_n(z)=z^2\left(\alpha_0 z+1\right)^{n-2}=z T_{n-1}(z)=\frac{1}{\alpha_0}T_n(z)-\frac{1}{\alpha_0}T_{n-1}(z).
\end{equation*}
We use the following recursive relation for the polynomials of $T_n$.
\begin{equation}\label{recursive_Tn}
\begin{split}
T_{n+1}(z)=&(\alpha_0 z+1)T_n(z)\\
=&(\alpha_0^2 z^2+2\alpha_0 z+1)T_{n-1}(z)\\
=&\alpha_0^2 z^2T_{n-1}(z)+2\alpha_0 z^2(\alpha_0z+1)^{n-2}+T_{n-1}(z)\\
=&\alpha_0^2 z^2T_{n-1}(z)+2\alpha_0\widetilde{T}_n(z)+T_{n-1}(z)\\
=&\alpha_0^2 z^2T_{n-1}(z)+2(T_n(z)-T_{n-1}(z))+T_{n-1}(z)\\
=& 2T_n(z)-T_{n-1}(z)+\alpha_0^2z^2T_{n-1}(z).
\end{split}
\end{equation}
We note $z^2T_{n-1}(z)=z\widetilde{T}_n(z)$ for later use.

\begin{remark}
This choice of polynomials does not work well for large dampings, for which the eigenvalues $\lambda_j^{-1}$ for $j$ small might be far away from the circle $B_{(-\frac{1}{\alpha_0},0)}(\frac{1}{\alpha_0})$. See Figure \ref{eigen_distribution_extreme} for the distribution of the eigenvalues for an example of Freitas \cite{freitas1999optimizing},
\begin{equation}\label{example_larged}
\alpha(x)=\frac{3.1133\pi}{2}+1.4896\pi\cos2\pi x.
\end{equation}
Notice that we actually have $2b = \sup_{x\in[0,1]}\alpha(x)> 2\pi=2\sqrt{\mu_1(T^*T)}$.
However, it still can be used for low frequency approximation, which will be demonstrated by Example \ref{example5} in the next section. Also, a more complicated strategy for choosing polynomials might enable one to go to higher frequencies. 
\end{remark}

\begin{figure}[t!]
    \begin{subfigure}{0.32\textwidth}
      \includegraphics[width=\textwidth]{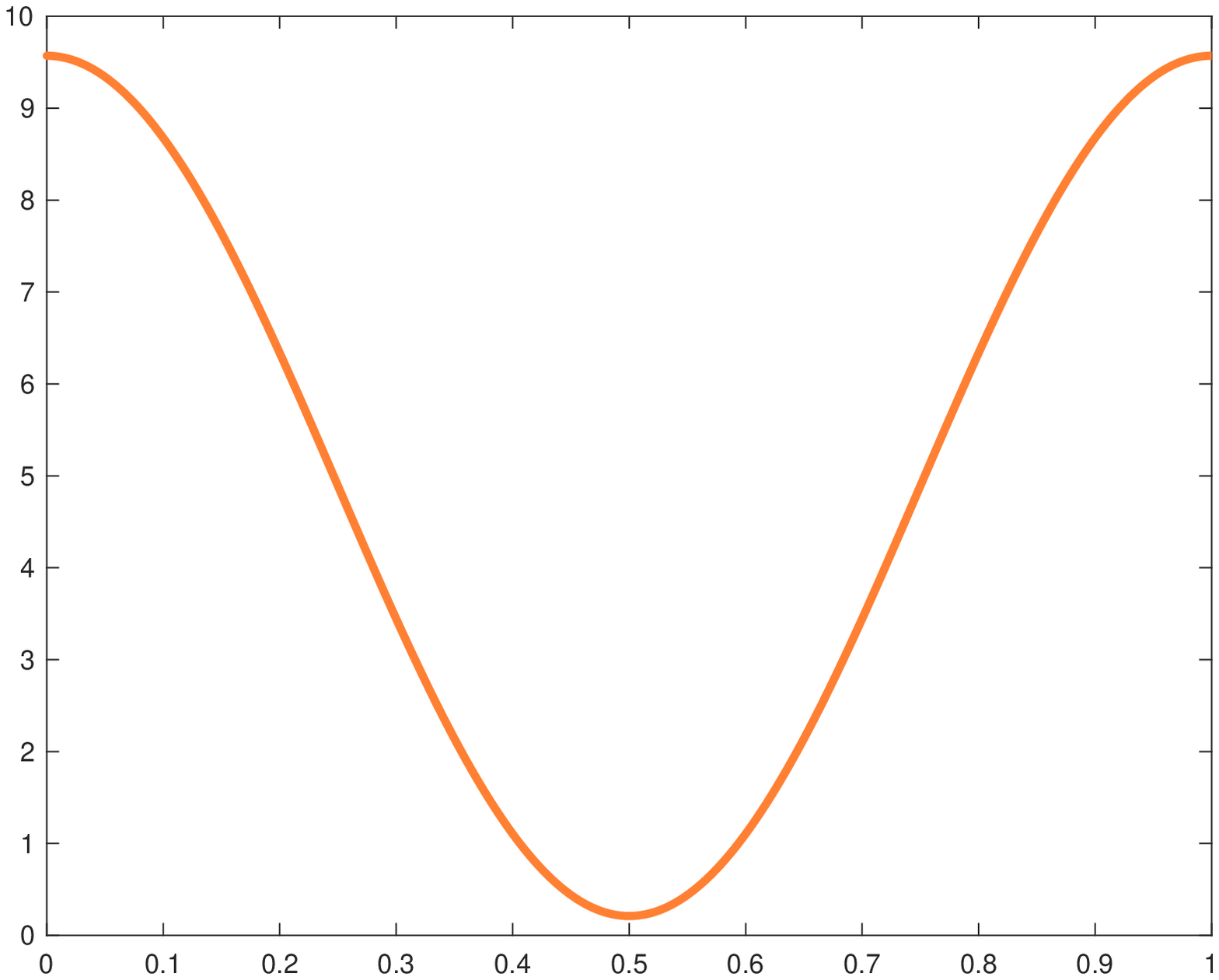}
      \caption{ the damping coefficient}
    \end{subfigure}
    \begin{subfigure}{0.32\textwidth}
      \includegraphics[width=\textwidth]{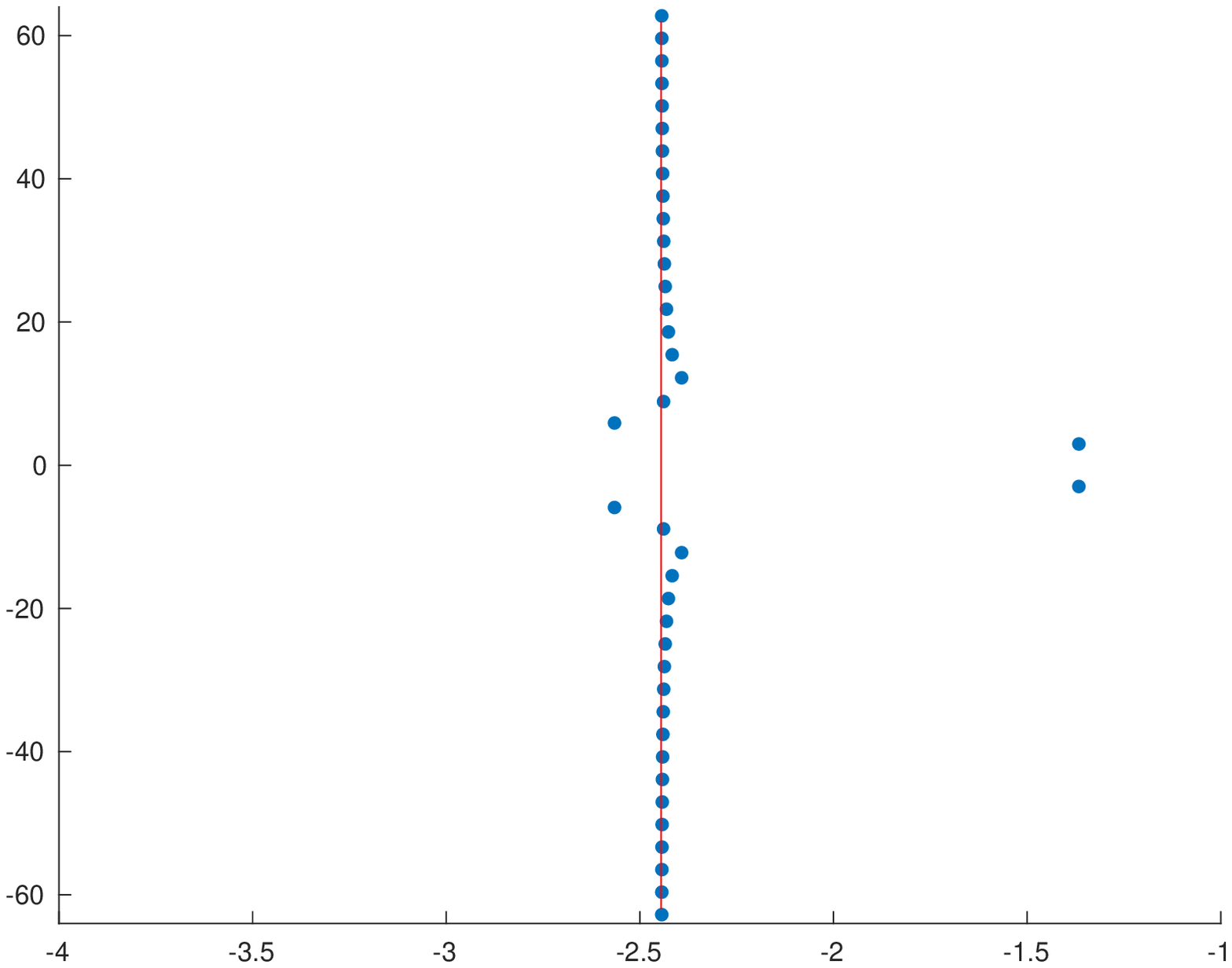}
      \caption{ distribution of the eigenvalues}
    \end{subfigure}
     \begin{subfigure}{0.32\textwidth}
      \includegraphics[width=\textwidth]{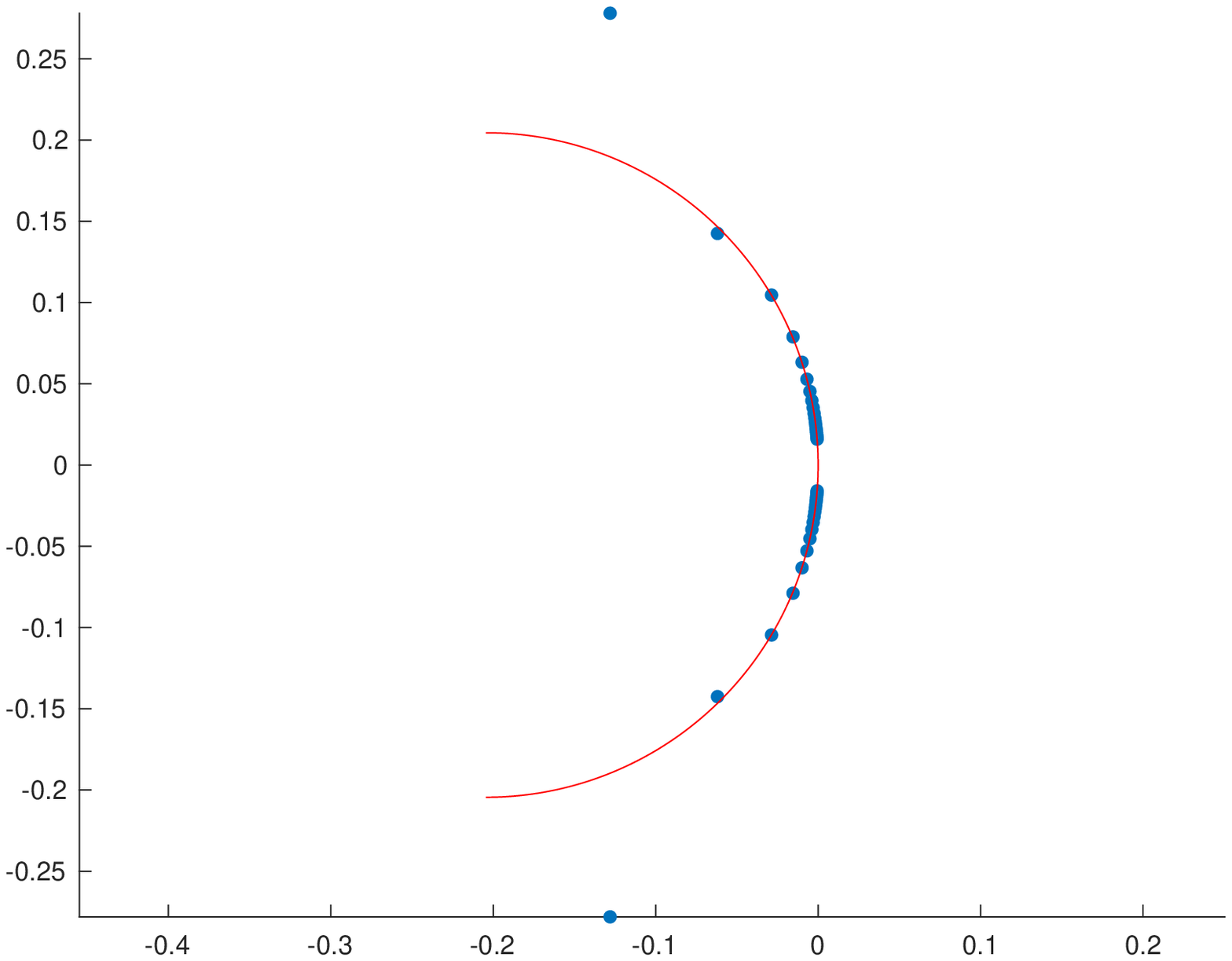}
      \caption{  the reciprocal of the eigenvalues}
    \end{subfigure}

    \caption{Distribution of eigenvalues for large damping}\label{eigen_distribution_extreme}
  \end{figure}

We use truncated Fourier cosine series to approximate an even damping coefficient,
 \begin{equation}\label{cosine_expansion1}
\alpha_M(x)=\sum_{m=1}^M a_m\cos 2(m-1)\pi x,
\end{equation}
and denote $\mathbf{a}=\{a_1,a_2,\cdots,a_M\}$.
With a little abuse of notations, we use $\mathbf{M}_n(\mathbf{a})$ in place of  $\mathbf{M}_n(\alpha)$ in the following. Then
\begin{equation*}
\mathbf{M}_1(\mathbf{a})=\sum_{m=1}^M a_m \mathbf{M}_1(\mathbf{e}_m),
\end{equation*}
where
\begin{equation*}
\begin{split}
(\mathbf{M}_1(\mathbf{e}_m))_{ij}=&-\frac{2}{\pi^2j^2}\int_0^1\sin i\pi x\,\sin j\pi x\cos2(m-1)\pi x\mathrm{d}x\\
=&\begin{cases}
\frac{1}{2 \pi^2j^2},\quad & i+j+2m-2=0,\\
\frac{1}{2 \pi^2j^2},\quad & i+j-2m+2=0,\\
-\frac{1}{2 \pi^2j^2},\quad & i-j+2m-2=0,\,m\neq 1,\\
-\frac{1}{2 \pi^2j^2},\quad & i-j-2m+2=0,\,m\neq 1,\\
-\frac{1}{\pi^2j^2},\quad & i=j,m=1,\\
0,\quad &\text{otherwise},
\end{cases}
\end{split}
\end{equation*}
and $\mathbf{e}_m=\{a_1=0,\cdots,a_{m-1}=0,a_m=1,a_{m+1}=0,\cdots,a_M=0\}$.
\begin{remark}
The matrix $\mathbf{M}_1(\mathbf{e}_m)$ here is not a symmetric matrix, in contrast to the one defined in \cite{XZ1}.
\end{remark}

Next we define
\begin{equation*}
\mathbf{T}_1(\mathbf{a})=\mathbf{M}_1(\mathbf{a}),
\end{equation*}
\begin{equation*}
\mathbf{T}_2(\mathbf{a})=\alpha_0\mathbf{M}_2(\mathbf{a})+\mathbf{M}_1(\mathbf{a})=\alpha_0(2\mathbf{M}_1(\mathbf{e}_1)+\mathbf{M}_1(\mathbf{a})^2)+\mathbf{M}_1(\mathbf{a}),
\end{equation*}
and
\begin{equation*}
\widetilde{\mathbf{T}}_n(\mathbf{a})=\frac{1}{\alpha_0}(\mathbf{T}_n(\mathbf{a})-\mathbf{T}_{n-1}(\mathbf{a})),
\end{equation*}
\begin{equation*}
\begin{split}
\mathbf{T}_{n+1}(\mathbf{a})=&2\mathbf{T}_n(\mathbf{a})-\mathbf{T}_{n-1}(\mathbf{a})+\alpha_0^2 \left(\mathbf{T}_{n-1}(\mathbf{a})\mathbf{M}_1(\mathbf{e}_1)+\widetilde{\mathbf{T}}_n(\mathbf{a})\mathbf{M}_1(\mathbf{a})\right)\\
=&2\mathbf{T}_n(\mathbf{a})-\mathbf{T}_{n-1}(\mathbf{a})+\alpha_0^2 \mathbf{T}_{n-1}(\mathbf{a})\mathbf{M}_1(\mathbf{e}_1)+\alpha_0(\mathbf{T}_n(\mathbf{a})-\mathbf{T}_{n-1}(\mathbf{a}))\mathbf{M}_1(\mathbf{a}),
\end{split}
\end{equation*}
for $n=2,3,\cdots$, in parallel with \eqref{recursive_Tn}.
One can check that if $T(z)=\sum_{i=1}^n b_iz^i$, then $\mathbf{T}_{n}(\mathbf{a})=\sum_{i=1}^nb_i\mathbf{M}_i(\mathbf{a})$. Therefore, we obtain
\begin{proposition}
For $n=1,2,\cdots$,
\begin{equation*}
\mathrm{trace}(\mathbf{T}_{n}(\mathbf{a}))=\sum_{j\in J}T_{n}(\lambda_j^{-1}).
\end{equation*}
\end{proposition}

In light of the above proposition, we invert the map
\begin{equation*}
\mathbf{a}\rightarrow\{\mathrm{trace}(\mathbf{T}_{n}(\mathbf{a}))\}_{n=1}^N.
\end{equation*}

Applying the chain rule and the recursive formula for $\mathbf{T}_n(\mathbf{a})$, we have the following recursive formula for the Fr{\'e}chet derivatives
\begin{equation*}
\frac{\partial\mathbf{T}_1(\mathbf{a})}{\partial a_m}=\mathbf{M}_1(\mathbf{e}_m),
\end{equation*}
\begin{equation*}
\frac{\partial\mathbf{T}_2(\mathbf{a})}{\partial a_m}=\mathbf{M}_1(\mathbf{e}_m)+\alpha_0\mathbf{M}_1(\mathbf{a})\mathbf{M}_1(\mathbf{e}_m)+\alpha_0\mathbf{M}_1(\mathbf{e}_m)\mathbf{M}_1(\mathbf{a}),
\end{equation*}
and
\begin{equation*}
\begin{split}
\frac{\partial \mathbf{T}_{n+1}(\mathbf{a})}{\partial a_m}=&2\frac{\partial \mathbf{T}_{n}(\mathbf{a})}{\partial a_m}-\frac{\partial \mathbf{T}_{n-1}(\mathbf{a})}{\partial a_m}+\alpha_0^2\frac{\partial \mathbf{T}_{n-1}(\mathbf{a})}{\partial a_m}\mathbf{M}_1(\mathbf{e}_1)\\
&+\alpha_0\left(\frac{\partial \mathbf{T}_{n}(\mathbf{a})}{\partial a_m}-\frac{\partial \mathbf{T}_{n-1}(\mathbf{a})}{\partial a_m}\right)\mathbf{M}_1(\mathbf{a})+\alpha_0(\mathbf{T}_n(\mathbf{a})-\mathbf{T}_{n-1}(\mathbf{a}))\frac{\partial\mathbf{M}_1(\mathbf{a})}{\partial a_m}.
\end{split}
\end{equation*}
Now we can summarize the algorithm in Algorithm \ref{algorithm}.
\begin{algorithm}
\caption{Inversion of trace formulas for the damped wave operator}\label{algorithm}
\begin{algorithmic}[1]
\STATE precompute $\mathbf{M}(\mathbf{e}_m)$, traces $r_1^{true}=\sum_{k=1}^\infty T_1(\lambda_k^{-1}),\ldots, r_N^{true}=\sum_{k=1}^\infty T_N(\lambda_k^{-1})$
\STATE get an approximate value of $\alpha_0$ from measured eigenvalues
\STATE given initial guess $\mathbf{a}_0$
\FOR{$1\leq n\leq$ max number of iterations}
\STATE form $\mathbf{T}_1(\mathbf{a}_{n-1})$, $\mathbf{T}_2(\mathbf{a}_{n-1})$, $\frac{\partial\mathbf{T}_1(\mathbf{a}_{n-1})}{\partial a_m}$, $\frac{\partial\mathbf{T}_2(\mathbf{a}_{n-1})}{\partial a_m}$
\STATE $r_{1}=\textbf{trace}(\mathbf{T}_1(\mathbf{a}_{n-1}))$
\STATE $r_{2}=\textbf{trace}(\mathbf{T}_2(\mathbf{a}_{n-1}))$
\FOR{$1\leq m\leq M$}
\STATE $J_{1,m}=\textbf{trace}\left(\frac{\partial\mathbf{T}_1(\mathbf{a}_{n-1})}{\partial a_m}\right)$
\STATE $J_{2,m}=\textbf{trace}\left(\frac{\partial\mathbf{T}_2(\mathbf{a}_{n-1})}{\partial a_m}\right)$
\ENDFOR
\FOR{$2\leq j\leq N-1$}
\STATE $\mathbf{T}_{j+1}(\mathbf{a}_{n-1})=2\mathbf{T}_j(\mathbf{a}_{n-1})-\mathbf{T}_{j-1}(\mathbf{a})+\alpha_0^2 \mathbf{T}_{j-1}(\mathbf{a}_{n-1})\mathbf{M}_1(\mathbf{e}_1)+\alpha_0(\mathbf{T}_j(\mathbf{a}_{n-1})-\mathbf{T}_{j-1}(\mathbf{a}_{n-1}))\mathbf{M}_1(\mathbf{a}_{n-1})$
\FOR{$1\leq m\leq M$}
\STATE $\frac{\partial \mathbf{T}_{j+1}(\mathbf{a}_{n-1})}{\partial a_m}=2\frac{\partial \mathbf{T}_{j}(\mathbf{a}_{n-1})}{\partial a_m}-\frac{\partial \mathbf{T}_{j-1}(\mathbf{a}_{n-1})}{\partial a_m}+\alpha_0(\mathbf{T}_j(\mathbf{a}_{n-1})-\mathbf{T}_{j-1}(\mathbf{a}_{n-1}))\frac{\partial\mathbf{M}_1(\mathbf{a}_{n-1})}{\partial a_m}$\\ $
\quad\quad\quad\quad\quad\quad+\alpha_0\left(\frac{\partial \mathbf{T}_{j}(\mathbf{a}_{n-1})}{\partial a_m}-\frac{\partial \mathbf{T}_{j-1}(\mathbf{a}_{n-1})}{\partial a_m}\right)\mathbf{M}_1(\mathbf{a}_{n-1})+\alpha_0^2\frac{\partial \mathbf{T}_{j-1}(\mathbf{a}_{n-1})}{\partial a_m}\mathbf{M}_1(\mathbf{e}_1)$
\ENDFOR
\STATE $r_{j+1}=\textbf{trace}(\mathbf{T}_{j+1}(\mathbf{a}_{n-1}))$
\FOR{$1\leq m\leq M$}
\STATE $J_{j+1,m}=\textbf{trace}\left(\frac{\partial \mathbf{T}_{j+1}(\mathbf{a}_{n-1})}{\partial a_m}\right)$
\ENDFOR
 \ENDFOR
 \STATE compute $\delta \mathbf{a}$ using Jacobi $\mathbf{J}=(J_{j,m})_{N\times M}$ and residual $\mathbf{r}^{true}-\mathbf{r}=(r_j^{true}-r_j)_{N\times 1}$
 \STATE $\mathbf{a}_n=\mathbf{a}_{n-1}+\delta\mathbf{a}$
 \ENDFOR
\end{algorithmic}
\end{algorithm}
\begin{remark}
Note that the trace formulas involve infinite sums. But realistically we can only have a finite number of measured eigenvalues. Assume we have $2K$ measured eigenvalues, say $\{\lambda_j\}_{j=-K}^K$, we can approximate the infinite sum
\begin{equation*}
\sum_{j\in J}T_{n}(\lambda_j^{-1})=\sum_{j=-K}^KT_{n}(\lambda_j^{-1})+\sum_{|j|\geq K+1}T_{n}(\lambda_j^{-1}),
\end{equation*}
by
\begin{equation*}
\sum_{j\in J}T_{n}(\lambda_j^{-1})\approx\sum_{j=-K}^KT_{n}(\lambda_j^{-1})+\sum_{K+1\leq |j|\leq K_1}T_{n}((-\frac{\alpha_0}{2}+j \pi\i)^{-1}),
\end{equation*}
noticing $\lambda_j\approx-\frac{\alpha_0}{2}+j \pi\i$ (cf. \eqref{eigenvalue_asymptotic}).
\end{remark}

\section{Numerical experiments}\label{sec5}

In this section we conduct some numerical experiments to illustrate the efficiency of Algorithm \ref{algorithm}. We design five examples to show reconstructions for smooth or non-smooth damping coefficients with accurate or inaccurate data. To generate synthetic data, we use Chebyshev pseudo-spectral collocation method to discretize the Laplacian operator $\Delta=\frac{\mathrm{d}^2}{\mathrm{d}x^2}$, using Trefethen's $\mathtt{cheb.m}$ routine \cite{trefethen2000spectral}.
We use $400$ Chebyshev points to discretize the Laplacian.
For all computations, Gauss-Newton is used as the optimization algorithm with tolerance set to $10^{-5}\times N$.

The parameters in the algorithm are listed in Table \ref{Symbol}. We discuss the impacts of different choices of these parameters on the performance of the algorithm.

\begin{table}[h]
\centering
  \begin{tabular}{|c | c |}
  \hline
 notation & parameter\\
  \hline
$K$ & $2K$: number of ``true" eigenvalues measured \\
    \hline
 $M$&number of basis functions\\
  \hline
  $J$ & $J\times J$: size of the truncated matrix $\mathbf{M}$\\
  \hline
    $N$ & highest degree of the polynomials \\
  \hline
  $K_1$ &$2K_1$: total number of eigenvalues utilized in traces\\
        & i.e., $2(K_1-K)$ ``approximated" eigenvalues\\
  \hline
  \end{tabular}
  \caption{Parameters for the algorithm}\label{Symbol}
\end{table}
 It is learned from \cite{Cox2011reconstructing} that the $m$-th eigenvalue may encode the $m$-th Fourier modes information of $\alpha$. Hence, we usually take $M= K$ for numerical reconstructions.

\begin{example} \label{example1}
Set the damping coefficient as follows:
$$\alpha(x) = - \exp(-(x-\frac{1}{2})^2)+ 8(x-\frac{1}{2})^4 + 6(x-\frac{1}{2})^2  + 1.25.$$
In Table \ref{Table1}, we list the first $K=4$ eigenvalues with positive imaginary parts for the true damping $\alpha$, the reconstructed one $\alpha_M$ and the Fourier approximation $\alpha_F$.
We see that when the number $K=M$ increases from 4 to 8 simultaneously, the accuracy of reconstruction will be improved.
\begin{table}[h]
{\small
\caption{True eigenvalues \textit{vs} eigenvalues for reconstructed $a_M(x)$ and the Fourier approximation}\label{Table1}
\begin{tabular}{|c|c|c|c|c|}
  \hline
       & $\lambda_1$ & $\lambda_2$ & $\lambda_3$ & $\lambda_4$ \\
    \hline
    true $\lambda_j$  & -0.2493 + 3.1335i & -0.3996 + 6.2742i &  -0.4343 + 9.4142i & -0.4469 +12.5566i \\
    \hline
       $\tilde{\lambda}_j, K=M=4 $ & -0.2493 + 3.1335i  & -0.3997 + 6.2744i & -0.4380 + 9.4141i & -0.4483 +12.5560i \\
   $|\lambda_j-\tilde{\lambda}_j|$   & 0.0000 & 0.0002 & 0.0036 & 0.0015 \\
   \hline
    $\tilde{\lambda_j}, K=M=8$ &  -0.2493 + 3.1335i &  -0.3996 + 6.2742i & -0.4342 + 9.4143i &  -0.4487 +12.5563i \\
    $|\lambda_j-\tilde{\lambda}_j|$ & 0.0000  & 0.0000 & 0.0001 &  0.0018 \\
   \hline
     $\tilde{\lambda}_j(A(\alpha_F))$  & -0.2493 + 3.1335i & -0.3996 + 6.2742i & -0.4343 + 9.4142i & -0.4469 +12.5566i \\
    $|\lambda_j-\tilde{\lambda}_j|$  & 0.0000 & 0.0000 & 0.0000 & 0.0000 \\
  \hline
\end{tabular}
}
\end{table}

However, there exists a balance between different parameters. When we fix $K$ and $K_1$ and then increase $M$, it does not always give a better result, see Table \ref{Table2} where the error is defined in $L^2$-norm, i.e., $\int_{0}^{1}|\alpha(x)-\alpha_M(x)|^2dx$. For instance, from Table \ref{Table2}, we can find that when $K_1$ and $K$ are fixed and $M$ is increasing, the error decreases at the beginning and then increases. It indicates that $M$ does play the role as a regularization parameter and depend on the accuracy of trace formulas, which is in fact determined by the number of known eigenvalues $K_1$ and $J$. In the following numerical simulations, we take a reasonable choice of $K_1 =J$ to avoid rounding error which may affect the accuracy of approximation of trace formulas.
\begin{table}[H]
	\begin{center}
		\caption{Inversion errors of damping coefficients in $L^2$ norm.}\label{Table2}
		\begin{tabular}{ccccccc}
			\hline
  $K=8$    &         M=3        & M=4    & M=5 & M=6   & M=7 & M =8 \\
 $K_1=J=25$, $N=25$&       $0.0051$   & 0.0144 &  0.0216   & $ 0.0242$  & 0.0247&0.0248 \\
			$K_1=J=50$, $N=50$&   $ 0.0071$  & 0.0052 &   0.0194   & $ 0.0206$ &  0.0209 &0.0209\\
			 $K_1=J=100$, $N=100$&  $0.0080 $  & 0.0032 & 0.0061    & $0.0090$ & 0.0209 & 0.0294\\
			 $K_1 =J=150$, $N=150$& $ 0.0081$  & 0.0052 & 0.0025    & $ 0.0023$   &  0.0021 & 0.0114 \\
\hline
		\end{tabular}
	\end{center}
	\label{tab1}
\end{table}

When $M=K$ and $N$ are fixed, it is shown from Figure \ref{Fig_ex1_1}(a-b) that $K_1$ and $J$ actually do not affect the final reconstruction too much. The curves in Figure \ref{Fig_ex1_1} are almost flat for different cases. However, the gaps between different cases are large, which indicates that the number of measured spectral data $K$ is of more importance than other parameters in reconstruction. Moreover, when $N$ is small, the error may increase with larger $M=K$, see Figure \ref{Fig_ex1_1}. The reason lies in the fact that small $N$ does not discriminate enough eigenvalues in reconstruction. When $N$ is large in Figure \ref{Fig_ex1_1}(b), it is clear that the error decreases with $M$.



\begin{figure}[t!]
  \begin{center}
  \begin{subfigure}{0.42\textwidth}
      \includegraphics[width=\textwidth]{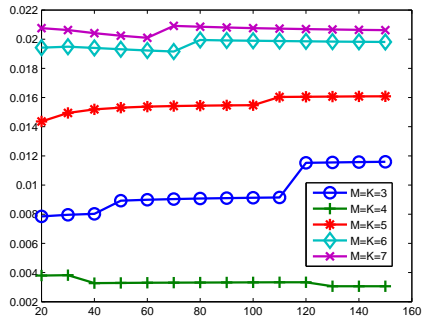}
      \caption{ $N=50$, Error with respect to $J=K_1$ }
    \end{subfigure}
    \begin{subfigure}{0.42\textwidth}
      \includegraphics[width=\textwidth]{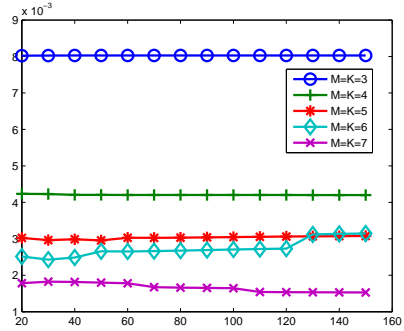}
      \caption{ $N=200$, Error with respect to $J$ }
      \end{subfigure}
    \caption{Impact of $K_1$, $J$ for fixed $M,K$ and $N$ }\label{Fig_ex1_1}
    \end{center}
  \end{figure}


Figure \ref{Fig_ex1} actually shows part of numerical inversion results for $K=8$, $K_1=J=N=150$, where the dashed lines represents the initial guess of $\alpha_M$, the orange solid line represents the exact $\alpha(x)$ and the blue solid line represents the reconstruction.
\begin{figure}[t!]
    \begin{subfigure}{0.32\textwidth}
      \includegraphics[width=\textwidth]{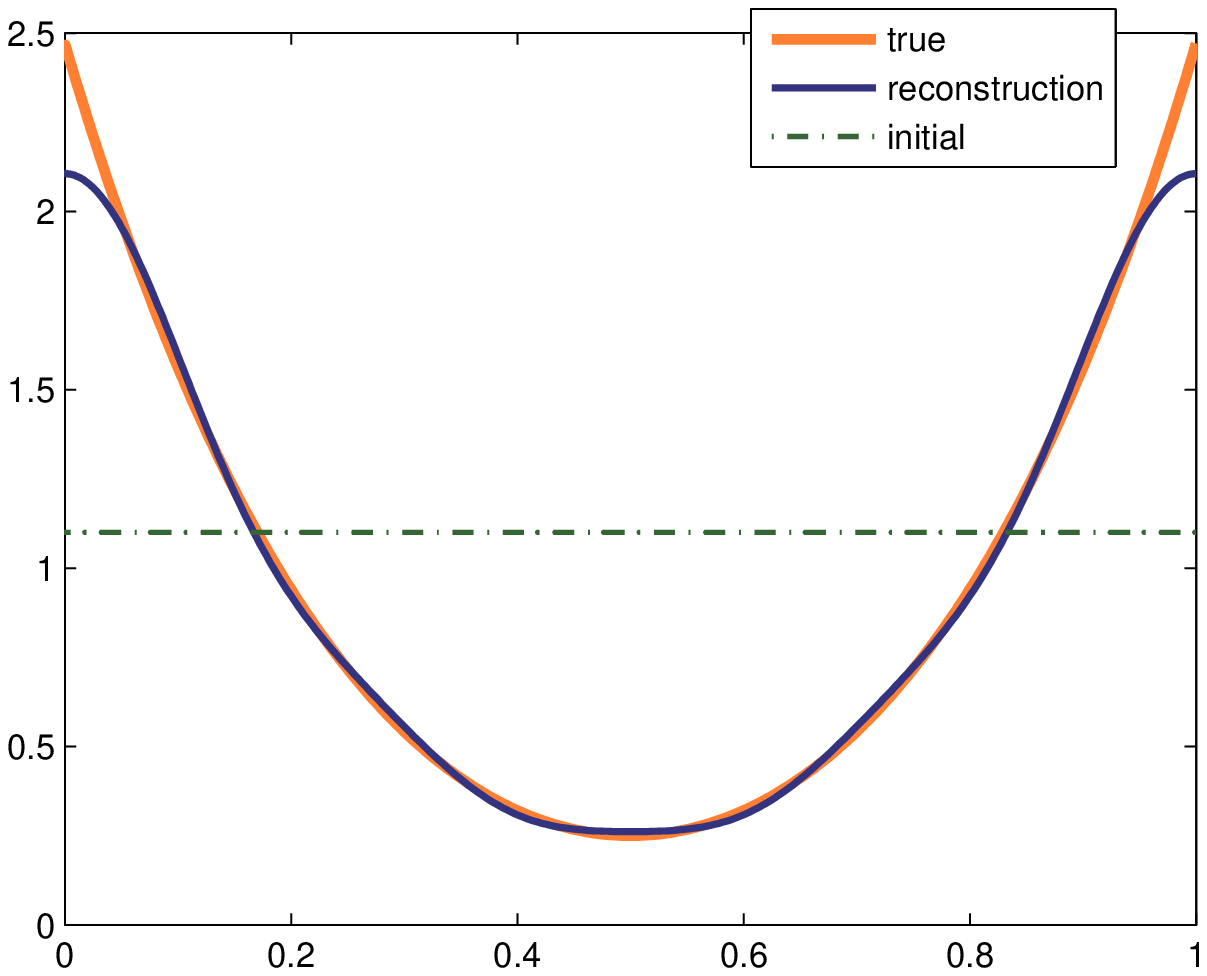}
      \caption{ M=5}
    \end{subfigure}
     \begin{subfigure}{0.32\textwidth}
      \includegraphics[width=\textwidth]{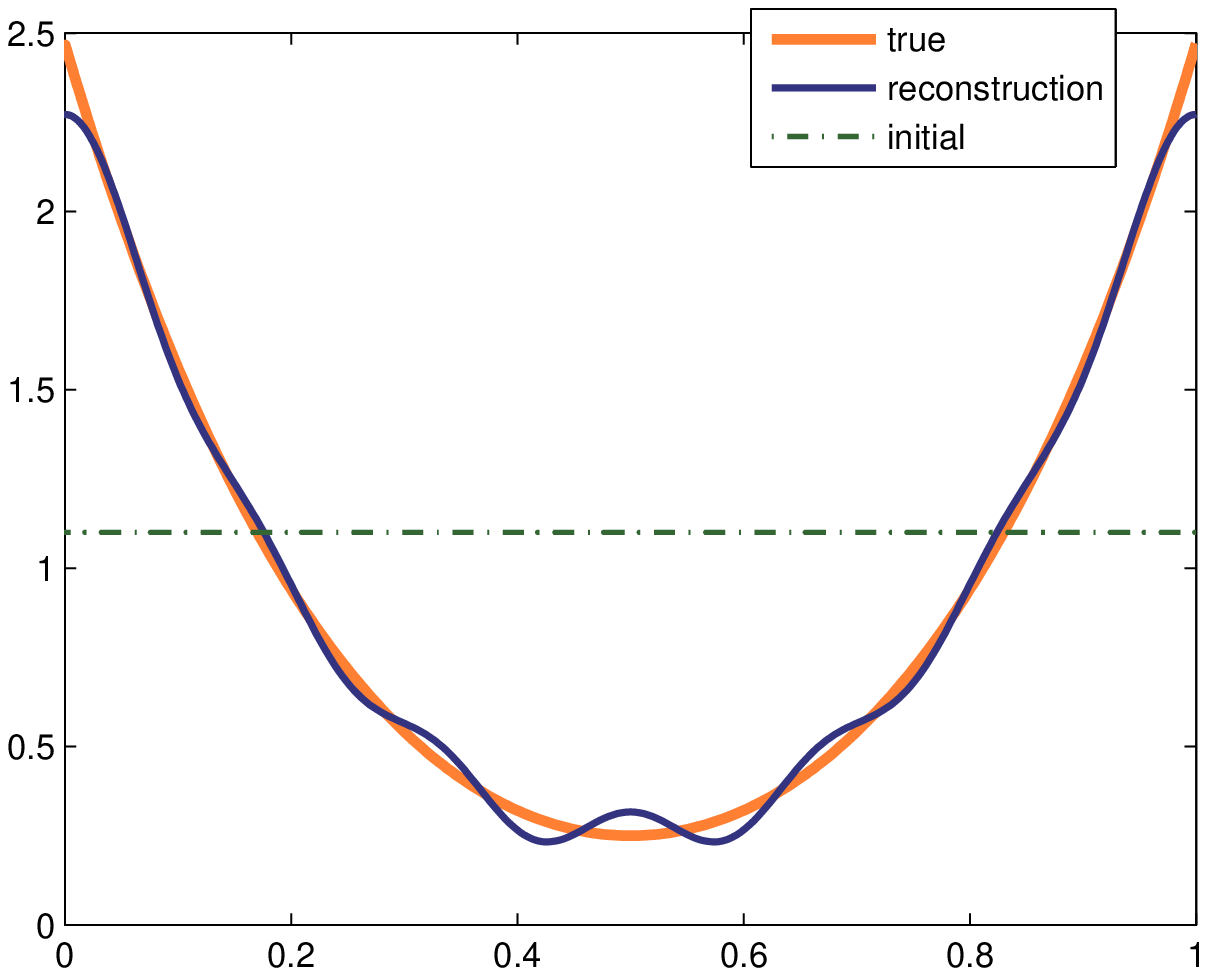}
      \caption{  M=7}
    \end{subfigure}
    \begin{subfigure}{0.32\textwidth}
      \includegraphics[width=\textwidth]{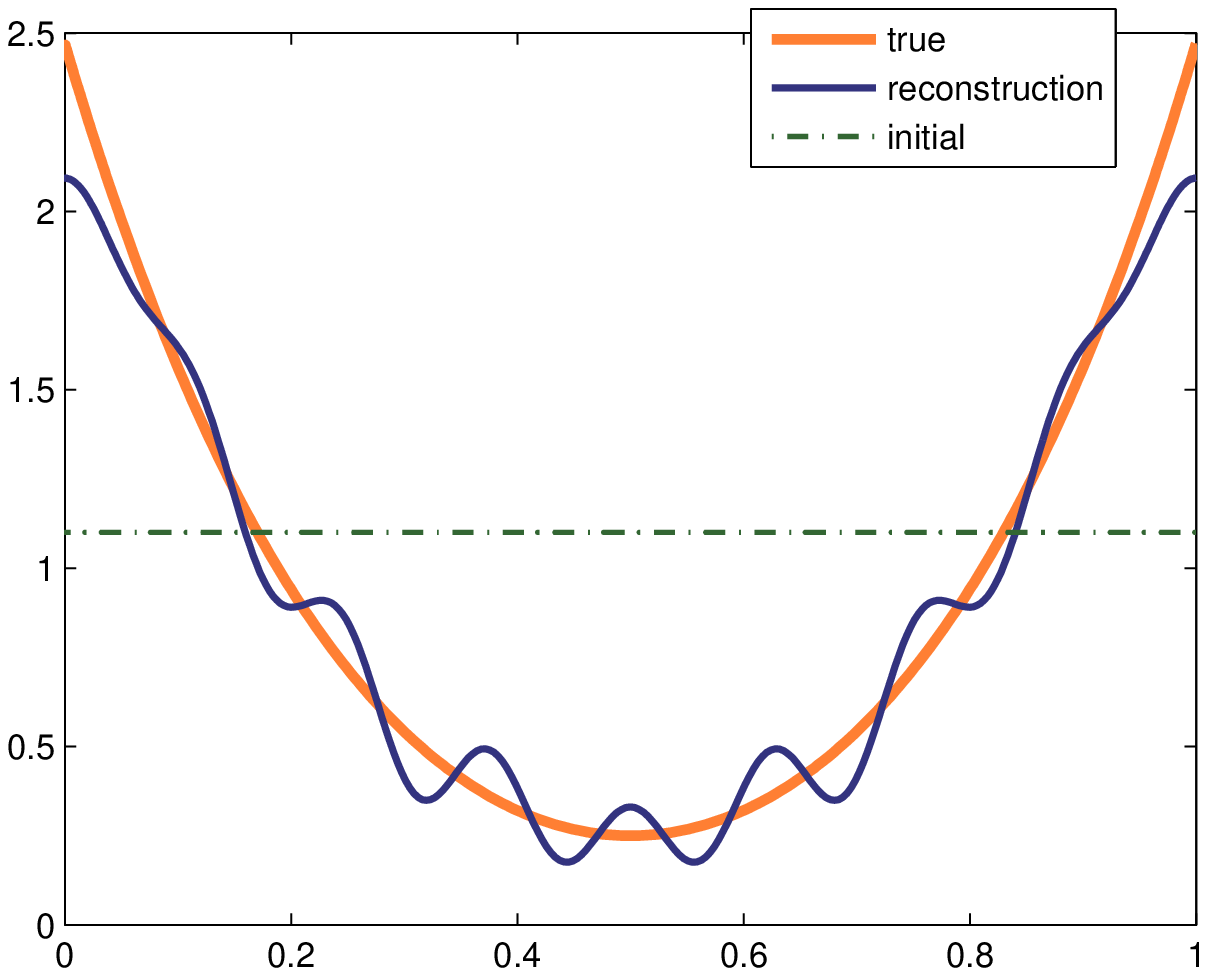}
      \caption{  M=9}
    \end{subfigure}
    \caption{Reconstruction of $\alpha_M$ in Example \ref{example1} with $K=8$, $K_1=J=N=150$. }\label{Fig_ex1}
  \end{figure}
\end{example}

\begin{example} \label{example2}
In this example, we set
\begin{align*}\alpha(x)=& 1.4062 - 0.6951\cos(2\pi x) + 0.2967\cos(4\pi x) + 0.1368\cos(6\pi x)- 0.2103\cos(8\pi x) \\
&+0.031\cos(10\pi x) + 0.153\cos(12\pi x) - 0.0718\cos(14\pi x) - 0.0512\cos(16\pi x) + 0.1258\cos(18\pi x)\\
&+ 0.04\cos(20\pi x) + 0.02\cos(22\pi x)- 0.0132\cos(24\pi x) + 0.02\cos(26\pi x) + 0.02\cos(28\pi x).
\end{align*}
Notice that this function is highly oscillatory, and thus the reconstruction needs more Fourier basis functions to see the fine structure. Therefore, in contrast to Example \ref{example1}, we need to have more eigenvalues to get an accurate reconstruction.

In the numerical experiments, we fix $K_1=J=100$ and $N=300$. To illustrate the impact of the number of ``accurate" eigenvalues $2K$ on the performance, we test three cases: $K=4$, $K=10$ and $K=50$. See Figure \ref{Fig_ex2}(a-c), (d-f) and (g-i) respectively.  One can see that for the first case $K=4$, we can only recover lower frequency information of $\alpha$. Though we can set $M>K$, i.e., Figure \ref{Fig_ex2}(a-c), the fine structure can not be recovered as not sufficient information is given. For the similar reason of $K=10$, the reconstruction for $M=12$ and $M=8$ are both worse than for $M=10$, see Figure \ref{Fig_ex2}(d-f). However, for $K=50$, the reconstruction for $M=12$ is better than for $M=10$ and $M=8$, which indicates more ``accurate" measured eigenvalues give a better reconstruction.
\begin{figure}[t!]
    \begin{subfigure}{0.32\textwidth}
       \includegraphics[width=\textwidth]{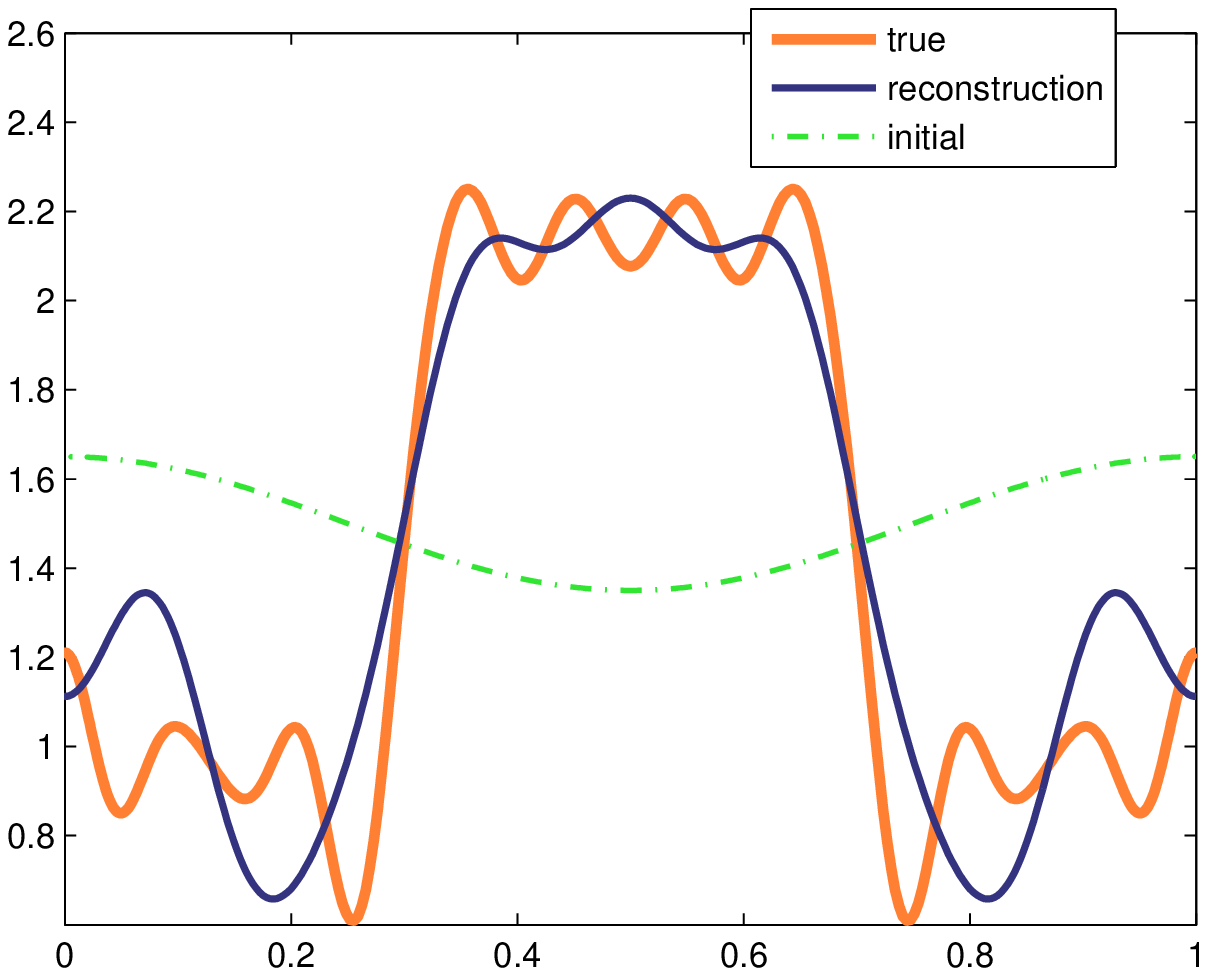}
      \caption{ $K=4$, $M=8$  }
    \end{subfigure}
    \begin{subfigure}{0.32\textwidth}
      \includegraphics[width=\textwidth]{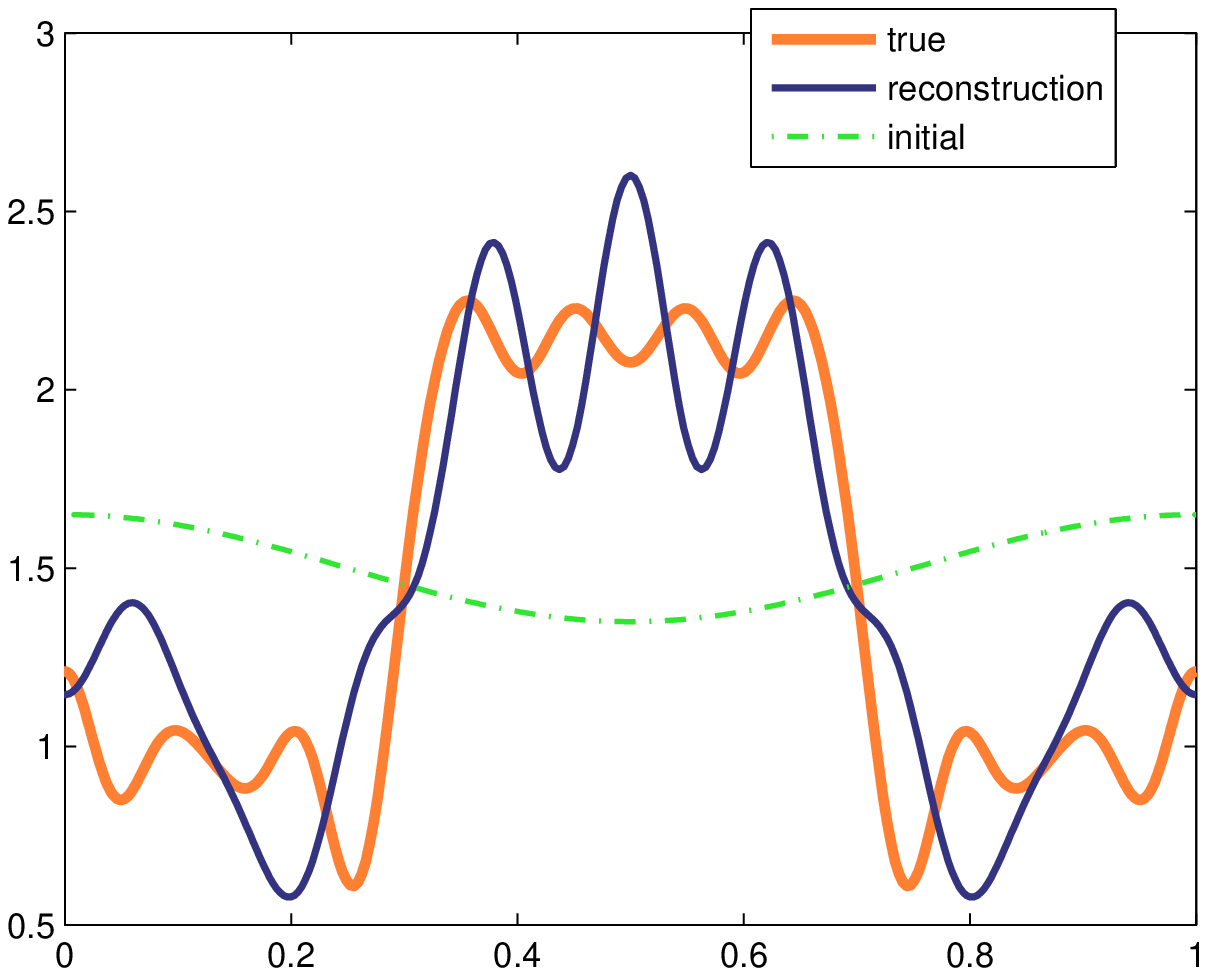}
      \caption{ $K=4$, $M=10$ }
    \end{subfigure}
    \begin{subfigure}{0.32\textwidth}
      \includegraphics[width=\textwidth]{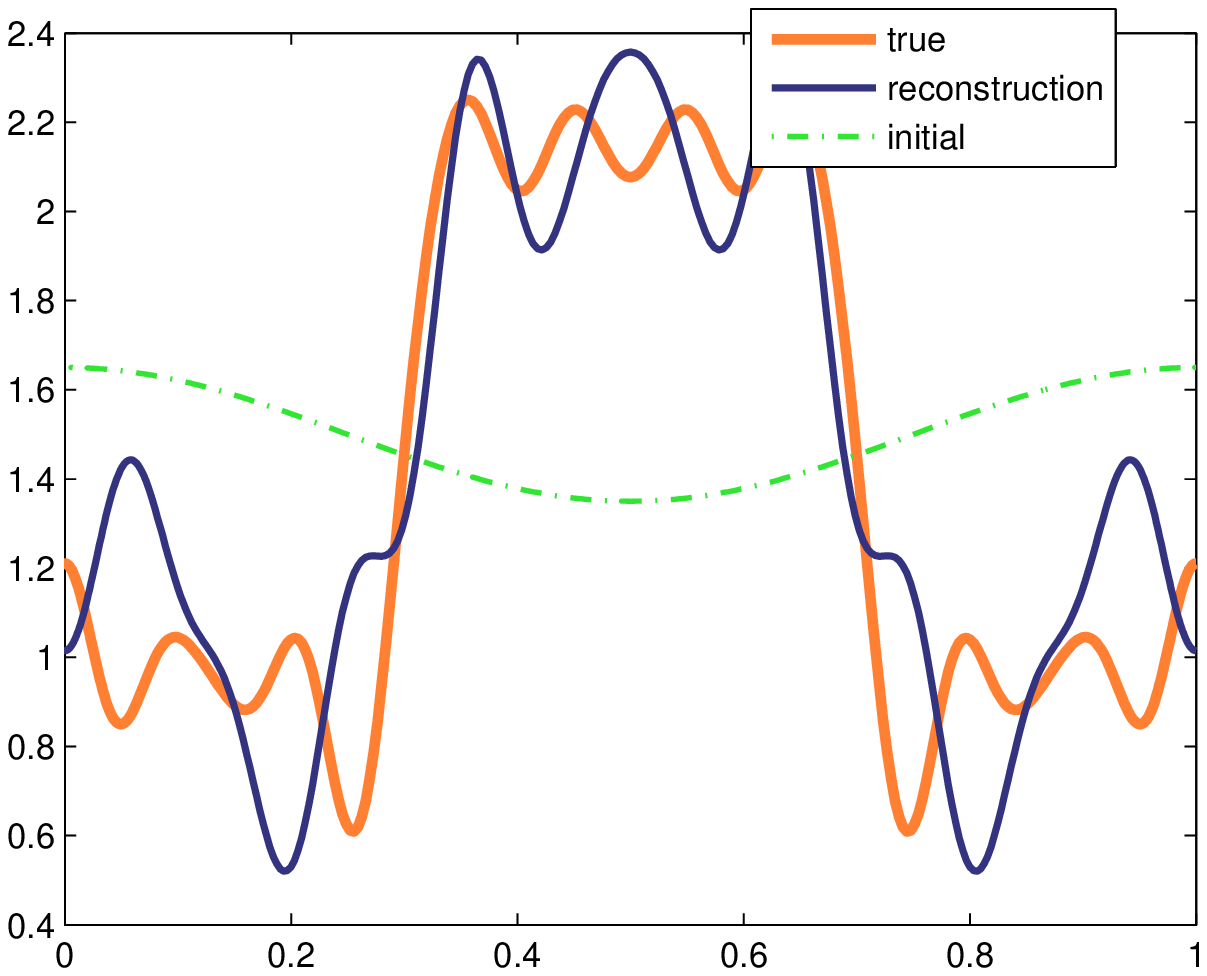}
      \caption{ $K=4$, $M=12$}
    \end{subfigure}\\
        \begin{subfigure}{0.32\textwidth}
      \includegraphics[width=\textwidth]{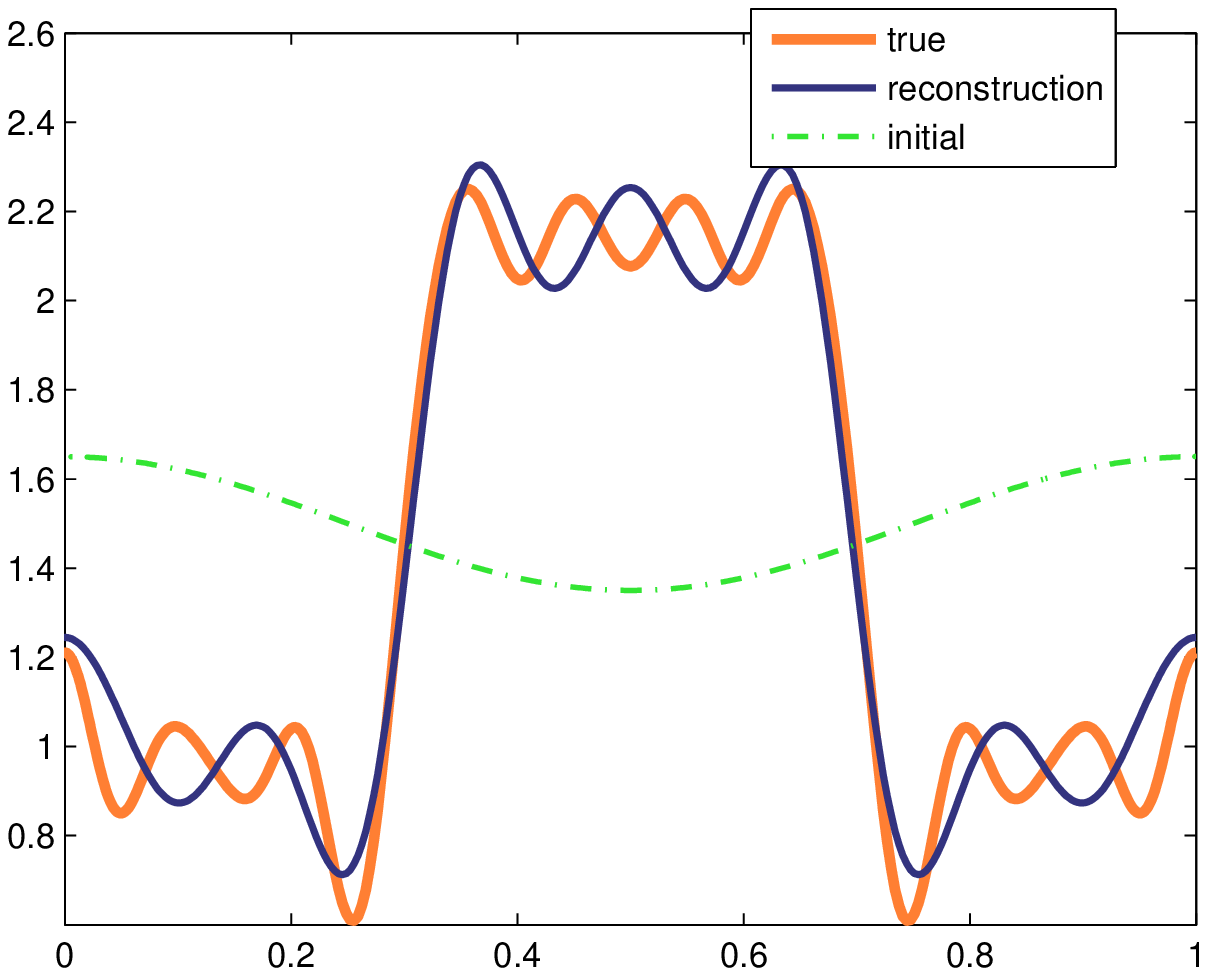}
      \caption{ $K=10$, $M=8$  }
    \end{subfigure}
    \begin{subfigure}{0.32\textwidth}
      \includegraphics[width=\textwidth]{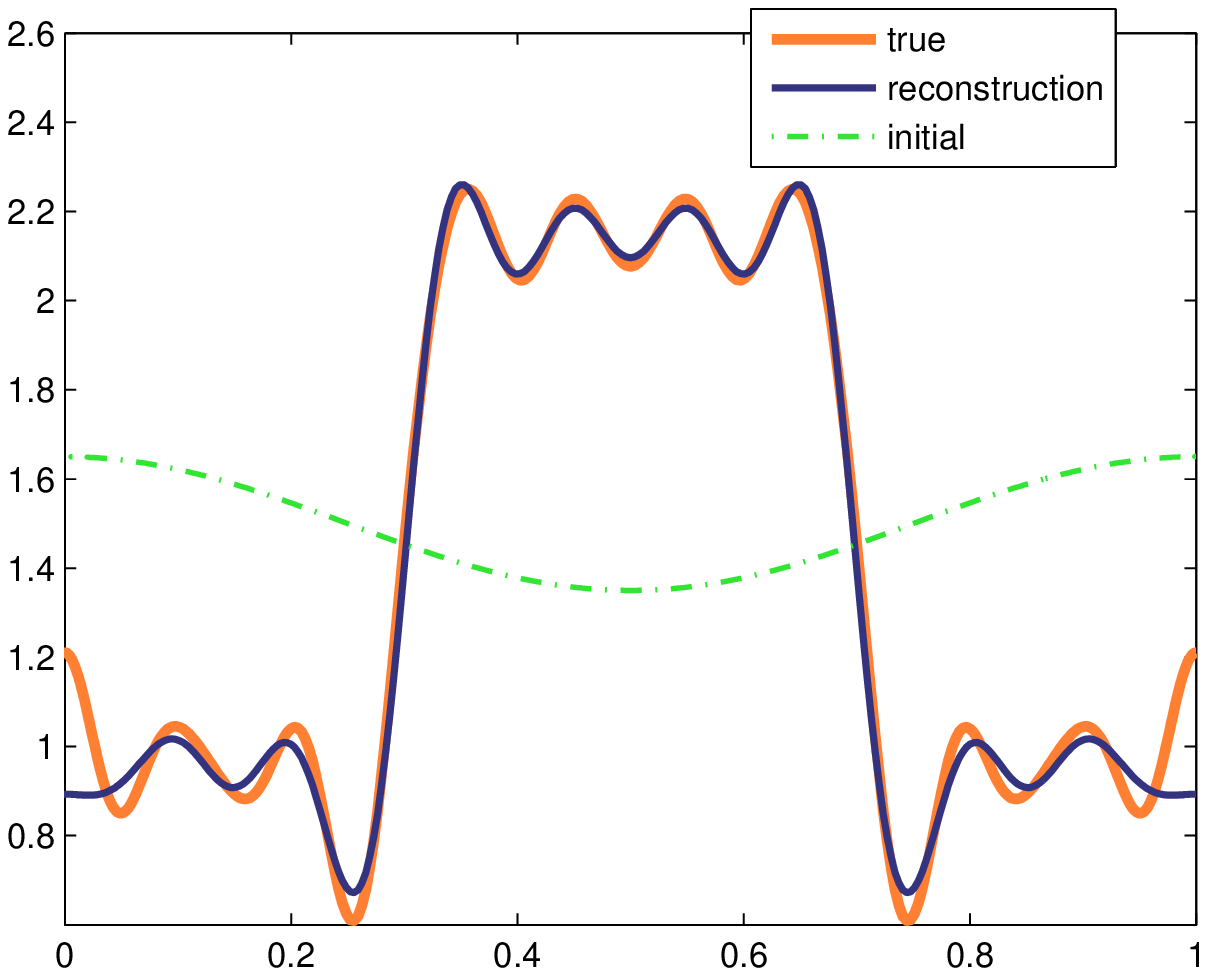}
      \caption{ $K=10$, $M=10$ }
    \end{subfigure}
    \begin{subfigure}{0.32\textwidth}
      \includegraphics[width=\textwidth]{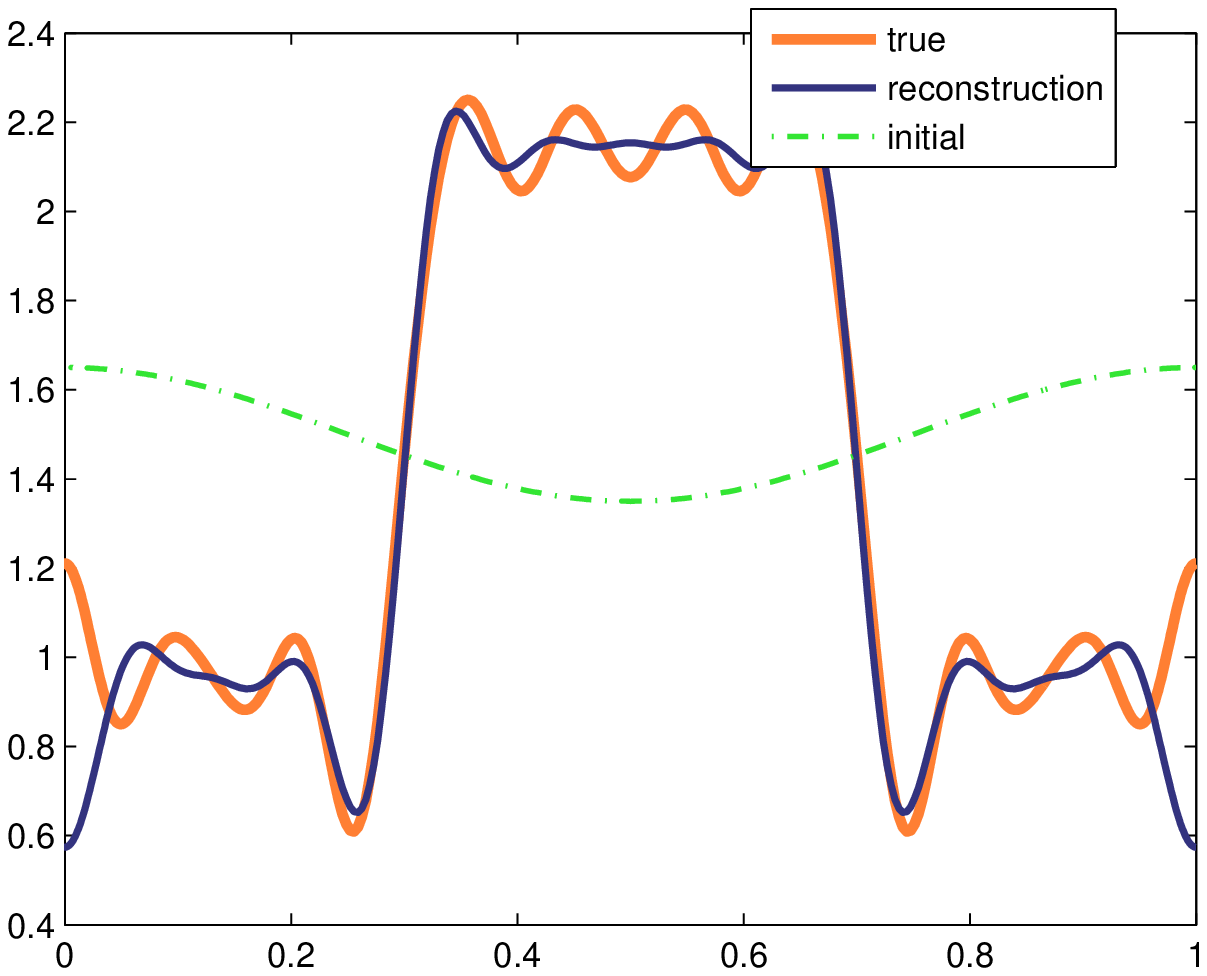}
      \caption{ $K=10$, $M=12$}
    \end{subfigure}\\
    \begin{subfigure}{0.32\textwidth}
      \includegraphics[width=\textwidth]{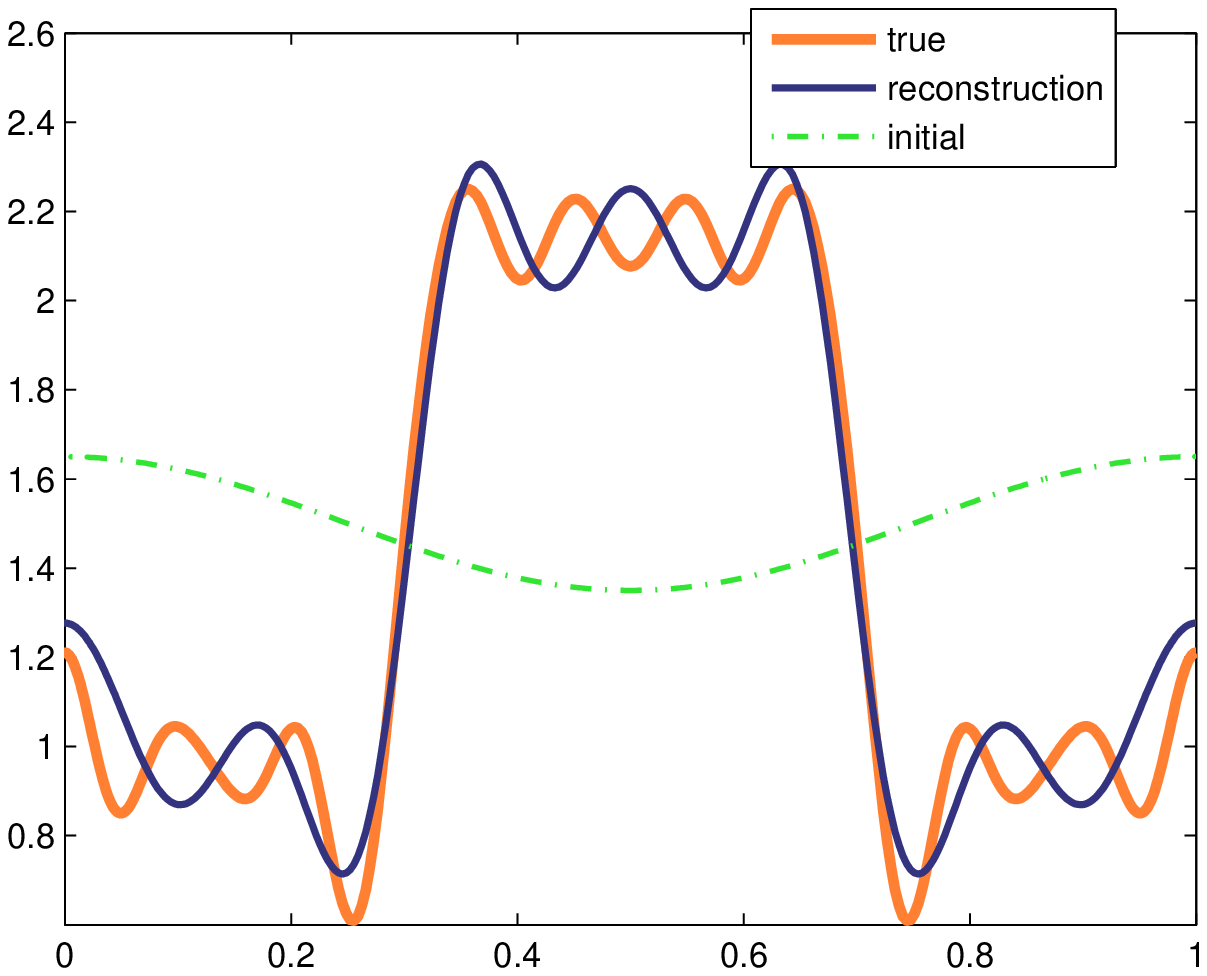}
      \caption{ $K=50$, $M=8$}
    \end{subfigure}
    \begin{subfigure}{0.32\textwidth}
      \includegraphics[width=\textwidth]{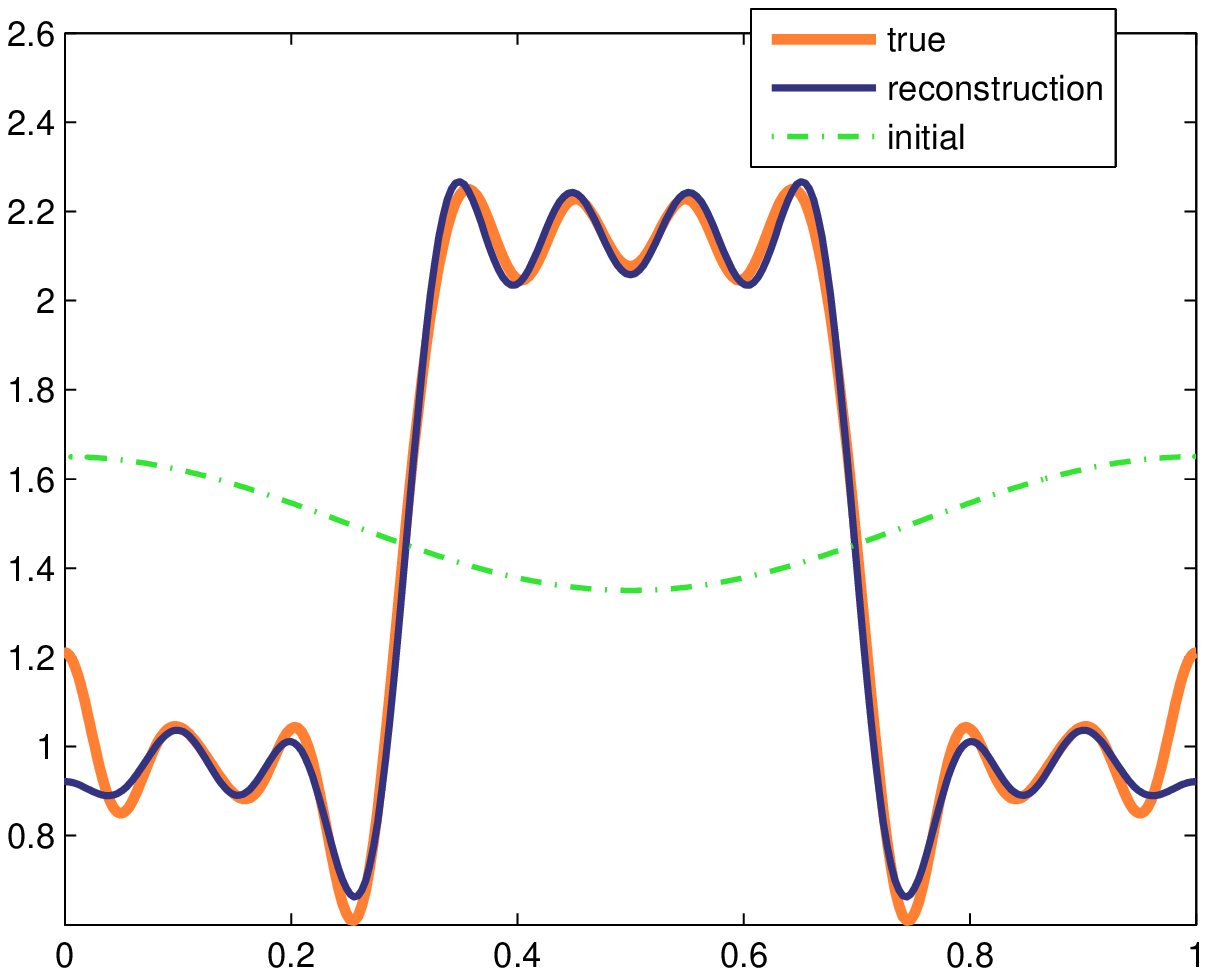}
      \caption{ $K=50$, $M=10$}
    \end{subfigure}
    \begin{subfigure}{0.32\textwidth}
      \includegraphics[width=\textwidth]{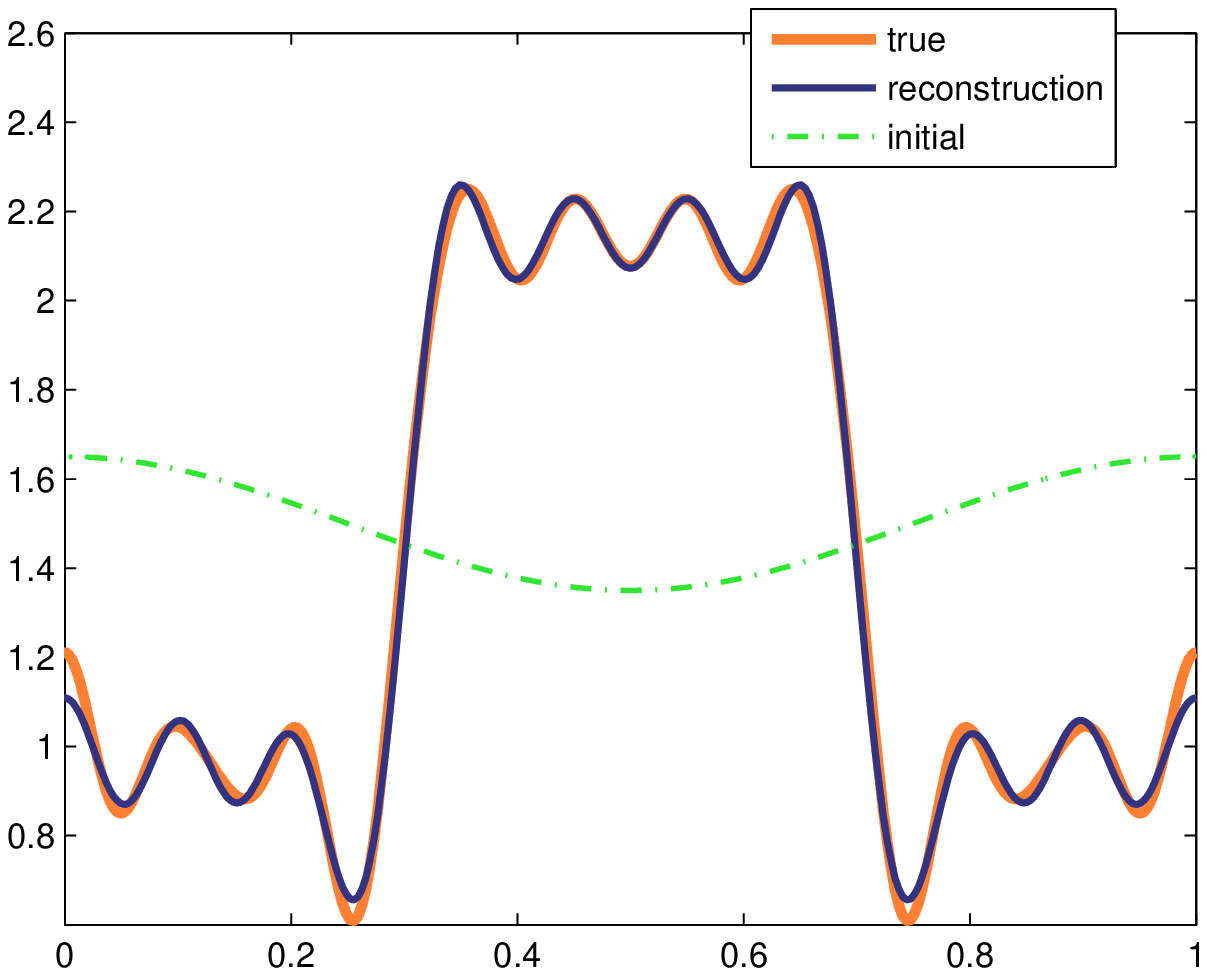}
      \caption{ $K=50$, $M=12$}
    \end{subfigure}
    \caption{Reconstruction of $\alpha_M$ in Example \ref{example2} with  $K_1=J=100$ and $N=300$. }\label{Fig_ex2}
  \end{figure}

\end{example}

\begin{example} \label{example3}
In this example, we show a non-smooth damping coefficient reconstruction. Here we set
\begin{align*}
\alpha(x) = \left\{
         \begin{array}{ll}
           2, & x\in [0,0.3] \\
           3, & x\in (0.3,0.7) \\
           2, & x\in [0.7,1]
         \end{array}
       \right.
\end{align*}
The non-smoothness inevitably results in more difficulties for reconstruction. In order to capture the discontinuity, we actually need quite a lot modes in Fourier expansion. However, on the other hand, the number $M$ needs to be chosen as a regularization parameter. The results are shown in Figure \ref{Fig_ex3}.
\begin{figure}[t!]
\begin{subfigure}{0.32\textwidth}
      \includegraphics[width=\textwidth]{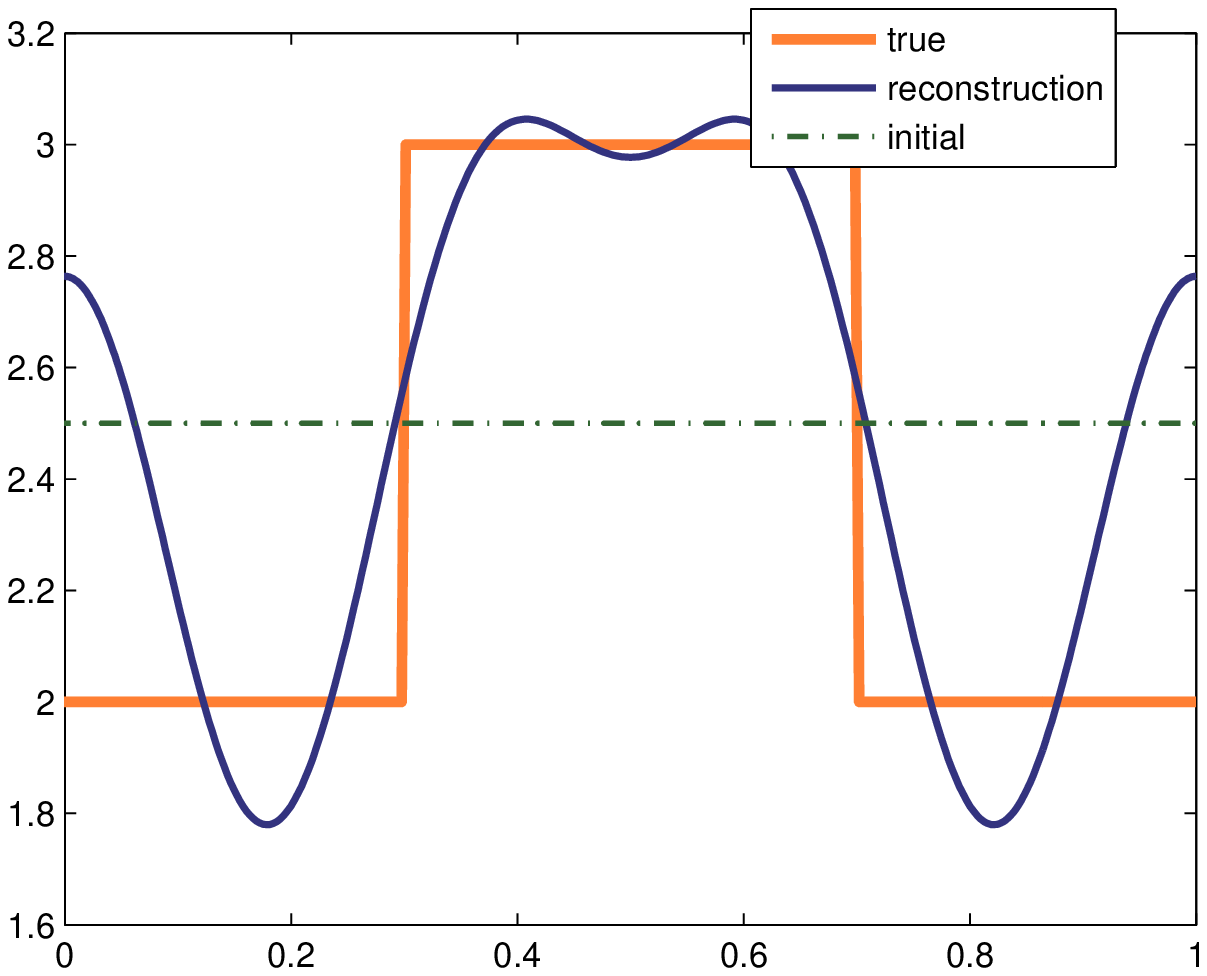}
      \caption{ M=4}
    \end{subfigure}
    \begin{subfigure}{0.32\textwidth}
      \includegraphics[width=\textwidth]{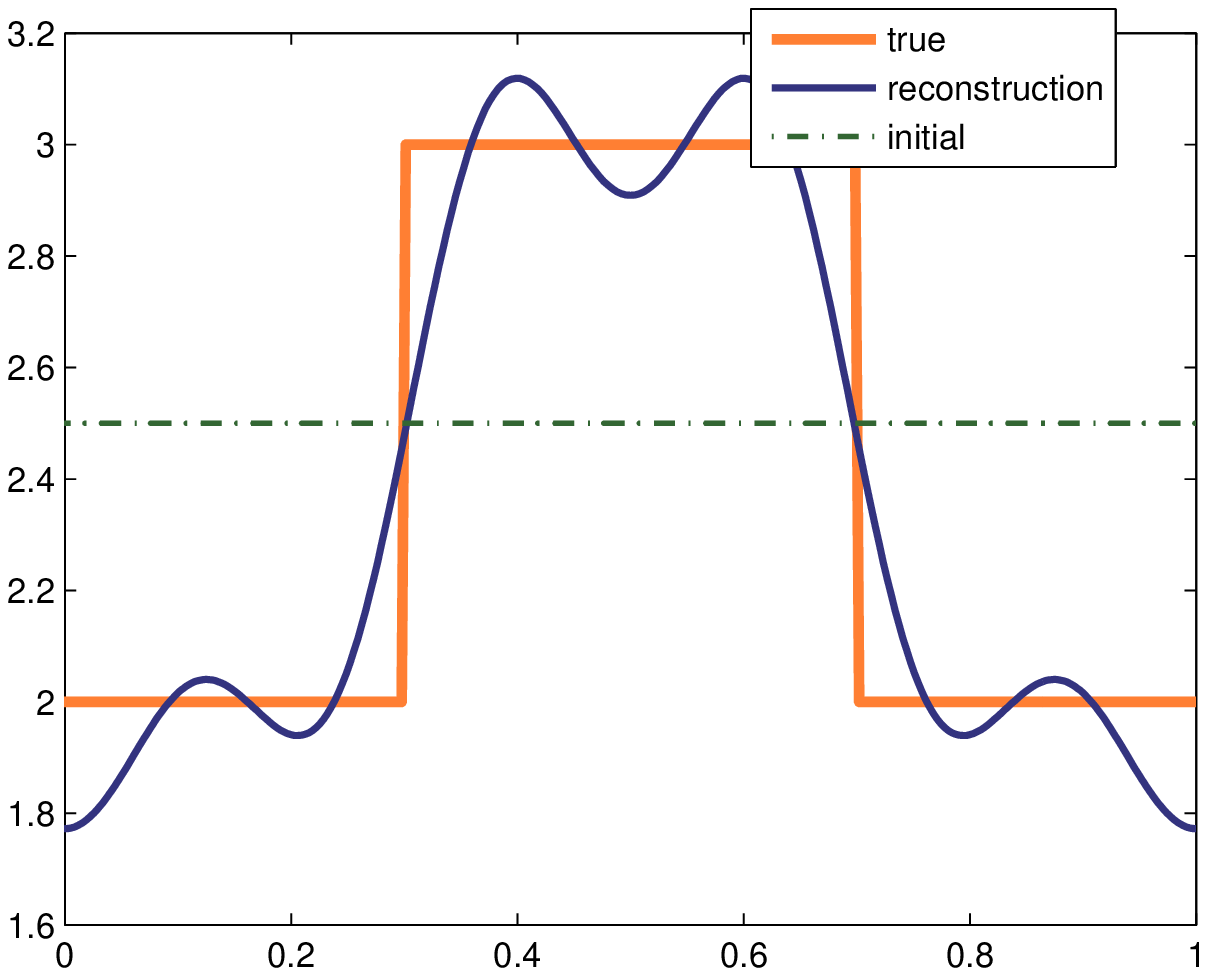}
      \caption{ M=5}
    \end{subfigure}
    \begin{subfigure}{0.32\textwidth}
      \includegraphics[width=\textwidth]{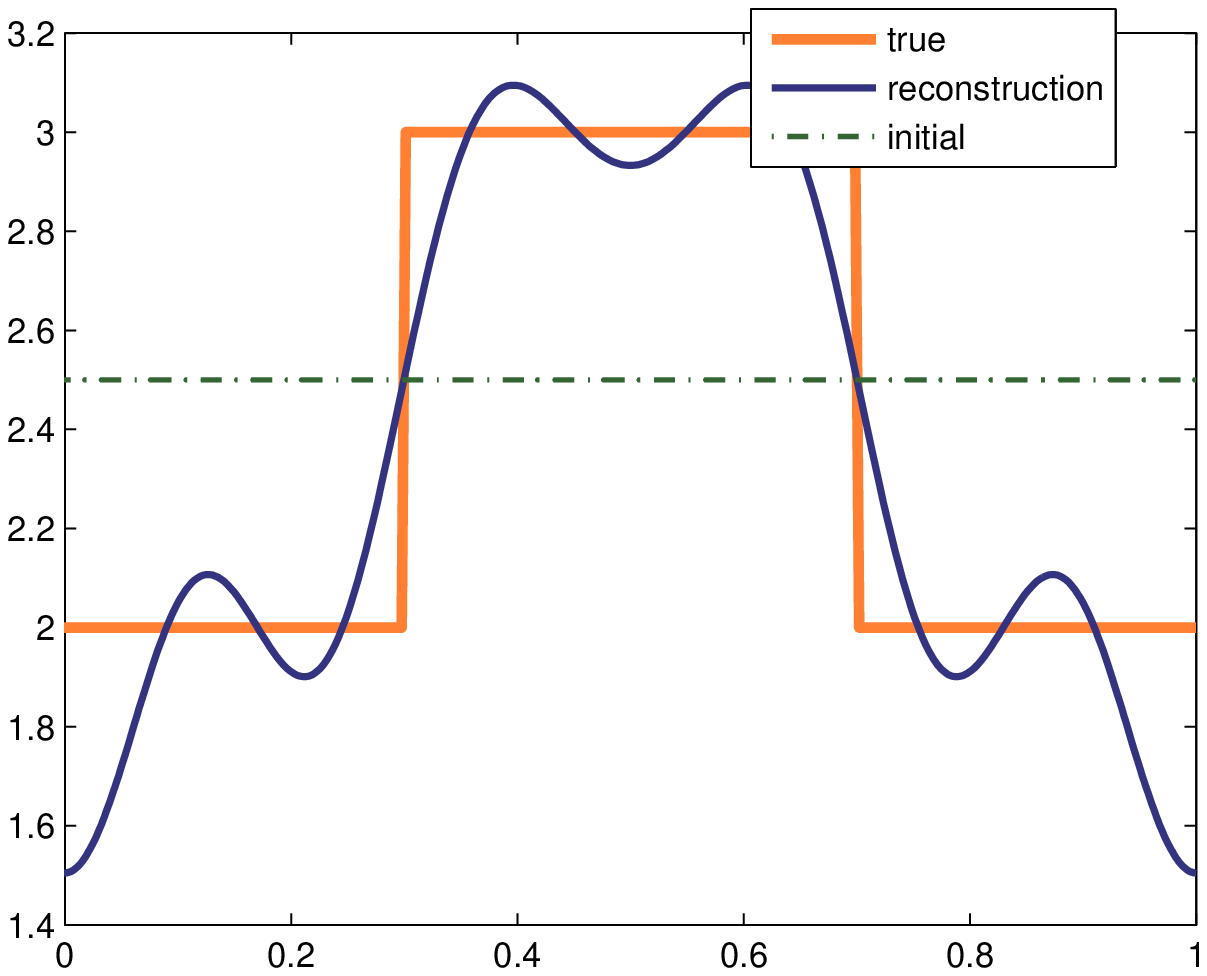}
      \caption{ M=6}
    \end{subfigure}\\
     \begin{subfigure}{0.32\textwidth}
      \includegraphics[width=\textwidth]{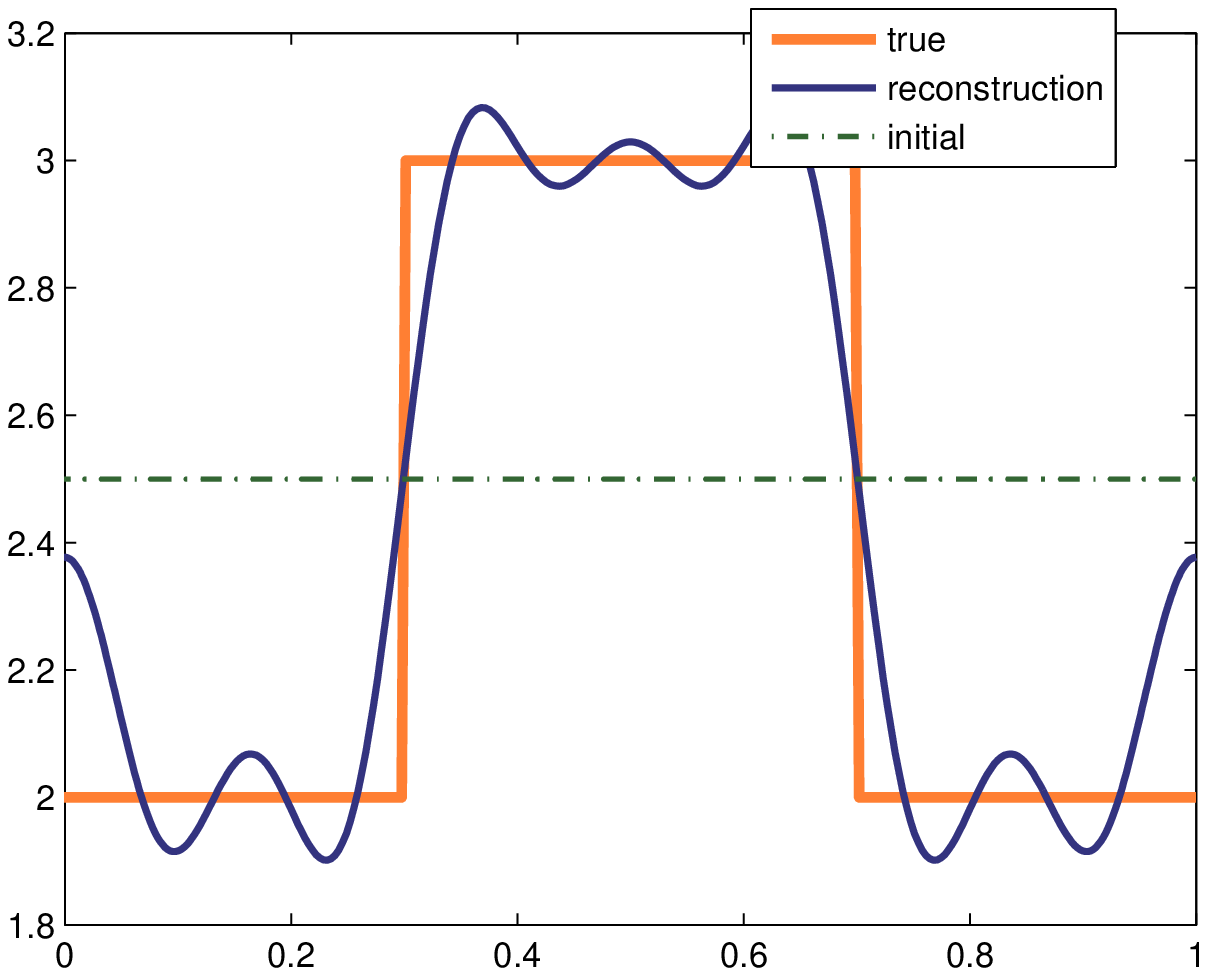}
      \caption{  M=7}
    \end{subfigure}
         \begin{subfigure}{0.32\textwidth}
      \includegraphics[width=\textwidth]{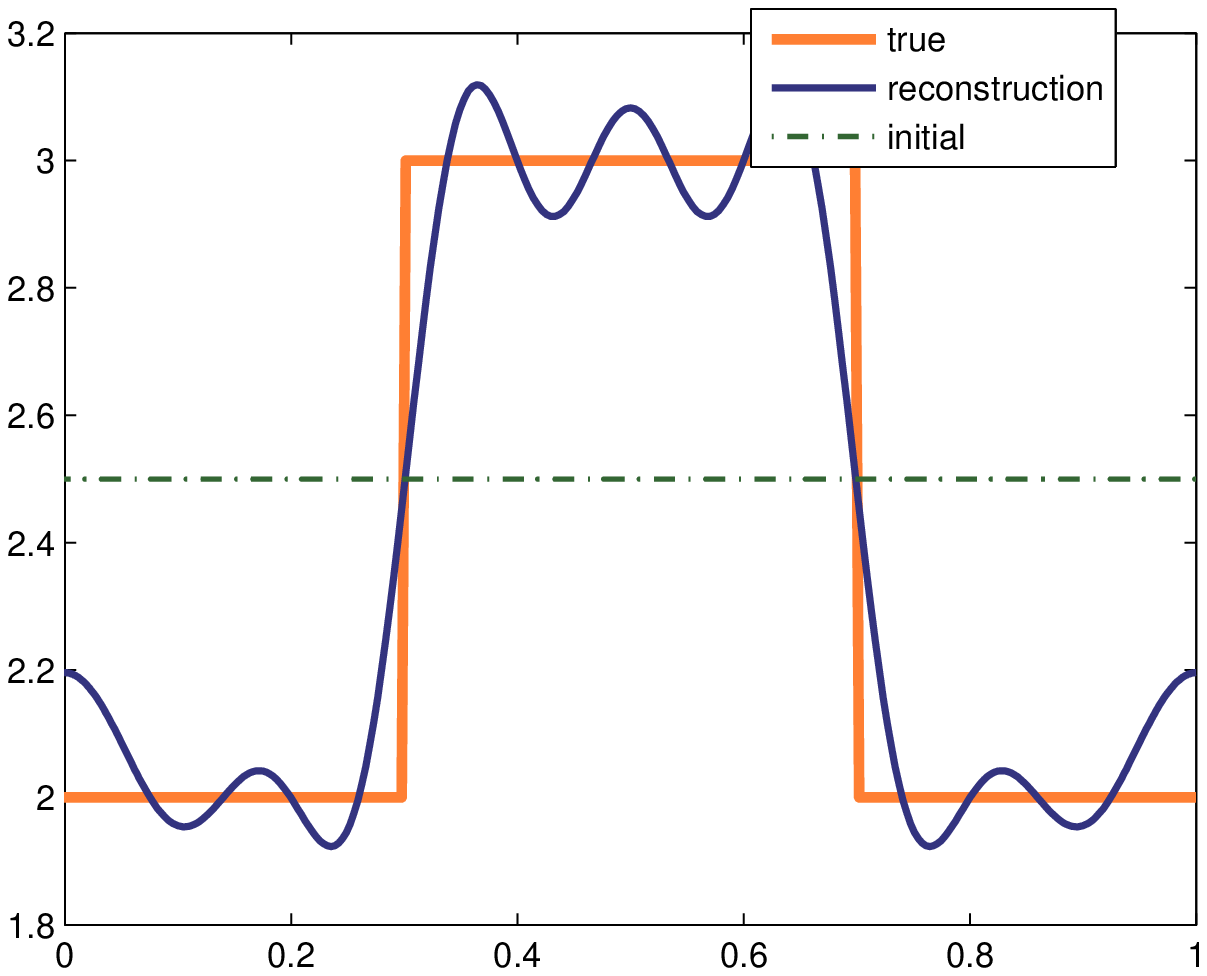}
      \caption{  M=8}
    \end{subfigure}
    \begin{subfigure}{0.32\textwidth}
      \includegraphics[width=\textwidth]{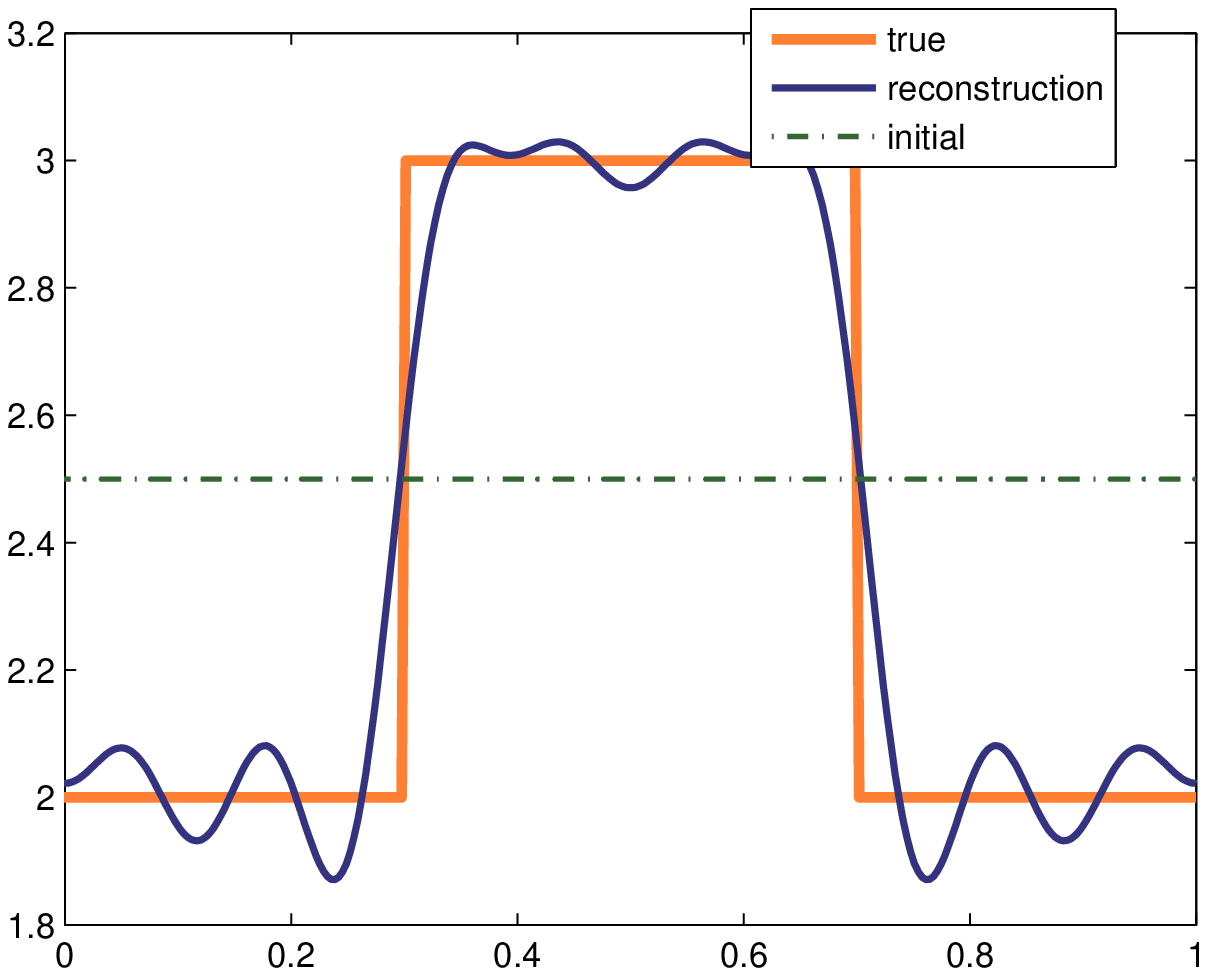}
      \caption{  M=9}
    \end{subfigure}
    \caption{Reconstruction of $\alpha_M$ in Example \ref{example3} with $K=10$, $K_1=J=100$ and $N=100$. }\label{Fig_ex3}
  \end{figure}
\end{example}
\begin{example}\label{example4}

 In this example, we test the stability of the algorithm with noisy data. Suppose the spectral data is polluted by random noise
\begin{equation*}
\lambda_j^{\delta} = \lambda_j + \delta\times\mbox{rand}(0,1)\times(1+i),\;\;j=1,2,\cdots
\end{equation*}
where $\delta$ is noise level and rand(0,1) represents the standard
    uniform distribution on the open interval (0,1). Moreover, we set the damping coefficient $\alpha(x)$ as follows
\begin{align*}
\alpha(x) = 1.5 + 0.2\cos(2\pi x) + 0.1\cos(4\pi x) - 0.04\cos(6\pi x) + 0.03\cos(8\pi x)
\end{align*}
As we know that both the noisy spectral data and the finite truncated series of eigenvalues $M$ result in approximation error in trace formulas. Hence the reconstruction of the damping coefficient is definitely influenced by these two parameters. Figure \ref{Fig_ex4}(a-c) shows numerical inversion results when $\delta=0.1\%, 0.5\%, 1\%$, respectively. It is clear that for $\delta=0.1\%$, when $M$ increases from 3 to 6, the reconstruction becomes better and better. However, when $\delta=1\%$ and $M$ increases, the reconstruction becomes better first and then worse, hence the optimal choice of $M$ is $M=4$ in Figure \ref{Fig_ex4}(c). We believe that if we utilize clean spectral data, i.e., $\delta=0$, the optimal $M$ should be larger.
\begin{figure}[t!]
    \begin{subfigure}{0.32\textwidth}
      \includegraphics[width=\textwidth]{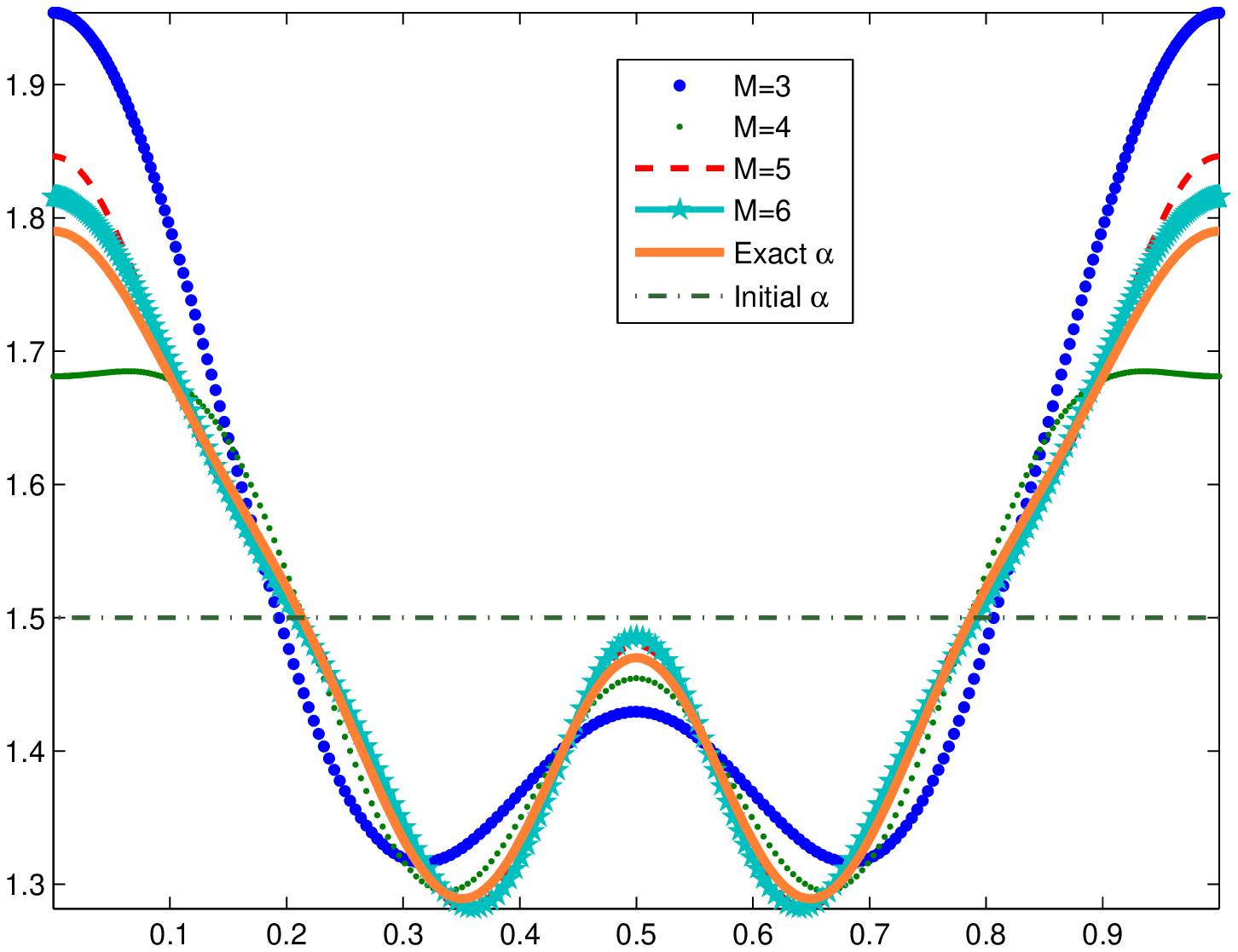}
      \caption{ $\delta$=0.1\%}
    \end{subfigure}
     \begin{subfigure}{0.32\textwidth}
      \includegraphics[width=\textwidth]{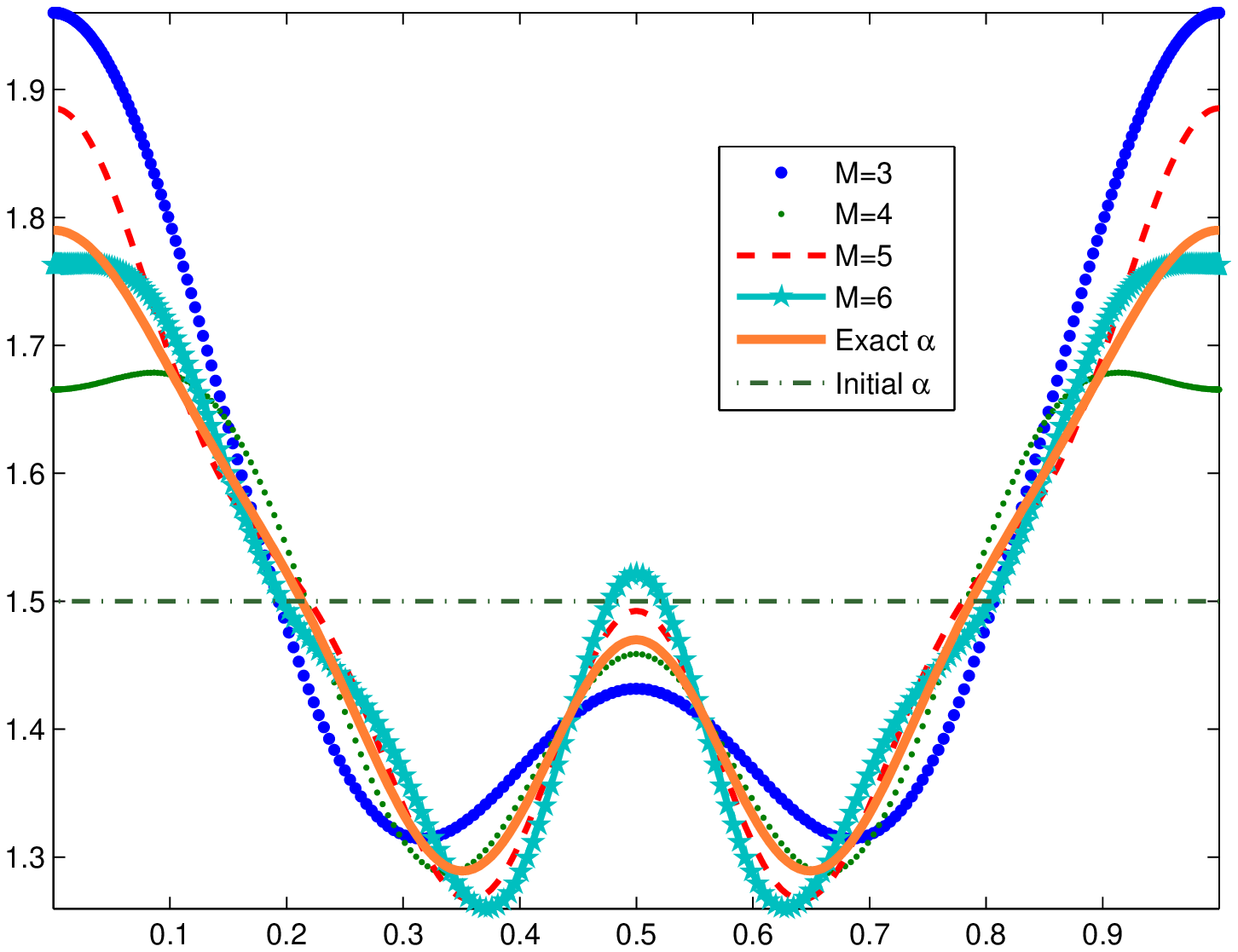}
      \caption{  $\delta$ = 0.5\%}
    \end{subfigure}
    \begin{subfigure}{0.32\textwidth}
      \includegraphics[width=\textwidth]{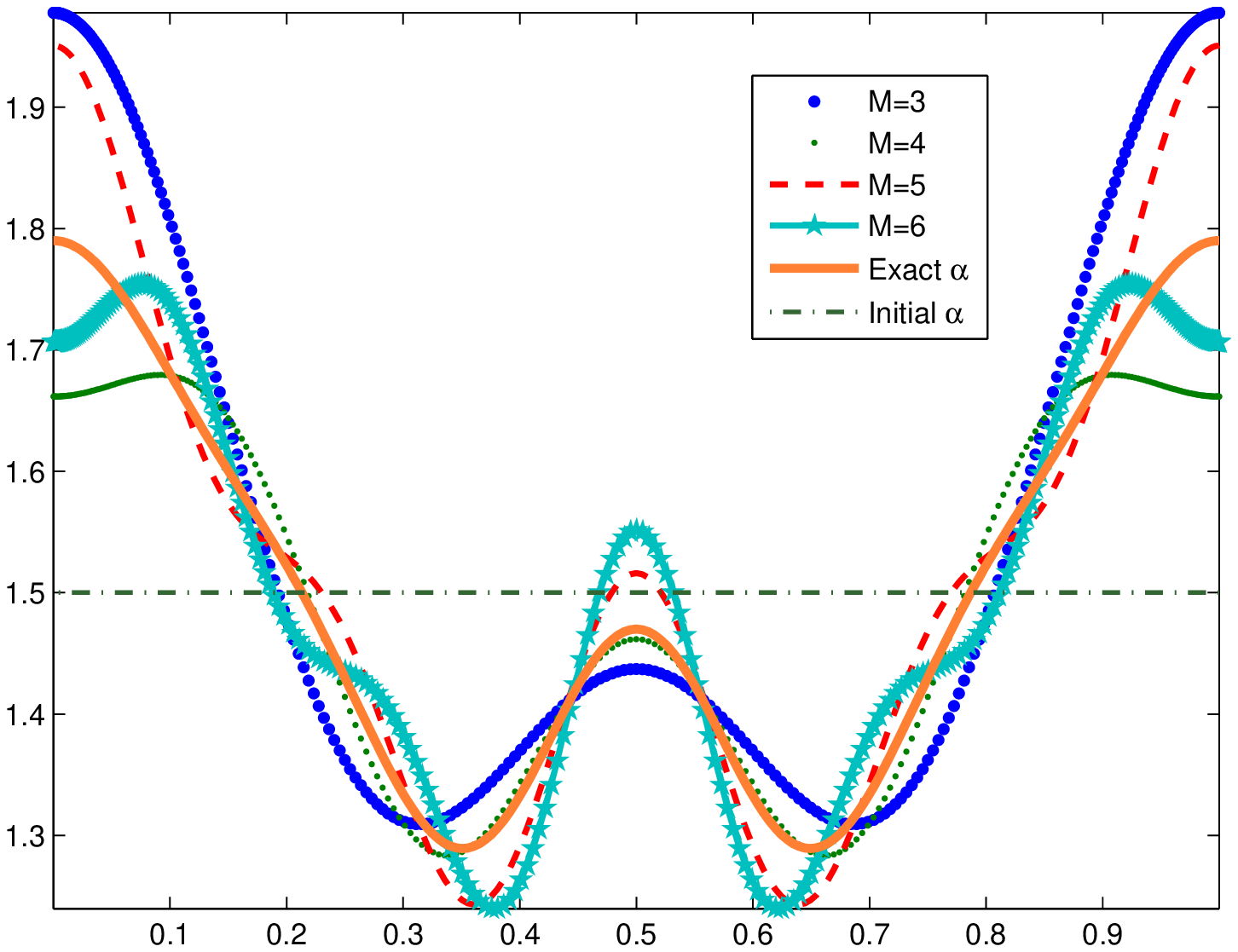}
      \caption{  $\delta$ = 1\%}
    \end{subfigure}
    \caption{Reconstruction of $\alpha_M$ in Example \ref{example4} with $K=M$, $K_1=J=75$ and $N=75$. }\label{Fig_ex4}
  \end{figure}

\end{example}

\begin{example}\label{example5}
In this example, we reconstruct a large damping coefficient. We set 
\begin{align*}
\begin{split}
\alpha_{true}(x)=&\pi\Big(1.5567+1.4896\cos2\pi x+0.3\cos4\pi x+0.1\cos6\pi x \\
&\quad +0.2\cos8\pi x+0.2\cos 10\pi x + 0.2\cos12\pi x\Big).
\end{split}
\end{align*}
which can be viewed as a perturbation of \eqref{example_larged}. According to the discussion for previous examples, we choose $K=M$, $K_1=J=75$. We remark here that the parameter $N$, the highest degree of the polynomials used in the algorithm, can not be large. The underlying reason lies in the behaviors of the chosen polynomials. From Figure \ref{eigen_distribution_extreme}, we see that the reciprocal of some eigenvalues, i.e., $z = \lambda^{-1}$ are not close to the circle $B_{(-\frac{1}{\alpha_0},0)}(\frac{1}{\alpha_0})$, and $T_n(z) = z(\alpha_0z +1)^{n-1}$ changes rapidly away from the circle when $n$ is large. Since the limited number of polynomials can not discriminate enough eigenvalues, the number of basis functions $M$ can not be large either. Also, for large damping term, the convergence of the algorithm is very sensitive to the initial guess. However, one can adopt a multi-step optimization scheme as mentioned in \cite{XZ1}: starting with small $M$ and use the reconstructed profile as the initial guess for the reconstruction with a slightly larger $M$, and so forth.

The results of numerical experiments are shown in Figure \ref{Fig_ex5}. We test for different $M$ and $N$. Since the true damping has 7 modes, it is clearly that the reconstruction for $M=K=7$ is better than $M=K=6$ and $M=K=8$ for the same $N$.
%

\begin{figure}[t!]
    \begin{subfigure}{0.32\textwidth}
      \includegraphics[width=\textwidth]{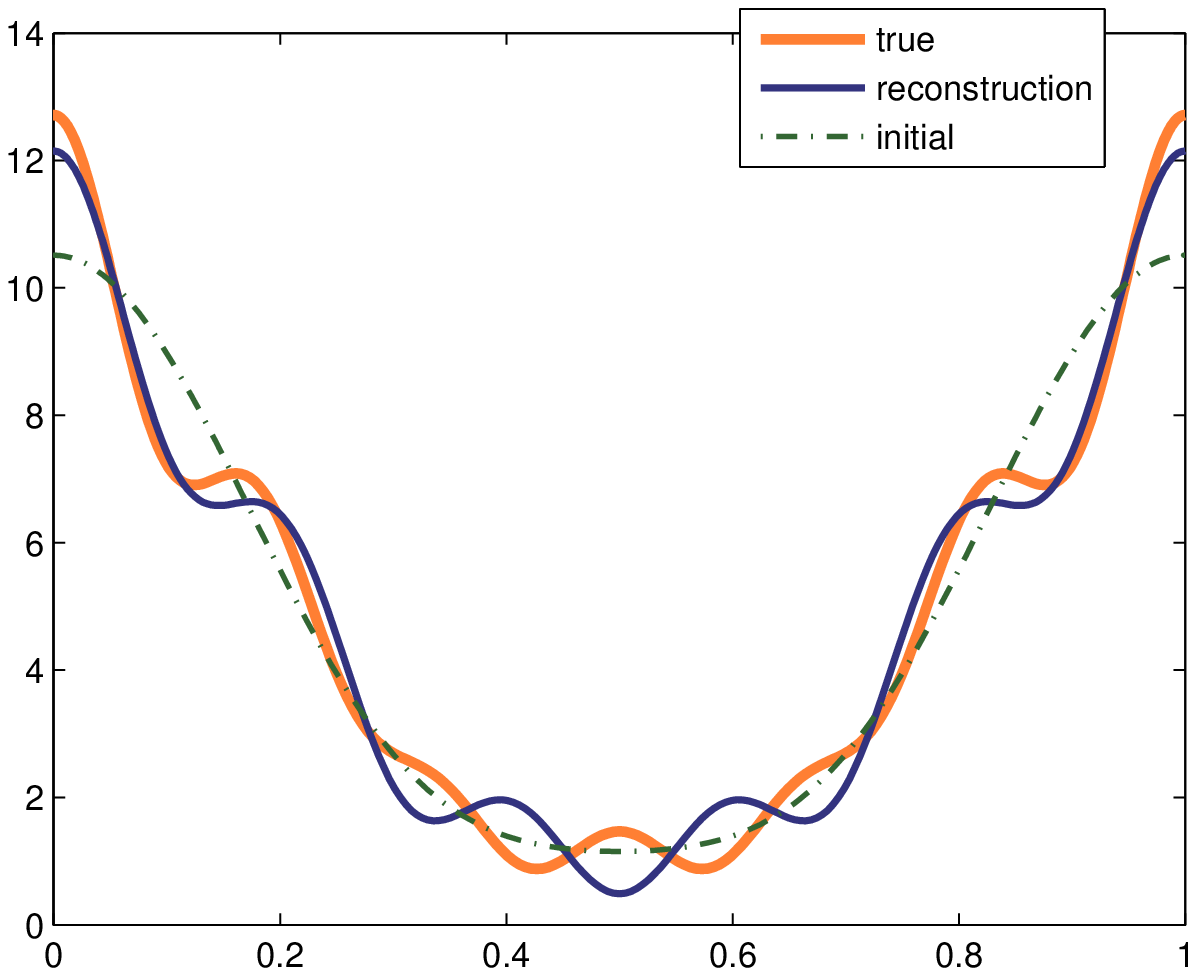}
      \caption{ K=6,N=25}
    \end{subfigure}
     \begin{subfigure}{0.32\textwidth}
      \includegraphics[width=\textwidth]{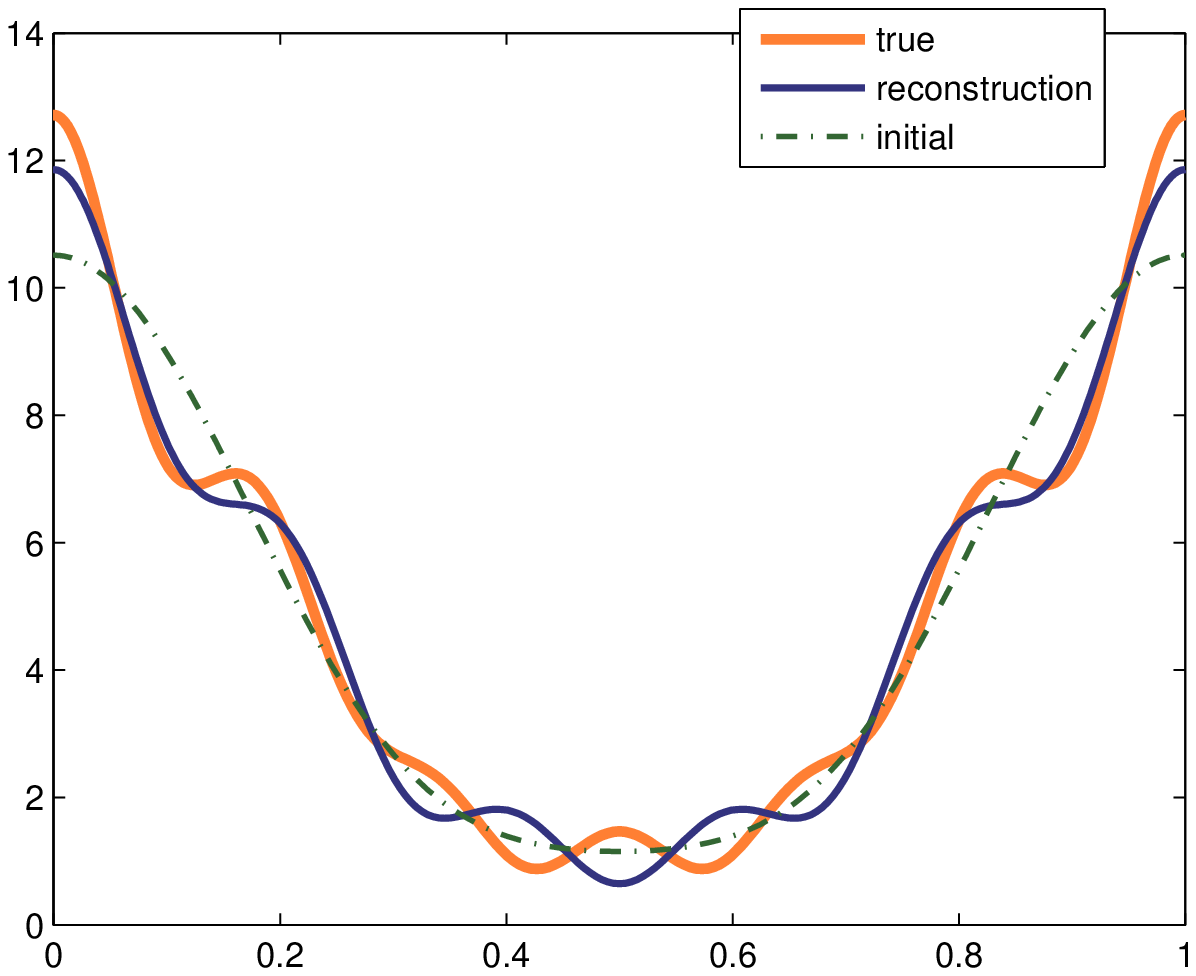}
      \caption{ K=6, N=30}
    \end{subfigure}
    \begin{subfigure}{0.32\textwidth}
      \includegraphics[width=\textwidth]{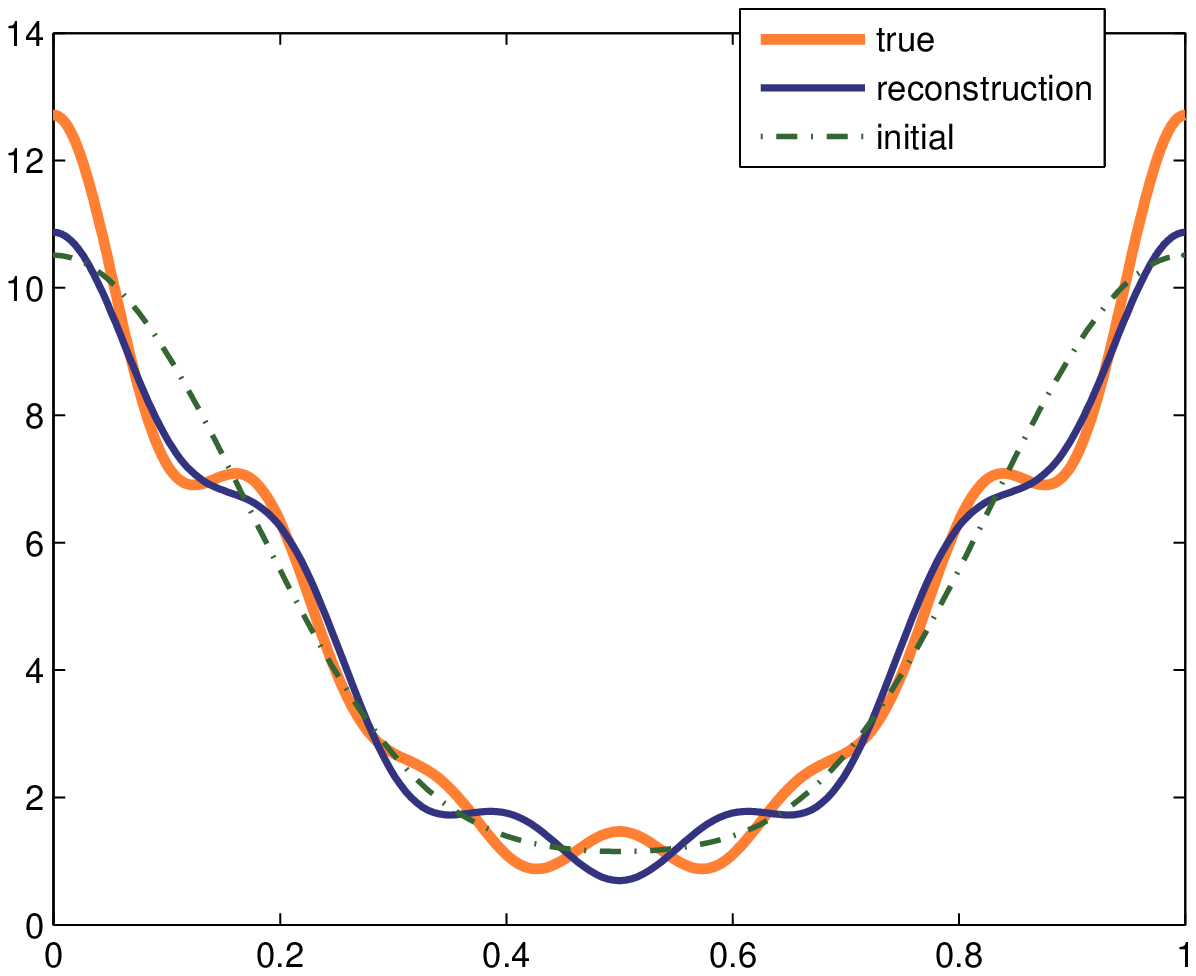}
      \caption{ K=6, N=35}
    \end{subfigure}\\
        \begin{subfigure}{0.32\textwidth}
      \includegraphics[width=\textwidth]{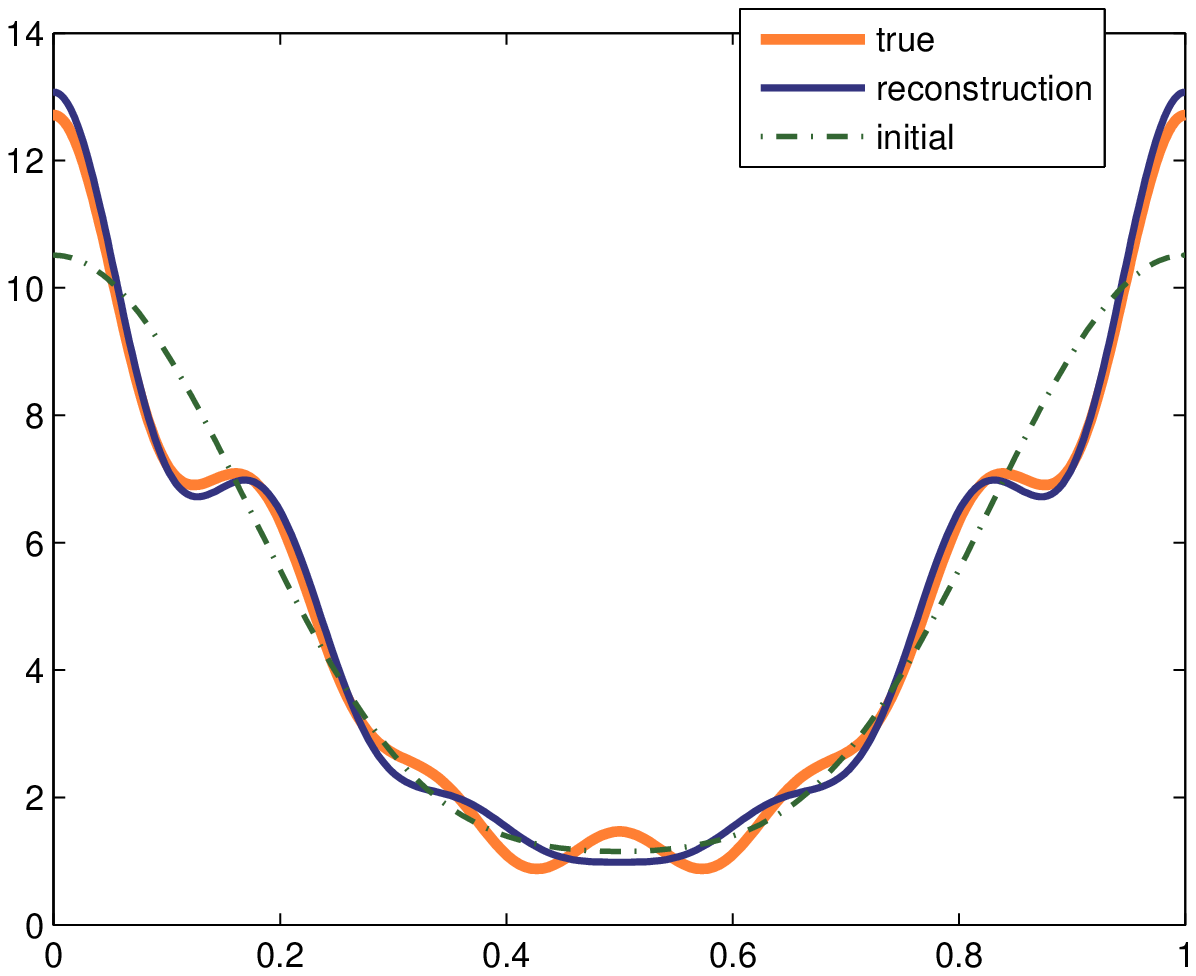}
      \caption{ K=7,N=25}
    \end{subfigure}
     \begin{subfigure}{0.32\textwidth}
      \includegraphics[width=\textwidth]{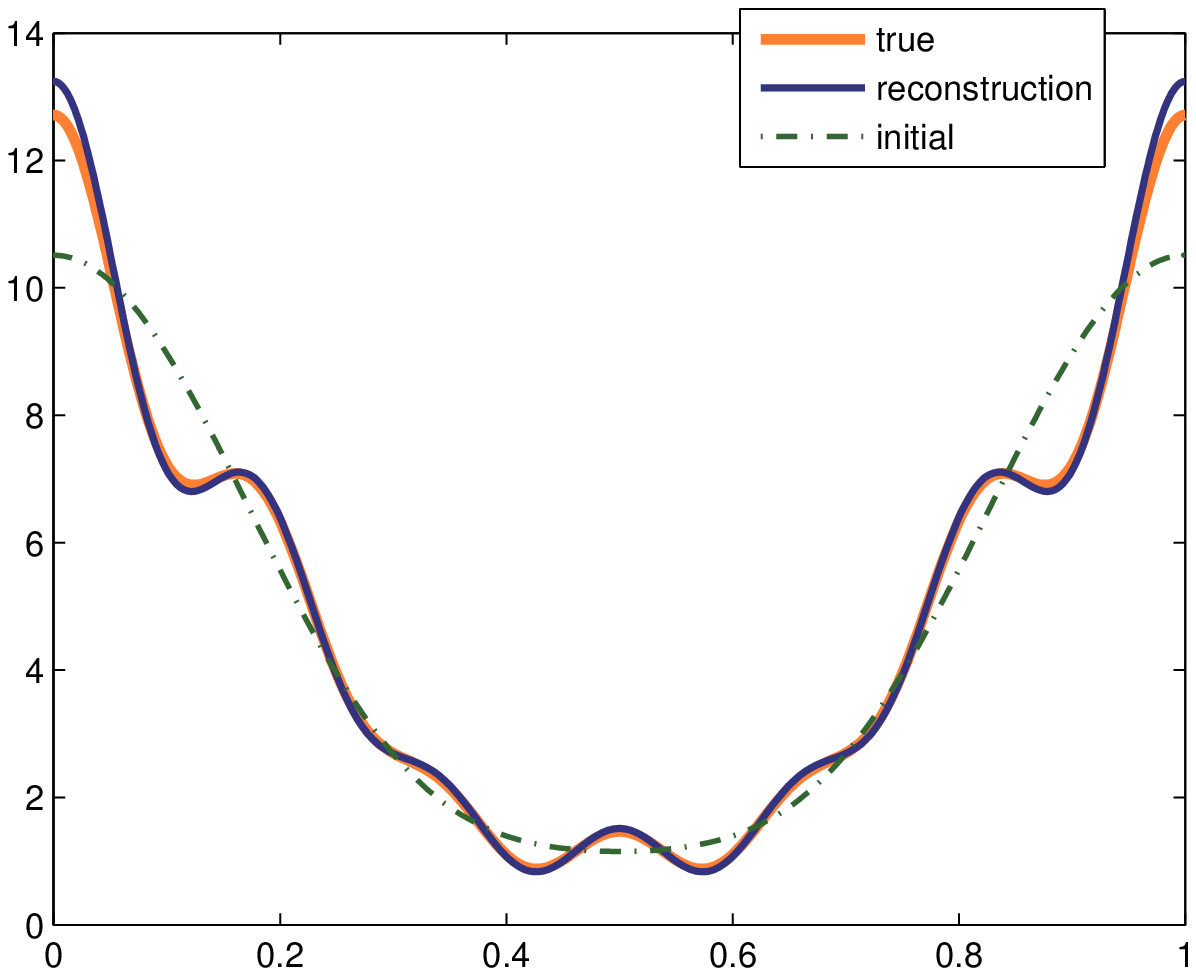}
      \caption{ K=7, N=30}
    \end{subfigure}
    \begin{subfigure}{0.32\textwidth}
      \includegraphics[width=\textwidth]{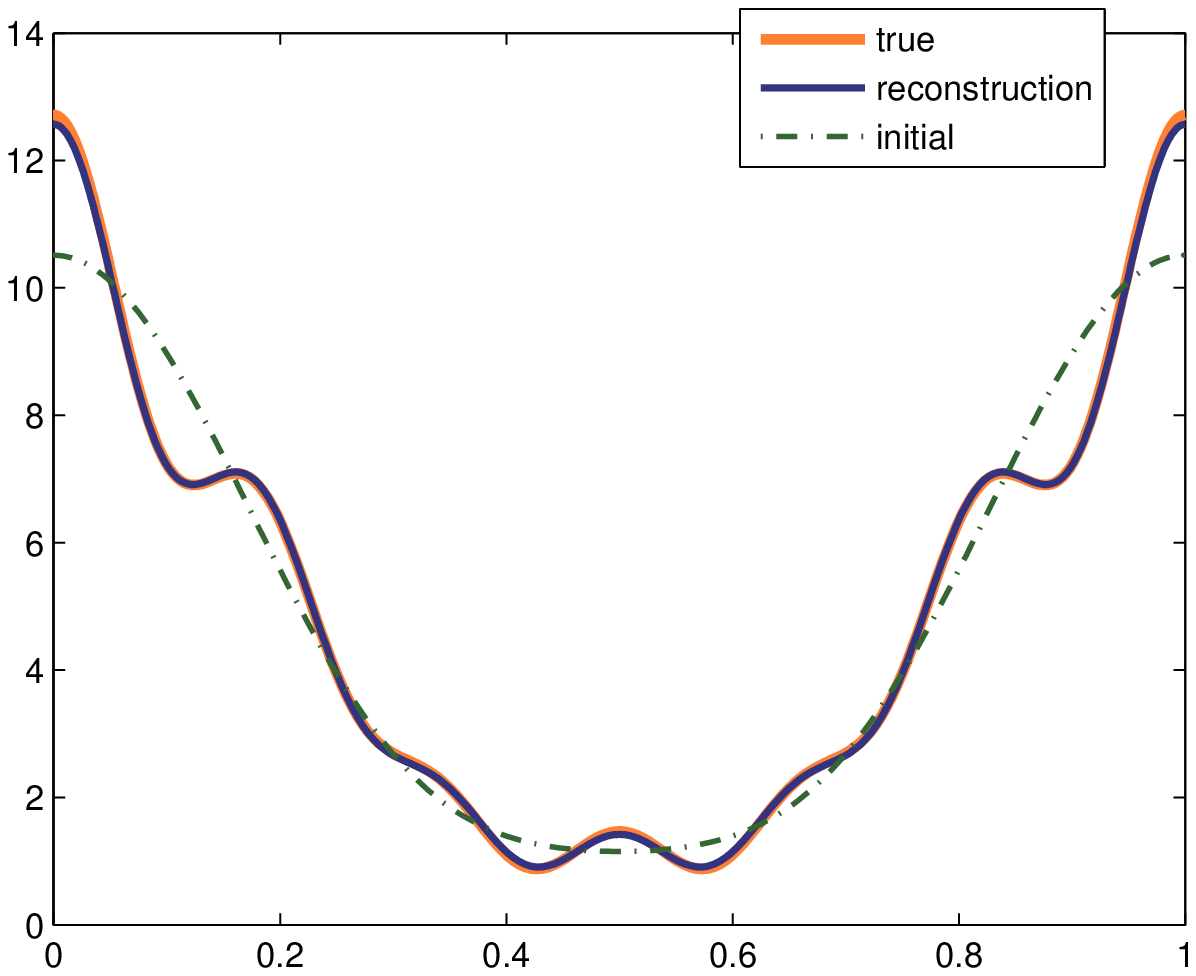}
      \caption{ K=7, N=35}
    \end{subfigure}\\
        \begin{subfigure}{0.33\textwidth}
      \includegraphics[width=\textwidth]{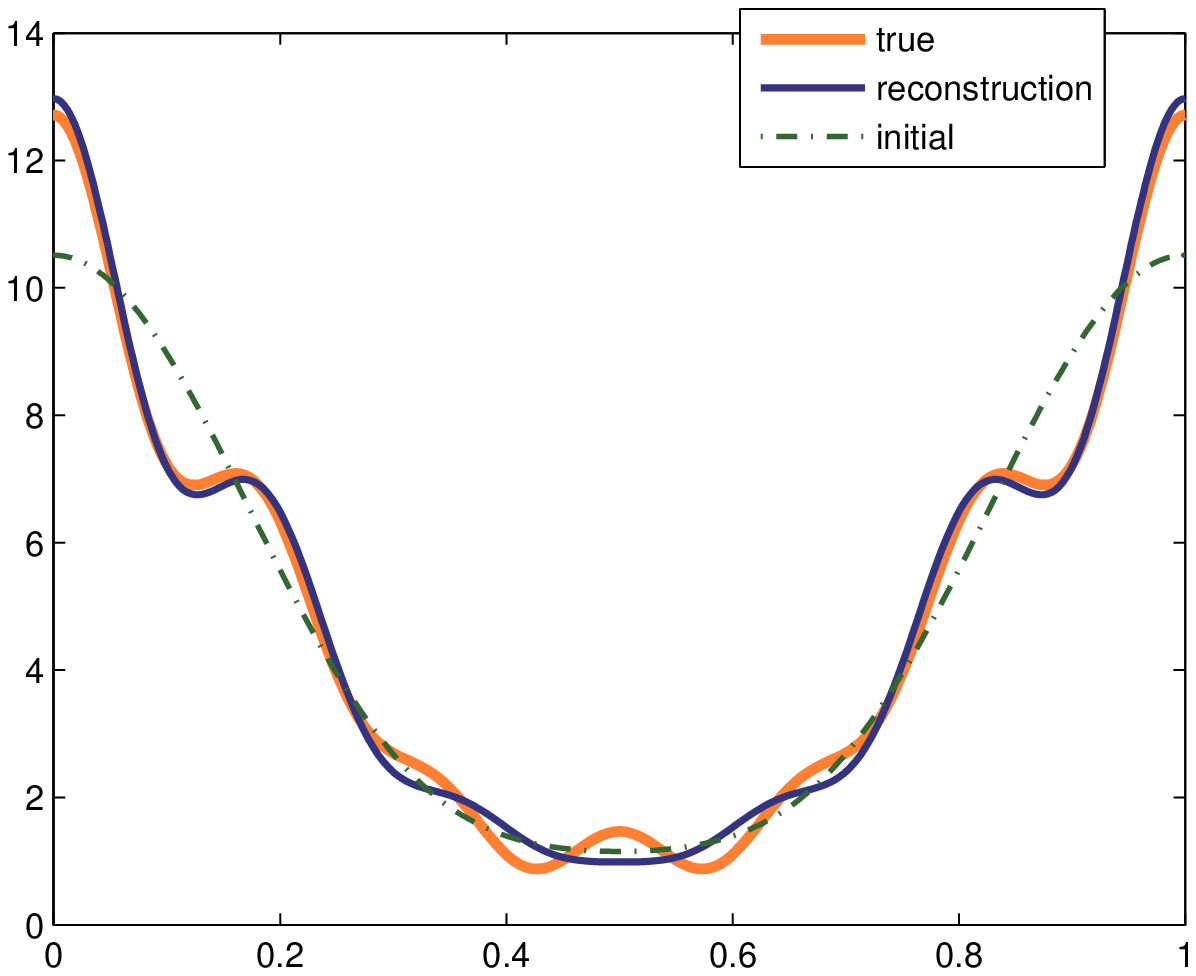}
      \caption{ K=8,N=25}
    \end{subfigure}
     \begin{subfigure}{0.32\textwidth}
      \includegraphics[width=\textwidth]{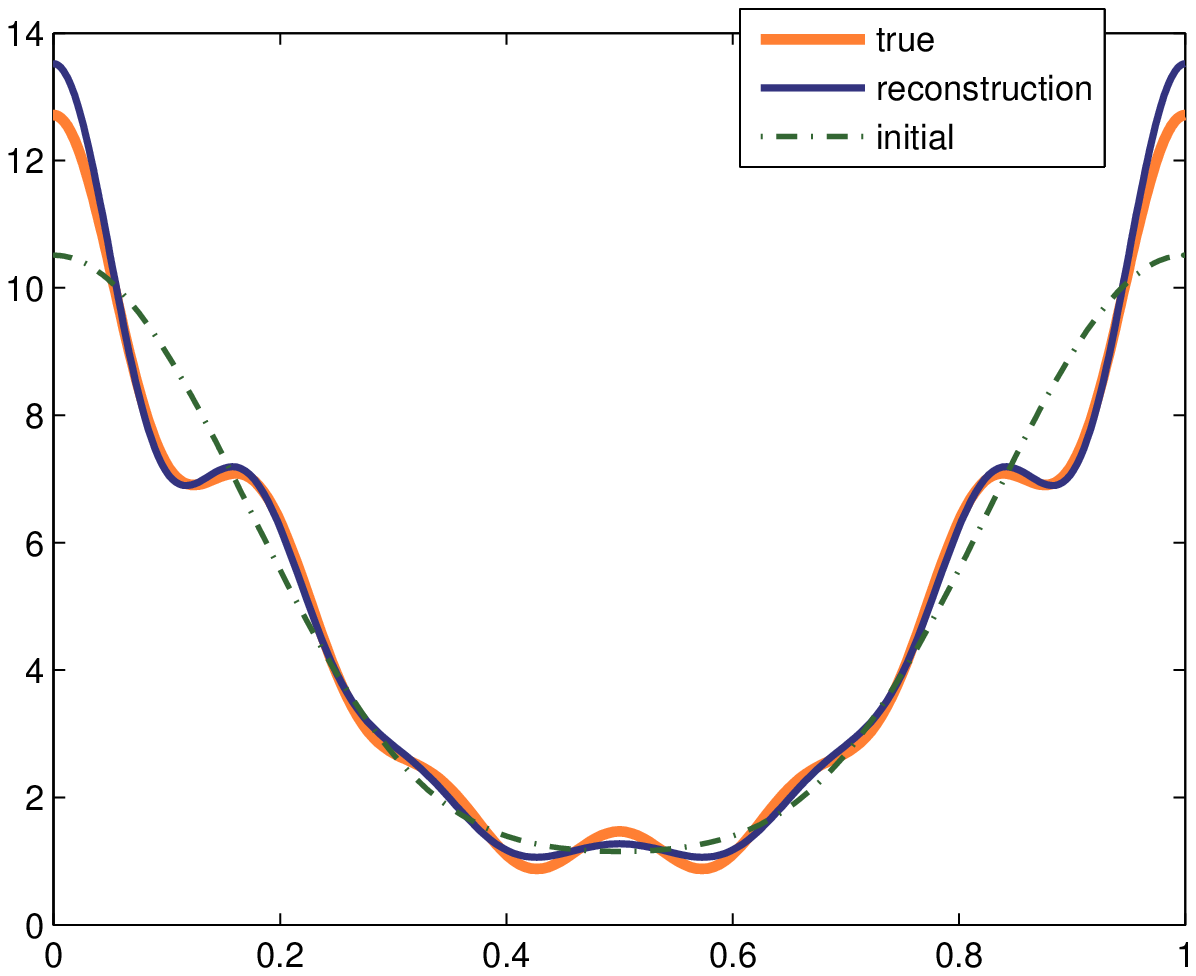}
      \caption{ K=8, N=30}
    \end{subfigure}
    \begin{subfigure}{0.32\textwidth}
      \includegraphics[width=\textwidth]{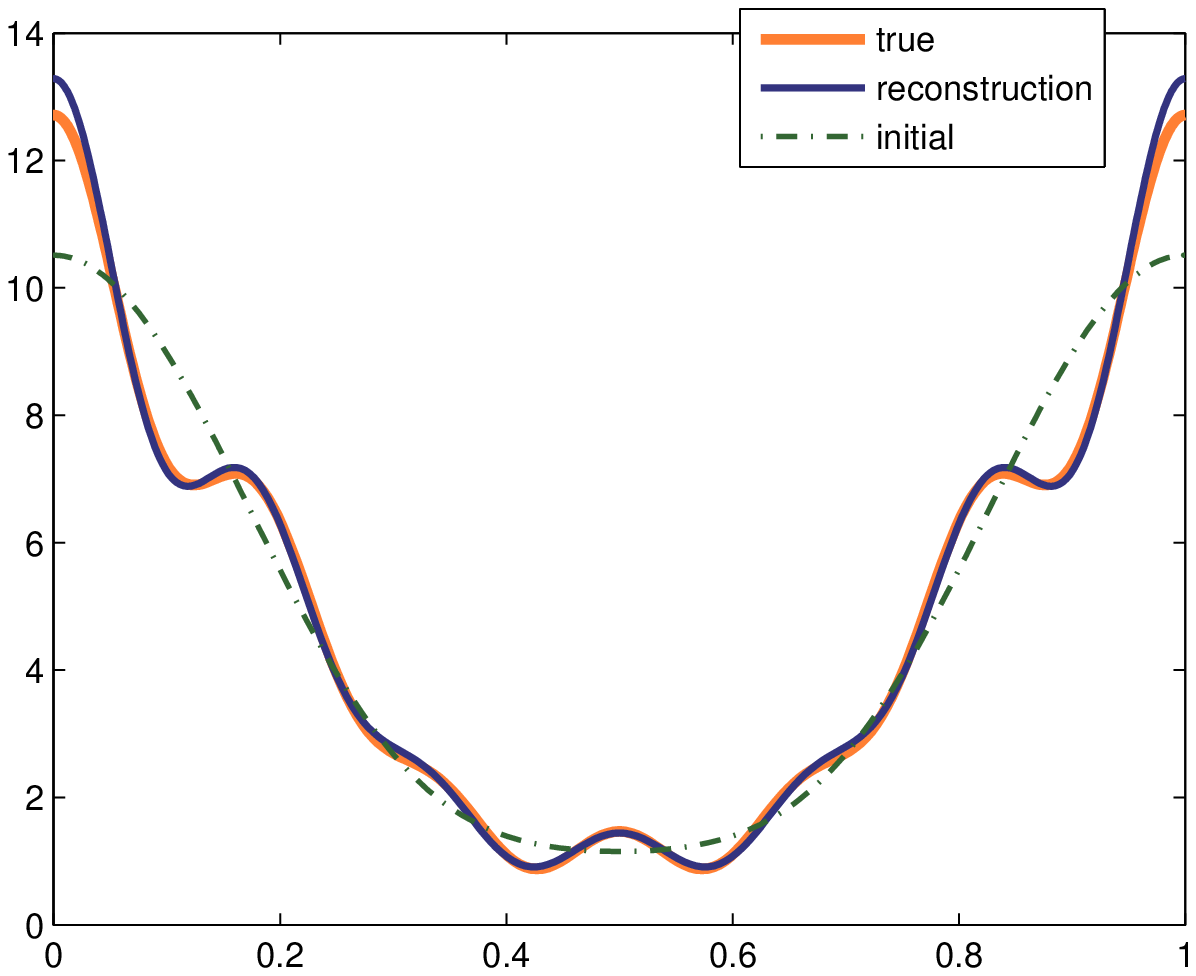}
      \caption{ K=8, N=35}
    \end{subfigure}
    \caption{Reconstruction of large damping $\alpha_M$ in Example \ref{example5} with $K=M$, $K_1=J=75$. }\label{Fig_ex5}
  \end{figure}

\end{example}

\section{Conclusion}
We have developed a novel inversion algorithm to recover the damping coefficient in a wave operator. A sequence of trace formulas are derived in a recursive form by investigating the resolvent properties of the damped wave operator, for which the inversion scheme is devised. Moreover, a class of polynomials needs to be chosen for the success of the inversion. Based on the distribution of eigenvalues and the properties of trace class operators, a sequence of proper polynomials is used. Numerical examples in Section \ref{sec5} illustrate the efficiency of the Algorithm.

\section*{Acknowledgements}
JZ thanks the many stimulating discussions with Steven Cox and Julio Moro.

\bibliographystyle{siam}
\bibliography{biblio}

\end{document}